\documentclass[10pt]{article}
\usepackage[utf8]{inputenc}

\usepackage[margin=1in]{geometry} 
\usepackage{graphicx}

\usepackage{booktabs} 
\usepackage{array} 
\usepackage{paralist} 
\usepackage{verbatim}
\usepackage{mathrsfs}
\usepackage{amssymb}
\usepackage{amsthm}
\usepackage{amsmath,amsfonts,amssymb}
\usepackage{esint}
\usepackage{graphics}
\usepackage{enumerate}
\usepackage{mathtools}
\usepackage{xfrac}
\usepackage{bbm}
\usepackage{caption}
\usepackage{subcaption}
\usepackage{tikz-cd}
\usepackage{adjustbox}

\usepackage[colorlinks=true, pdfstartview=FitV, linkcolor=blue, citecolor=blue, urlcolor=blue]{hyperref}

\numberwithin{equation}{section}
\numberwithin{figure}{section}

\newtheorem{theorem}{Theorem}[section]

\newtheorem{assumption}[theorem]{Assumption}

\newtheorem{proposition}[theorem]{Proposition}
\newtheorem{lemma}[theorem]{Lemma}
\theoremstyle{definition}
\newtheorem{definition}[theorem]{Definition}

\newtheorem{remark}[theorem]{Remark}

\newcommand*{\N}{\ensuremath{\mathbb{N}}}

\newcommand*{\Z}{\ensuremath{\mathbb{Z}}}

\newcommand*{\R}{\ensuremath{\mathbb{R}}}
\newcommand*{\Zd}{\ensuremath{\mathbb{Z}^d}}

\newcommand{\eps}{\varepsilon}

\renewcommand*{\tilde}{\widetilde}

\newcommand{\ep}{\eps}

\newcommand{\E}{\mathbb{E}}

\DeclareSymbolFont{boldoperators}{OT1}{cmr}{bx}{n}
\SetSymbolFont{boldoperators}{bold}{OT1}{cmr}{bx}{n}
\edef\bar{\unexpanded{\protect\mathaccentV{bar}}\number\symboldoperators16}
\renewcommand{\a}{\mathbf{a}}

\definecolor{labelkey}{rgb}{0,0,1}

\newcommand{\indc}{{\boldsymbol{1}}}

\newcommand{\var}{\mathrm{Var}}

\renewcommand{\P}{\mathbb{P}}

\usepackage{titlesec}

\newcommand{\addperiod}[1]{#1.}
\titleformat{\section}
   {\centering\normalfont\Large}{\thesection.}{0.5em}{}
\titleformat*{\subsection}{\bfseries}
\titleformat{\subsubsection}[runin]
  {\normalfont\bfseries}
  {\thesubsubsection.}
  {0.5em}
  {\addperiod}
\titleformat*{\subsubsection}{\bfseries}
\titleformat*{\paragraph}{\bfseries}
\titleformat*{\subparagraph}{\large\bfseries}

\title{Upper bounds on the fluctuations for a class of degenerate convex $\nabla \phi$-interface models}

\author{Paul Dario
\thanks{CNRS and LAMA, Universit\'e Paris-Est Cr\'eteil, Cr\'eteil, France.
{\footnotesize paul.dario@u-pec.fr.}
}
}
\date{ }

\usepackage[nottoc,notlot,notlof]{tocbibind}

\begin{document}

\maketitle

\begin{abstract}
We derive upper bounds on the fluctuations of a class of random surfaces of the $\nabla \phi$-type with convex interaction potentials. The Brascamp-Lieb concentration inequality provides an upper bound on these fluctuations for uniformly convex potentials. We extend these results to twice continuously differentiable convex potentials whose second derivative grows asymptotically like a polynomial and may vanish on an (arbitrarily large) interval. Specifically, we prove that, when the underlying graph is the $d$-dimensional torus of side length $L$, the variance of the height is smaller than $C \ln L$ in two dimensions and remains bounded in dimension $d \geq 3$.

The proof makes use of the Helffer-Sj\"{o}strand representation formula (originally introduced by Helffer and Sj\"{o}strand (1994) and used by Naddaf and Spencer (1997) and Giacomin, Olla Spohn (2001) to identify the scaling limit of the model), the anchored Nash inequality (and the corresponding on-diagonal heat kernel upper bound) established by Mourrat and Otto (2016) and Efron's monotonicity theorem for log-concave measures (Efron (1965)).
\end{abstract}

\setcounter{tocdepth}{1}
\tableofcontents

\section{Introduction}

The aim of this paper is to obtain fluctuations upper bounds for a class of random surfaces subject to $\nabla \phi$ type interaction arising in statistical physics. These models are used to model phase separation in $\R^{d+1}$, and are defined as follows. For any fixed dimension $d \geq 2$ and integer $L \geq 1$, we let $\mathbb{T}_L := (\Z/ (2L+1)\Z)^d$ be the $d$-dimensional torus of side length $2L+1$. We endow the edges of the torus with an orientation, let $E \left( \mathbb{T}_L \right)$ be the set of positively oriented edges of $\mathbb{T}_L$, and let $V$ be a potential, i.e., a measurable function $V : \R \to \R$ satisfying suitable properties. The random surface on $\mathbb{T}_L$ with potential $V$ is then the probability measure $\mu_{\mathbb{T}_L}$ on the set of functions $\Omega^\circ_{ \mathbb{T}_L} := \left\{ \phi : \mathbb{T}_L \to \R \, : \, \sum_{x \in \mathbb{T}_L} \phi(x) = 0 \right\}$ defined by
\begin{equation} \label{def.muTL}
    \mu_{\mathbb{T}_L}(d \phi) := \frac{1}{Z_{L}} \exp \left( - \sum_{e \in E \left( \mathbb{T}_L \right) } V\left( \nabla \phi(e) \right)\right) d \phi,
\end{equation}
where the discrete gradient is defined by $\nabla \phi(e) := \phi(y) - \phi(x)$ for the positively oriented edge $e = (x , y)$, $d\phi$ denotes the Lebesgue measure on the space $\Omega^\circ_{ \mathbb{T}_L}$ (equipped with the $L^2$ scalar product) and the normalization constant (or partition function)
\begin{equation*}
    Z_{L} := \int_{\Omega^\circ_{ \mathbb{T}_L}} \exp \left( - \sum_{e \in E \left( \mathbb{T}_L \right) } V\left( \nabla \phi(e) \right)\right) d \phi
\end{equation*}
is chosen so that $\mu_{\mathbb{T}_L}$ is a probability distribution. The model~\eqref{def.muTL} is known as the $\nabla \phi$ interface model or discrete Ginzburg-Landau model and has received considerable attention since its introduction in the seminal work of Brascamp, Lieb and Lebowitz~\cite{BLL75} (see~\cite{F05, V06} and Section~\ref{sectiondiscadback}). A natural property to investigate on the model is the question of its \emph{localization} or \emph{delocalization}, that is, to establish whether the variance $\var_{\mu_{\mathbb{T}_L}} \left[ \phi(0) \right]$ remains bounded or diverges to infinity as $L$ tends to infinity. Explicit computations available in the case $V(x) = x^2$, i.e., in the case of the discrete Gaussian free field, show that this variance diverges as the size $L$ of the torus tends to infinity in two dimensions (the random surface is said to be \emph{delocalized}, and the divergence is in fact logarithmic in $L$), and remains bounded uniformly in $L$ in higher dimensions (the random surface is then said to be \emph{localized}). Brascamp, Lieb and Lebowitz~\cite{BLL75} conjectured that this result should remain valid for any potential $V$ satisfying $\int_\R \exp \left( - p V(x) \right) \, dx < \infty$ for all $p > 0$ and obtained a sharp (up to multiplicative constant) upper bound on the fluctuations of the random surface, using the celebrated Brascamp-Lieb concentration inequality, for twice-continuously differentiable potentials satisfying $\inf V'' > 0$ and for a class of convex potential with quadratic growth. Since the results of~\cite{BLL75}, the localization and delocalization upper bounds have been extended to various settings including:
\begin{itemize}
    \item Non-convex potentials arising as a perturbation of uniformly convex potentials by Cotar, Deuschel and M\"{u}ller~\cite{CDM09, CD12};
    \item Non-convex potentials which are a perturbation of uniformly convex potentials and are amenable to renormalization group analysis by Adams, Buchholz, Koteck\'{y}, M\"{u}ller~\cite{AKM16, adams2019cauchy}, Hilger~\cite{hilger2016scaling, hilger2020scaling, hilger2020decay}, Adams, Koller~\cite{adams2023hessian} and Bauerschmidt, Park and Rodriguez~\cite{BPR1, BPR2} (for the discrete Gaussian model in the latter case);
    \item Potentials which can be written as a mixture of Gaussians by Biskup, Koteck\'{y}~\cite{BK07}, Biskup, Spohn~\cite{BS11}, Brydges, Spencer~\cite{brydges2012fluctuation} and Ye~\cite{ye2019models};
    \item Convex potentials satisfying that the set $\left\{ x \in \R \, : \, V''(x) = 0 \right\}$ has Lebesgue measure $0$ by Peled and Magazinov~\cite{magazinov2020concentration};
    \item The potential $V(x) = |x|$ using the infra-red bound of Bricmont, Fontaine and Lebowitz~\cite{bricmont1982surface} (this case is also covered in~\cite{brydges2012fluctuation}).
\end{itemize}
Lower bounds on the fluctuations of the random surface have been established in a much more general setting, and Mermin-Wagner type arguments have been used successfully to prove logarithmic lower bounds for the variance of the height in two dimensions for a large class of potentials including all the twice-continuously differentiable $V$~\cite{BLL75, dobrushin1980nonexistence, frohlich1981absence, ioffe20022d}, as well as for models with hard-core constraint~\cite{MilosPeled2015}. Section~\ref{sectiondiscadback} discusses additional results beyond the questions of localization and delocalization (such as hydrodynamic limit, scaling limit, strict convexity of the surface tension, decay of covariances, large deviations) which have been proved for this model.

In this article, we are interested the class of convex potentials whose second derivative grows like a polynomial, formally defined in Assumption~\ref{assumption1} below.

\begin{assumption} \label{assumption1}
We assume that $V : \R \to \R$ is a potential satisfying the assumptions:
\begin{itemize}
    \item[(i)] \textit{Regularity and convexity:} we assume that $V$ is twice-continuously differentiable and convex;
    \item[(ii)] \textit{Growth of the second derivative:} we assume that the second derivative of $V$ satisfies a power-law growth condition: there exist an exponent $r > 2$ and two constants $c_+ , c_- \in (0 , \infty)$ such that 
    \begin{equation*}
        0 < c_- \leq \liminf_{|x| \to \infty} \frac{V''(x)}{|x|^{r-2}} \leq \limsup_{|x| \to \infty} \frac{V''(x)}{|x|^{r-2}} \leq c_+ < \infty.
    \end{equation*}
\end{itemize}
\end{assumption}
The main theorem of this paper establishes that the variance of the random surface grows at most logarithmically fast in two dimensions and remains bounded in dimensions $3$ and higher for the class of potentials satisfying Assumption~\ref{assumption1}.

\begin{theorem}[Localization and Delocalization] \label{main.theorem}
Under Assumption~\ref{assumption1}, there exists a constant $C := C(d , V) < \infty$ such that, for any $L \geq 2$,
\begin{equation*}
    \var_{\mathbb{T}_L} \left[\phi(0) \right] \leq 
    \left\{ \begin{aligned}
        C \ln L & ~~ \mbox{if} ~ d=2, \\
        C & ~~\mbox{if} ~ d \geq 3.
    \end{aligned} \right.
\end{equation*}
\end{theorem}

\begin{remark}
    The convexity of the potential $V$ implies that the measure~\eqref{def.muTL} is log-concave. Since log-concavity is a property which is closed under marginalization (by the Pr\'{e}kopa-Leindler inequality~\cite{prekopa1971logarithmic, prekopa1973logarithmic, leindler1972certain}), this implies that the distribution of the height $\phi(0)$ is also log-concave. Since the tail of a log-concave distribution decays at least exponentially fast on the scale of its standard deviation, the result of Theorem~\ref{main.theorem} can be extended from a bound on the variance to a bound on exponential moments.
\end{remark}

\begin{remark}
        It is plausible that the techniques developed in this article can be further extended to obtain more precise properties on the behavior of the model (such as its hydrodynamic and scaling limits). These questions are further discussed in Section~\ref{sectionfurthercomm} below.
\end{remark}

\subsection{Outline of the proof}

In order to highlight the main ideas and techniques used to prove Theorem~\ref{main.theorem}, we present below a sketch of the argument for potentials satisfying the following assumptions: we assume that $V : \R \to \R$ is twice-continuously differentiable, convex and that there exists $c_1 \in (0 , 1)$ such that
\begin{equation} \label{eq:1329}
     0 < c_1 \leq \liminf_{|x| \to \infty} V''(x) \hspace{5mm}\mbox{and} \hspace{5mm} \sup_{x \in \R} V''(x) \leq 1.
\end{equation}
Note that this is more restrictive than Assumption~\ref{assumption1}; the full argument will require some notational and technical adjustments.

\subsubsection{The Helffer-Sj\"{o}strand representation formula}

One of the main tools used to prove fluctuation upper bounds is the Helffer-Sj\"{o}strand representation formula, initially introduced by Helffer and Sj\"{o}strand~\cite{helffer1994correlation} and used by Naddaf and Spencer~\cite{NS} and Giacomin, Olla and Spohn~\cite{GOS} in order to identify the scaling limit of the model, and by Deuschel, Giacomin and Ioffe~\cite{DGI00} to establish a large deviation principle for the model (among other results, see Section~\ref{sectiondiscadback}). In the setting of this paper, the formula reads as follows. Let $\phi_L$ be the stationary Langevin dynamic associated with the Gibbs measure $\mu_{\mathbb{T}_L}$, i.e., the solution of the system of stochastic differential equations
\begin{equation} \label{Langevinsketchofproof}
\left\{ \begin{aligned}
d \phi_L(t , x) &= \nabla \cdot V' \left( \nabla \phi_L \right)(t , x) + \sqrt{2} dB_t(x) &~~\mbox{for} &~~(t , x) \in (0 , \infty) \times \mathbb{T}_L, \\
\phi_L(0 , x) &= \phi(x) & \mbox{for} &~~x \in \mathbb{T}_L,
\end{aligned}
\right.
\end{equation}
where $\left\{ B_t(x) \, : \, t \geq 0 , x \in \mathbb{T}_L \right\}$ is collection of independent Brownian motions, and the initial condition $\phi$ is sampled according to $\mu_{\mathbb{T}_L}$ independently of the Brownian motions. Then, one has the identity
\begin{equation} \label{eq:20150901}
    \var_{\mu_{\mathbb{T}_L}} \left[\phi(0) \right] = \E \left[ \int_0^\infty P_\a (t , 0) \, dt \right],
\end{equation}
where $P_\a$ is the heat kernel associated with the discrete parabolic equation (using the notation of Section~\ref{section2})
\begin{equation*}
    \left\{ \begin{aligned}
     \partial_t P_\a(t , x) - \nabla \cdot \a \nabla P_\a (t , x) &= 0 &~~\mbox{for} &~~ (t , x) \in (0 , \infty) \times \mathbb{T}_L, \\
    P_\a (0 , x) & = \delta_0(x) - \frac{1}{\left| \mathbb{T}_L \right|} & ~~\mbox{for} &~~ x \in \mathbb{T}_L,
    \end{aligned} \right.
\end{equation*}
with the environment $\a(t , e) := V''(\nabla \phi_L(t , e))$.

As has been observed in~\cite{helffer1994correlation, NS, DGI00, GOS}, the Helffer-Sj\"{o}strand representation formula can be combined with tools of elliptic regularity, in the form of on-diagonal heat kernel estimates, to prove upper bounds on the fluctuations on the random surface. For instance, if the potential $V$ is assumed to be uniformly convex, i.e., if $0 < c_- \leq V'' \leq 1$, then one has the bound $c_- \leq \a(t , e) \leq 1$. In this setting the parabolic equation arising from the Helffer-Sj\"{o}strand representation formula is uniformly elliptic, and this property is sufficient to prove the following on-diagonal upper bound on the heat kernel
\begin{equation} \label{eq:Panashsketchproof}
    P_\a (t , 0) \leq \frac{C}{(1 + t)^{\frac d2}} \exp \left( - \frac{t}{CL^2} \right).
\end{equation}
Integrating the bound~\eqref{eq:Panashsketchproof} over the times $t \in [0 , \infty)$ and using the identity~\eqref{eq:20150901} yields the variance estimate stated in Theorem~\ref{main.theorem}.

The proof of Theorem~\ref{main.theorem} follows the strategy described in the previous paragraph, but some additional arguments are required to take into account that the second derivative of a potential $V$ satisfying~\eqref{eq:1329} can vanish.

\subsubsection{The on-diagonal heat kernel upper bound in degenerate environment of Mourrat and Otto}

Under Assumption~\eqref{eq:1329}, the upper bound on the fluctuations of the random surface can be obtained by first extending the on-diagonal upper bound for the heat kernel~\eqref{eq:Panashsketchproof} to \emph{degenerate environments}, i.e., environments $\a : (0 , \infty) \times E \left( \mathbb{T}_L \right) \to [0 , 1]$ which may vanish (or take values arbitrarily close to $0$). This question has received significant attention from the mathematical community (see Section~\ref{sectiondiscadback}), and, in this article, we rely on the approach of Mourrat and Otto~\cite{MO16} and of Biskup and Rodriguez~\cite{biskup2018limit} who respectively proved an on-diagonal upper bound for the heat kernel  and a quenched invariance principle for a large class of \emph{dynamic degenerate environments}. The exact result of the former (stated in infinite volume) can be found in~\cite[Theorem 4.2]{MO16}. Their proof could be adapted to the setting considered here, and would show the following result. Given an environment $\a : (0 , \infty) \times E\left( \mathbb{T}_L \right) \to [0 , 1]$, if we define the \emph{moderated environment} by
\begin{equation} \label{eq:modeenvtwdefsketxh}
    w(t , e) := \int_t^\infty \frac{\a (s , e)}{(1 + s-t)^4} \, ds,
\end{equation}
then there exists a function $t \mapsto \mathscr{M}_t \in [1 , \infty]$ depending only on the dimension $d$ and the moderated environment $w$ such that, for any $t \geq 0$,
\begin{equation} \label{eq:08231101}
    P_\a (t , 0) \leq \frac{\mathscr{M}_t}{(1 + t)^{\frac d2}} \exp \left( - \frac{t}{\mathscr{M}_t L^2 } \right).
\end{equation}
The dependency of the function $\mathscr{M}$ on the parameter $w$ is explicit and it satisfies the following property: if we assume that the environment $\a$ is random, that its law is stationary with respect to both space and time translations and reversible, and if, for any $k \in \N$ and any $(t , e) \in (0 , \infty) \times E(\mathbb{T}_L)$,
\begin{equation} \label{momentsassumptionmoderatedenvironement}
    \E \left[ w(t , e)^{-k} \right] < \infty,
\end{equation}
then $(\mathscr{M}_t)_{t \geq 0}$ is a stationary process and, for any $k \in \mathbb{N},$
\begin{equation} \label{momentsassumptionmoderatedenvironement2}
    \E \left[ \mathscr{M}_t^k \right] < \infty.
\end{equation} 
The result of Mourrat and Otto can thus be applied to establish upper bounds on the fluctuations of random surfaces as follows: by the Helffer-Sj\"{o}strand representation formula (noting that the law of the environment $\a(t , e) := V''(\nabla \phi_L(t , e))$ is stationary with respect to both the space and time variables), we see that, if the moment assumption~\eqref{momentsassumptionmoderatedenvironement} can be verified, then the inequality~\eqref{momentsassumptionmoderatedenvironement2} implies that
\begin{align*}
    \var_{\mathbb{T}_L} \left[\phi(0) \right]  = \E \left[ \int_0^\infty P_\a (t , 0) \, dt \right] & \leq \E \left[ \int_0^\infty \frac{\mathscr{M}_t}{(1 + t)^{\frac d2}} \exp \left( - \frac{t}{\mathscr{M}_t L^2 } \right) \, dt \right] \\
    & \leq 
    \left\{ \begin{aligned}
        C \ln L & ~~ \mbox{if} ~ d=2, \\
        C & ~~\mbox{if} ~ d \geq 3.
    \end{aligned} \right.
\end{align*}
In other words, the question of establishing upper bounds on the fluctuations of the random surface can be reduced to proving the moment condition~\eqref{momentsassumptionmoderatedenvironement} on the moderated environment $w$. The strategy will thus be to prove~\eqref{momentsassumptionmoderatedenvironement}, and the argument is outlined in the following sections.

\subsubsection{A fluctuation estimate for the Langevin dynamic and stochastic integrability of the moderated environment} \label{eq:18111101}
In order to prove the moment condition~\eqref{momentsassumptionmoderatedenvironement}, we will prove the following fluctuation estimate for the Langevin dynamic: for any $R > 0$, there exists a constant $C_R$ depending only on $d$ and $R$ such that, for any time $T \geq 0$ and any edge $e \in E \left( \mathbb{T}_L \right)$,
\begin{equation} \label{eq:22440901}
    \P \left[ \forall t \in [0 , T], \, \left| \nabla \phi(t , e) \right| \leq R   \right] \leq C_R \exp \left( - \frac{T}{C_R} \right).
\end{equation}
Combining this result with assumption~\eqref{eq:1329} and the definition $\a (t , e) := V''\left( \nabla \phi(t , e) \right)$ shows that there exists a constant $C_V$ depending only on $d$ and $V$ such that, for any $T \geq 0$ and any edge $e \in E \left( \mathbb{T}_L \right)$,
\begin{equation} \label{eq:23030901}
    \P \left[ \forall t \in [0 , T], \, \a(t , e) = 0  \right] \leq C_V \exp \left( - \frac{T}{C_V} \right).
\end{equation}
The estimate~\eqref{eq:23030901} implies that the environment arising from the Helffer-Sj\"{o}strand representation cannot remain equal to $0$ for a long time; it can in fact be generalized (the argument is the one of Proposition~\ref{propmoderatedenvtstochint} below) so as to obtain the following stretched exponential stochastic integrability on the moderated environment: there exist an exponent $s > 0$ and a constant $C_V$ such that, for any $R \geq 0$,
\begin{equation*}
    \P \left[ w(t , e) \leq \frac{1}{R} \right] \leq C_V \exp \left( - \frac{R^s}{C_V}\right),
\end{equation*}
which then implies the moment condition~\eqref{momentsassumptionmoderatedenvironement}. 

In the rest of this section, we give an outline of the proof of~\eqref{eq:22440901} for potentials satisfying~\eqref{eq:1329}. The argument relies on three observations:
\begin{enumerate}
    \item[(i)] The Langevin dynamic $\phi_L$ defined in~\eqref{Langevinsketchofproof} can be seen as a \emph{deterministic function} of the initial condition $\phi$ and the Brownian motions $\left\{ B_t(x) \, : \, t \geq 0 , \, x \in \mathbb{T}_L \right\}$.
    \item[(ii)]  For any $x \in \mathbb{T}_L$, the Brownian motion $B_t(x)$ can be decomposed into a sum of independent increments and Brownian bridges as follows: if, for any $n \in \N$ and any $t \in [n , n+1]$, we define
    \begin{equation} \label{def.XnandWnsketch}
            X_n(x) := B_{n+1}(x) -  B_{n}(x) ~~\mbox{and}~~ W_n(t , x) := B_t(x) - B_{n}(x) - (t-n) X_n(x),
    \end{equation}
    then the random variables $\left\{ X_n(x) \, : \, n \in \N \right\}$ form a collection of independent Gaussian random variables (of variance $1$), and the stochastic processes $\left\{ W_n(\cdot , x) \, : \, n \in \N \right\}$ form a collection of independent Brownian bridges. Additionally, the increments are independent of the Brownian bridges.
    \item[(iii)]  Since the trajectory of the Brownian motion $B_t(x)$ can be reconstructed from the values of the increments $\left\{ X_n(x) \, : \, n \in \N \right\}$ and the Brownian bridges $\left\{ W_n(\cdot , x) \, : \, n \in \N \right\}$, we can see the Langevin dynamic $\phi_L$ as a deterministic function of the initial condition, the increments and the Brownian bridges. Using the definition~\eqref{def.XnandWnsketch}, we see that the Langevin dynamic solves the system of stochastic differential equations, for any $n \in \N$,
    \begin{equation} \label{Langevinreadyfordifferentiation}
\left\{ \begin{aligned}
d \phi_L(t , x) &= \nabla \cdot V' \left( \nabla \phi_L \right)(t , x) + \sqrt{2} X_n(x) + \sqrt{2} dW_n(t , x) &~~\mbox{for}~~(t , x) \in (n , n+1) \times \mathbb{T}_L, \\
\phi_L(0 , x) &= \phi(x) & \mbox{for} ~~x \in \mathbb{T}_L.
\end{aligned}
\right.
\end{equation}
\end{enumerate}
The strategy is then to study the partial derivative of the Langevin dynamic with respect to the increment~$X_n(x)$. To this end, we differentiate both sides of~\eqref{Langevinreadyfordifferentiation} with respect to the increment $X_n(x)$, and obtain that the partial derivative $w := \partial \phi_L/\partial X_n(x)$ solves the parabolic equation
\begin{equation*}
\left\{ \begin{aligned}
\partial_t w (t , y) &= \nabla \cdot \a \nabla w (t , y) + \sqrt{2} \indc_{\left\{ n \leq t \leq n+1 \right\}} \indc_{\{ y = x \}}&~~\mbox{for} &~~(t , y) \in (n , n+1) \times \mathbb{T}_L, \\
w(0 , x) &= 0 & \mbox{for} &~~x \in \mathbb{T}_L,
\end{aligned}
\right.
\end{equation*}
where $\a(t ,e) := V''(\nabla \phi_L(t , e))$ is the same environment as the one appearing in the Helffer-Sj\"{o}strand representation formula. The Duhamel's principle then yields the identity (using the definition of the heat kernel~\eqref{eq:11012312})
\begin{equation*}
    w(n+1 , x) = \sqrt{2} \int_{n}^{n+1} P_\a \left( n+1 , x ; s , x \right)  + \frac{1}{\left| \mathbb{T}_L \right|} \, ds.
\end{equation*}
The right-hand side of the previous display can be lower bounded as follows. We first note that $P_\a \left( t , x ; s , x \right)  + \frac{1}{\left| \mathbb{T}_L \right|} \in [0,1]$ by the maximum principle.  Combining these bounds with the upper bound $\a \leq 1$, the identity $\partial_t P_\a = \nabla \cdot \a \nabla P_\a$ and the definition of the discrete elliptic operator, we deduce that, for any $t , s \in (0 , \infty)$ with $t \geq s$,
\begin{equation} \label{eq:17571101}
    P_\a (s , x ; s , x) + \frac{1}{\left| \mathbb{T}_L\right|}= 1 ~~\mbox{and}~~  \left| \partial_t P_\a (t , x ; s , x) \right| \leq 2d, 
\end{equation}
which then implies
\begin{equation*}
    w(n+1 , x) \geq \sqrt{2} \int_{0}^{1} \max ( 1 - 2d s , 0 ) \, ds = \frac{\sqrt{2}}{4d} > 0.
\end{equation*}
In words, the partial derivative of the value $\phi_L(n+1 , x )$ with respect to the increment $X_n(x)$ is lower bounded by $\sqrt{2}/(4d)$ uniformly over all the realizations of the Brownian motions. This implies that $\phi_L(n+1,x)$ is an increasing function of the increment $X_n(x)$, and more specifically, that increasing the value of the increment $X_n(x)$ by a value $X\geq 0$ (while keeping the other increments and the Brownian bridges unchanged), causes the value of $\phi_L(n+1,x)$ to increase by at least $X/(4d)$.

The previous argument can be refined so as to prove that, for any edge $e =  (x_0 , x) \in E \left( \mathbb{T}_L \right)$, the derivative of the discrete gradient $\nabla \phi_{L}(n+1 , e)$ with respect to the increment $X_n(x)$ is lower bounded by a positive real number uniformly over the realizations of the increments and the Brownian bridges.

This property can then be used to prove the following result: for any $R > 0$, there exists $\ep := \ep(R) > 0$ such that, if we denote by $\mathcal{F}_{n , x}$ the $\sigma$-algebra generated by the initial condition $\phi$, the Brownian bridges $\left\{ W_m(\cdot , y) \, : \, m \in \N , \, y \in \mathbb{T}_L \right\}$ and the increments $\left\{ X_m(y) \, : \, m \neq n \, \mbox{or} \, x \neq y \right\}$, then one has the almost sure upper bound on the conditional probability
\begin{equation} \label{eq:20331001}
    \P \left[ \left| \nabla \phi_L(n+1 , e) \right| \leq R ~|~ \mathcal{F}_{n , x} \right] \leq 1- \ep. 
\end{equation}
In other words, the probability of the event $\{ \left| \nabla \phi_L(n+1 , e) \right| \leq R \}$ conditionally on all the randomness except the increment $X_n(x)$ is almost surely smaller than $1-\ep$.

The inequality~\eqref{eq:20331001} can then be iterated (making use of the independence between the increments and the Brownian bridges) to prove that, for any $N \in \N$,\begin{equation*}
    \P \left[ \forall n \in \{ 1 , \ldots , N \}, ~ \left| \nabla \phi_L (n , e) \right| \leq R \right] \leq ( 1 - \ep)^N,
\end{equation*}
which implies the exponential decay stated in~\eqref{eq:22440901}.

\subsubsection{Extension of the argument to the potentials satisfying Assumption~\ref{assumption1}}

In the case of potentials satisfying Assumption~\ref{assumption1}, the second derivative of the potential $V$ is unbounded from above, and thus the environment $\a$ appearing in the Helffer-Sj\"{o}strand representation formula can take arbitrarily large values. This implies that the argument written above needs to be modified in two aspects:
\begin{itemize}
    \item The proof of Mourrat and Otto~\cite[Theorem 4.2]{MO16} is written in infinite volume for degenerate dynamic environments satisfying the upper bound $\a \leq 1$. Their argument needs to be adapted the torus, and to cover a class of environments which may take arbitrarily large values. This is the subject of Section~\ref{section4ondiagupp}.
    \item In the situation where the environment $\a$ can take arbitrarily large values, the inequality on the time derivative of the heat kernel~\eqref{eq:17571101} does not hold uniformly over all the realizations of the Brownian motions, and thus the derivative of $\phi_L(n+1 , x )$ with respect to the increment $X_n(x)$ cannot be lower bounded (by a strictly positive real number) uniformly over all the realizations of the Brownian motions. This difficulty is handled by first establishing a sharp stochastic integrability estimate on the discrete gradient of a random surface distributed according to $\mu_{\mathbb{T}_L}$ (see Proposition~\ref{prop.stochintgradfield}). Once equipped with this result, we are able to adapt the argument outlined in Section~\ref{eq:18111101} to this setting (see Proposition~\ref{prop3.4}), at the cost of more technicalities and a deterioration of the stochastic integrability in the fluctuation estimate (from exponential rate to the super-polynomial rate in Proposition~\ref{prop3.4}).
\end{itemize}

\subsection{Discussion and background} \label{sectiondiscadback}

\subsubsection{Random surfaces}

The study of random surfaces was initiated in the 1970s by Brascamp, Lieb and Lebowitz~\cite{BLL75} who obtained sharp localization and delocalization estimates for potentials satisfying $\inf V'' > 0$ and for a class of convex potentials with quadratic growth. Since then, the result of localization and delocalization has been extended to different classes of potentials as mentioned above, and various other aspects of the model have been studied by the mathematical community (see~\cite{F05, V06}).

\emph{The hydrodynamic limit} of the $\nabla \phi$-model for uniformly convex potentials was established by Funaki and Spohn in the important contribution~\cite{FS}. The result was later extended to various settings: the hydrodynamic limit with Dirichlet boundary condition and with a conservation law was proved by Nishikawa~\cite{nishikawa2003hydrodynamic, nishikawa2002hydrodynamic}. More recently, the hydrodynamic limit was established for a class on non-convex potentials by Deuschel, Nishikawa and Vignaud~\cite{DNV19}.

On the level of fluctuations, it is expected that \emph{the scaling limit} of the $\nabla \phi$-model is a \emph{continuum Gaussian free field} under mild integrability conditions on the potential $V$. On a rigorous level, a general convergence result has been established for twice continuously differentiable and uniformly convex potentials by Brydges and Yau~\cite{BY}, Naddaf, Spencer~\cite{NS} and Giacomin, Olla and Spohn~\cite{GOS}. In particular, the contributions~\cite{NS, GOS} used the Helffer-Sj\"{o}strand representation formula (introduced in~\cite{helffer1994correlation}), which has become well-used technique to study the model, and is a central tool in the proof of Theorem~\ref{main.theorem}. The scaling limit has then been established in various different settings. A finite-volume version of the result was established by Miller~\cite{Mil}, a local limit theorem was established in two dimensions by Wu~\cite{wu2022local} and the scaling limit of the square of the field was identified by Deuschel and Rodriguez~\cite{deuschel2022isomorphism}. The scaling limit was proved for a class of convex potentials satisfying the assumption $\inf V'' > 0$ by Andres and Taylor~\cite{AT21}. In the nonconvex setting, it was established in the high temperature regime by Cotar and Deuschel~\cite{CD12}, in the low temperature regime by renormalization group arguments by Hilger~\cite{hilger2016scaling, hilger2020scaling} (buiding upon the techniques of~\cite{AKM16}), and in the case of non-convex potentials which can be written as a mixture of Gaussian by Biskup, Spohn~\cite{BS11} and Ye~\cite{ye2019models}.

Besides the hydrodynamic and scaling limits, other aspects of the model which have been the subject of consideration from the community include: the strict convexity of the surface tension for non-convex potentials by Adams, Koteck{\'y} and M\"{u}ller~\cite{AKM16} (the Cauchy-Born rule was also investigated in~\cite{adams2019cauchy}), Cotar, Deuschel and M\"{u}ller~\cite{CDM09}, its $C^2$-regularity by Armstrong, Wu~\cite{AW}, the decay of covariances for the gradient of the field by Delmotte, Deuschel~\cite{DD05}, Cotar, Deuschel~\cite{CD12}, and Hilger~\cite{hilger2020decay}, large deviations by Deuschel, Giacomin and Ioffe~\cite{DGI00}, Funaki, Nishikawa~\cite{funaki2001large}, entropic repulsion by Deuschel, Giacomin~\cite{deuschel2000entropic}, the maximum of the field by Belius, Wu~\cite{BeliusWumax} and Wu, Zeitouni~\cite{wu2018subsequential}, uniqueness (or lack of thereof) of shift-ergodic infinite-volume gradient Gibbs states by Biskup, Koteck\'{y}~\cite{BK07} and Buchholz~\cite{buch2021phase}. A more detailed account of the literature can be found in the review articles~\cite{F05, V06}.

We complete this section by mentioning some results and recent progress which have been obtained on a related model: the \emph{integer-valued} random surfaces (formally obtained by replacing the Lebesgue measure by the counting measure on $\Z$ in the right-hand side of~\eqref{def.muTL}). In this setting, a temperature is usually incorporated in the definition of the model. A different phenomenology is then observed and the model is known to exhibit a phase transition in two-dimensions between a localized regime (at low temperature where the variance of the field remains bounded) and a delocalized regime (at high temperature where the variance of the field grows logarithmically). The existence of this phase transition was originally established in the celebrated article of Fr\"{o}hlich and Spencer~\cite{frohlich1981kosterlitz, kharash2017fr} (we also refer to the work of Wirth~\cite{wirth2019} on the maximum of the field based on these techniques), and has been the subject of recent developments in a series of works by Lammers~\cite{lammers2022height, lammers2022dichotomy, lammers2023bijecting}, van Engelenburg and Lis~\cite{van2023elementary, van2023duality} and Aizenman, Harel, Peled and Shapiro~\cite{aizenman2021depinning}. In the high temperature regime and in the case of the discrete Gaussian model (i.e., when $V(x) = x^2/2$), the scaling limit of the model was recently identified by Bauerschmidt, Park and Rodriguez by implementing a delicate renormalization group argument in~\cite{BPR1, BPR2}.

\subsubsection{Parabolic equations with degenerate random coefficients and the random conductance model}

In the uniformly elliptic setting, upper bounds on the heat kernel were obtained in the celebrated work of Nash~\cite{nash1958continuity}. Due to the connections between heat kernels and reversible random walks, it has been an active line of research to extend these heat kernel estimates to random degenerate environments, and two cases can be distinguished: the \emph{static} environments and the \emph{dynamic} environments. A typical example of static random degenerate environment is the supercritical Bernoulli (bond) percolation cluster. In that case, the upper bounds on the heat kernel were established by Barlow~\cite{Barlowheat} and Mathieu, Remy~\cite{mathieu2004isoperimetry}. These bounds (or the ingredient developed to prove it) became one of the ingredients in the proof of the quenched invariance principle for the random walk on the percolation cluster by Sidoravicius, Sznitman~\cite{SS}, Berger, Biskup~\cite{BB}, Mathieu, Piatnitski~\cite{MP}, the parabolic Harnack inequality and the local limit theorem by Barlow, Hambly~\cite{BH}. The existence of heat kernel upper and lower bounds (matching the ones of the lattice) have been established for more general degenerate environments satisfying suitable moments assumptions by Andres, Deuschel, Slowik~\cite{andres2016heat, andres2019heat, ADS20} and Andres, Halberstam~\cite{andres2021lower}, but this phenomenon is not generic and anomalous heat kernel decay has been proved for some random degenerate environments by Berger, Biskup, Hoffman and Kozma~\cite{berger2008anomalous}, Boukhadra~\cite{boukhadra2010heat}, Biskup, Boukhadra~\cite{biskup2012subdiffusive} and Buckley~\cite{Buckley}. Besides the question of the behavior of the heat kernel, the invariance principle has been established for degenerate conductances by Biskup, Prescott~\cite{BP}, Andres, Barlow, Deuschel, Hambly~\cite{andres2013invariance}, Mathieu~\cite{mathieu2008quenched}, Procaccia, Rosenthal, Sapozhnikov~\cite{procaccia2016quenched} and Bella, Sch\"{a}ffner~\cite{bella2020quenched}. We refer to to~\cite{biskupsurvey} for a survey of the literature on the random conductance model

Significant progress have been achieved in the case of \emph{dynamic} environments (which is the relevant one for the problem considered in this article). In this setting, the invariance principle has been proved under various assumptions on the environment by Boldrighini, Minlos, Pellegrinotti~\cite{boldrighini1997almost, boldrighini2007random}, Rassoul-Agha, Sepp\"{a}l\"{a}inen~\cite{rassoul2005almost}, Bandyopadhyay, Zeitouni~\cite{bandyopadhyay2005random}, Dolgopyat, Keller, Liverani~\cite{dolgopyat2008random}, Avena~\cite{avena2012symmetric} and Redig, V\"{o}llering~\cite{redig2013random}, and for general ergodic degenerate conductances with moment conditions by Andres, Chiarini, Deuschel and Slowik~\cite{andres2018quenched} (a local limit theorem was further established in~\cite{andres2021quenched}).

Finally, heat kernel upper bounds were established for a class of degenerate dynamic environments satisfying a moment assumption on the \emph{moderated} environment introduced above by Mourrat and Otto in~\cite{MO16}. The proof of Theorem~\ref{main.theorem} strongly relies on their techniques. Combining and enhancing the techniques of~\cite{andres2018quenched} and~\cite{MO16}, Biskup and Rodriguez~\cite{biskup2018limit} established the quenched invariance principle for random walks evolving in a dynamic degenerate environment satisfying an assumption related to the one used in~\cite{MO16}. In this line of research, we finally mention the recent contribution of Biskup, Pan~\cite{biskup2022invariance} which establishes a quenched invariance principle for a class of ergodic degenerate environments in the one-dimensional setting.

\subsection{Further comments and perspective} \label{sectionfurthercomm}
It is plausible that the techniques developed in this article can be further developed to obtain more precise information on the behavior of the random surfaces with an interaction potential satisfying Assumption~\ref{assumption1}. It seems for instance reasonable to us that the fluctuation estimate of Proposition~\ref{prop3.4} can be used to prove that the surface tension of the model is strictly convex (i.e., that the eigenvalues of its Hessian are always strictly positive). The strict convexity of the surface tension plays an important role in the proof of the hydrodynamic limit in~\cite{FS}, and we further believe that this result could be combined with the estimate of Theorem~\ref{main.theorem} to prove a quantitative version of the hydrodynamic limit following the techniques of~\cite{armstrong2022quantitative}. Once the quantitative hydrodynamic limit has been established, it should be possible to develop a large-scale regularity theory for the model (see~\cite[Theorem 1.5]{armstrong2022quantitative}). This result would then be useful to quantify the ergodicity of the environment appearing in the Helffer-Sj\"{o}strand representation formula and would be helpful to establish a quantitative version of the scaling limit of the model (following the insight of~\cite{NS, GOS}). We refer to the introduction of~\cite{armstrong2022quantitative} for a more detailed description of this line of research. We plan to investigate this in a future work. On a qualitative level, we mention that it would be interesting to investigate whether the techniques of Biskup and Rodriguez~\cite{biskup2018limit} can be adapted to the framework considered here to also identify the scaling limit of the model.

\subsection{Organization of the paper}

The rest of the paper is organized as follows. Section~\ref{section2} collects some notation and preliminary results. In Section~\ref{Section3}, we prove a stochastic integrability estimate for the gradient of a random surface distributed according to the periodic Gibbs measure $\mu_{\mathbb{T}_L}$ (Proposition~\ref{prop.stochintgradfield}), and deduce from it a fluctuation estimate for the Langevin dynamic (Proposition~\ref{prop3.4}). Section~\ref{section4ondiagupp} combines the results of Section~\ref{Section3} with the techniques and results of Mourrat and Otto~\cite{MO16} (essentially adapting their argument to obtain an on-diagonal upper bound for the heat kernel in the case of the torus, and when the environment is not bounded from above but possesses strong stochastic integrability properties), and completes the proof of Theorem~\ref{main.theorem} by using the Helffer-Sj\"{o}strand representation formula.

\subsection{Convention for constants and exponents}

Throughout this article, the symbols $C$ and $c$ denote positive
constants which, except if explicitly stated, may vary from line to line, with $C$ increasing and $c$ decreasing. We will always assume that $C \in [1 , \infty)$ and $c \in ( 0 , 1]$. These constants may depend on various parameters which will be made explicit in the statements by the following convention: we will write $C: =C(d , V)$ to specify that the constant $C$ depends only on $d$ and $V$.

\bigskip

\textbf{Acknowledgments:} The author is indebted to S. Armstrong, M. Harel, P. Lammers, J.-C. Mourrat, R. Peled, F. Schweiger, O. Zeitouni for encouragement and helpful conversations on the topic of this work, and is specifically grateful to F. Schweiger, O. Zeitouni for suggesting to prove Proposition~\ref{prop.stochintgradfield}, to P. Lammers for explaining a short proof of this result in the case of symmetric potentials (on which the proof below is based), and to J.-C. Mourrat for explaining the arguments of~\cite{MO16}.

\section{Notation and preliminary results} \label{section2}

\subsection{General notation} \label{secgeneralnotation}

We fix an integer $L \in \N$ with $L \geq 1$, consider the torus $\mathbb{T}_L := (\Z / (2L+1)\Z)^d$, and denote by $\pi : \Zd \to \mathbb{T}_L$ the canonical projection. Given a subset $U \subseteq \mathbb{T}_L$ or $U \subseteq  \Zd$, we let $E(U)$ be the set of \emph{positively oriented} edges of $U$ (for some pre-determined orientation). For $r \in \N$, we let $\Lambda_r := \left\{ - r , \ldots , r \right\}^d \subseteq \Zd$ and identify these boxes as subsets of the torus using the canonical embedding $\pi$. We note that the canonical embedding $\pi_{| \mathbb{T}_L}$ restricted to the box $\Lambda_L$ is a bijection, whose inverse will be denoted by $\pi_{| \mathbb{T}_L}^{-1}$. We denote by $\left| \cdot \right|$ be the Euclidean norm on $\Zd$, and, for $x \in \mathbb{T}_L$, we write $\left| x \right| := | \pi_{| \mathbb{T}_L}^{-1}(x) |$.

Given an edge $e \in E \left( \mathbb{T}_L \right),$ and a vertex $x \in \Lambda_L$, we write $x \in e$ if $x$ is one of the endpoints of $e$. Given two edges $e , e' \in  E \left( \mathbb{T}_L \right)$, we write $e \cap e' \neq \emptyset$ if $e$ and $e'$ have at least one endpoint in common. 

Given an edge $e \in \mathbb{T}_L$, we write $\sum_{x \in e}$ and $\sum_{e' \cap e \neq \emptyset}$ to respectively sum over the endpoints of $e$ and over the edges which have (at least) one endpoint in common with $e$. Given a vertex $x \in \mathbb{T}_L$, we write $\sum_{e \ni x}$ to sum over the edges which have $x$ as an endpoint.

Given two real numbers $a , b$, we denote by $a \wedge b = \min(a , b)$ and by $a \vee b = \max(a , b)$, and by $\lfloor a \rfloor$ and $\lceil a \rceil$ the floor and ceiling of $a$. We denote by $\indc_A$ the indicator function of a set $A$ and, for $x \in \mathbb{T}_L$, we let $\delta_x$ be the function defined on the torus by the formula: $\delta_x(y) = 0$ if $y \neq x$ ans $\delta_x(x) = 1$.

For any potential $V : \R \to \R$ satisfying Assumption~\ref{assumption1}, we denote by
\begin{equation} \label{def.RV}
    R_V := 2 \inf \left\{ R \geq 1 \, : \, \inf_{|x| \geq R} V''(x) \geq 1 \right\}.
\end{equation}
Assumption~\ref{assumption1} guarantees that $R_V$ is a finite nonnegative real number.

\subsubsection{Functions} \label{subSectionsets}

Given a subset $U \subseteq \mathbb{T}_L$ or $U \subseteq  \Zd$, we denote by $\left| U \right|$ the cardinality of $U$. We have in particular $\left| \mathbb{T}_L \right| = (2L+1)^d$. For any function $f : U \to \R$, and any exponent $p \geq 1$, we define the $L^p$-norm and the normalized $L^p$-norm of $f$ by the formulae
\begin{equation*}
    \left\| f \right\|_{L^p(U)}^p := \sum_{x \in U} f(x)^p ~~\mbox{and}~~ \left\| f \right\|_{\underline{L}^p(U)}^p := \frac{1}{|U|} \sum_{x \in U} f(x)^p.
\end{equation*}
We denote the discrete gradient of a function $f : \mathbb{T}_L \to \R$ over a positively oriented edge $e = (x , y) \in E \left(  \mathbb{T}_L\right)$ by the formula
\begin{equation*}
    \nabla f(e) := f(y) - f(x).
\end{equation*}
We also define $f(e) = (f(x) + f(y))/2$. This definition is motivated by the following identity: for any pair of functions $f , g : \mathbb{T}_L \to \R$ and any $e \in E(\mathbb{T}_L),$
\begin{equation*}
    \nabla (fg)(e) = f(e) \nabla g(e) + g(e) \nabla f(e).
\end{equation*}
We extend the definition of $L^p$-norms to functions defined on edges by writing, for any function $u : E(\mathbb{T}_L) \to \R$, 
\begin{equation*}
    \left\| u \right\|_{L^p(U)}^p := \sum_{e \in E(U) } |u(e)|^p ~~\mbox{and}~~ \left\| u \right\|_{\underline{L}^p(U)}^p := \frac{1}{|U|} \sum_{e \in E(U) } |u(e)|^p.
\end{equation*}
We define the nonlinear elliptic operator $\nabla \cdot V' (\nabla u)$ by the formula
\begin{equation*}
    \nabla \cdot V'(\nabla u)(x) := \sum_{\substack{e \in E \left( \mathbb{T}_L\right) \\ e = (x , y)}} V'(\nabla u(e)) - \sum_{\substack{e \in E \left( \mathbb{T}_L\right) \\ e = (y , x)}} V'(\nabla u(e)) .
\end{equation*}
This definition takes into account the set $E \left( \mathbb{T}_L\right)$ is defined to be the set of positively oriented edges. Making this distinction is useful to cover the case of potentials $V$ which are not symmetric. The main property of this operator is that it satisfies the following discrete integration by parts property: for any pair of functions $u , v : \mathbb{T}_L \to \R$,
\begin{equation*}
    \sum_{x \in \mathbb{T}_L} \nabla \cdot V'(\nabla u)(x) v(x) = - \sum_{e \in E \left( \mathbb{T}_L \right)} V' \left( \nabla u(e) \right) \nabla v(e).
\end{equation*}

\subsubsection{Parabolic equations and heat kernel}

An environment is a measurable map $\a : (0,\infty) \times E(\mathbb{T}_L) \to [0 , \infty)$. Given an environment, we denote by $\nabla \cdot \a \nabla$ the dynamic elliptic operator defined by the formula: for any map $u : (0,\infty) \times \mathbb{T}_L  \to \R$ and any $(t , x) \in (0,\infty) \times \mathbb{T}_L$,
\begin{equation} \label{eq:discdynlaplacian}
    \nabla \cdot \a \nabla u (t , x) = \sum_{\substack{e \in E \left( \mathbb{T}_L\right) \\ e = (x , y)}} \a (t , e) \nabla u(e) - \sum_{\substack{e \in E \left( \mathbb{T}_L\right) \\ e = (y , x)}}  \a (t , e) \nabla u(e).
\end{equation}
This operator satisfies the discrete integration by parts property: for any pair of functions $u, v : (0,\infty) \times \mathbb{T}_L  \to \R$,
\begin{equation*}
    \sum_{x \in \mathbb{T}_L} \nabla \cdot \a \nabla u (t , x) v(t , x) = - \sum_{e \in E(\mathbb{T}_L)} \a(t , e) \nabla u (t , e) \nabla v (t , e).
\end{equation*}
For $(s, y) \in [0 , \infty) \times \mathbb{T}_L$, we define the heat kernel $P_\a(\cdot , \cdot ; s , y) : (s , \infty) \times \mathbb{T}_L \to \R$ to be the solution of the parabolic equation
\begin{equation} \label{eq:11012312}
    \left\{ \begin{aligned}
     \partial_t P_\a (t , x ; s , y) - \nabla \cdot \a \nabla P_\a (t , x ; s , y) &= 0 &~~\mbox{for} &~~ (t , x) \in (0 , \infty) \times \mathbb{T}_L, \\
    P_\a (s , x ; s , y) & = \delta_y(x) - \frac{1}{\left| \mathbb{T}_L \right|} & ~~\mbox{for} &~~ x \in \mathbb{T}_L.
    \end{aligned} \right.
\end{equation}
To simplify the notation, we write $P_\a(t , x)$ instead of $P_\a(t , x ; 0 , 0).$

\begin{remark}
The preservation of mass for parabolic equations shows that the sum $\sum_{x \in \mathbb{T}_L} P_\a(t , x)$ is constant in time. The normalizing term $1/ \left| \mathbb{T}_L \right|$ in~\eqref{eq:11012312} ensures that this sum is equal to $0$, and in fact ensures that the heat kernel $P_\a$ converges to $0$ as the time tends to infinity. 
\end{remark}

\begin{remark}
The maximum principle for parabolic equation ensures that, for any $(t , x) \in (0 , \infty) \times \mathbb{T}_L$,
\begin{equation} \label{upperbound.PAmaxprinc}
   -\frac{1}{\left| \mathbb{T}_L \right|}\leq P_\a (t , x) \leq 1 - \frac{1}{\left| \mathbb{T}_L \right|}.
\end{equation}
Using these inequalities and the identity $\sum_{x \in \mathbb{T}_L} P_\a(t , x) = 0$ for any $t \geq 0$, we see that, for any $t \geq 0$,
\begin{equation} \label{L1normheatkernelbounded}
    \left\| P_\a (t , \cdot)\right\|_{L^1 \left( \mathbb{T}_L \right)} \leq 2.
\end{equation}
\end{remark}

\subsection{The Langevin dynamic and the Helffer-Sj\"{o}strand representation formula} \label{section2.2}

The Gibbs measure $\mu_{\mathbb{T}_L}$ is naturally associated with the Langevin dynamic defined below. In the following definition, we let $L \in \N$ be an integer, consider a collection $\left\{ B_t(x) \, : \, t \geq 0, \, x \in \mathbb{T}_L \right\}$ of independent Brownian motions, and let $\phi : \mathbb{T}_L \to \R$ be a random surface sampled according to the Gibbs measure~$\mu_{\mathbb{T}_L}$ independently of the Brownian motions. 

\begin{definition}[Langevin dynamic in the torus]
We define the Langevin dynamic associated with the Gibbs measure $\mu_{\mathbb{T}_L}$ to be the solution $\phi_L : \mathbb{T}_L \to \R$ of the system of stochastic differential equations
\begin{equation} \label{def.Langevindynamics}
\left\{ \begin{aligned}
d \phi_L(t , x) &= \nabla \cdot V' \left( \nabla \phi_L \right)(t , x) + \sqrt{2} dB_t(x) &~~\mbox{for}~~(t , x) \in (0 , \infty) \times \mathbb{T}_L, \\
\phi_L(0 , x) &= \phi(x) & \mbox{for} ~~x \in \mathbb{T}_L.
\end{aligned}
\right.
\end{equation}
\end{definition}
We note that the dynamic $\phi_L$ can be seen as a \emph{deterministic} function of the initial condition $\phi$ and of the Brownian motions $\left\{ B_t(x) \, : \, t \geq 0, \, x \in \mathbb{T}_L \right\}$. To highlight this dependency, we will use the notation
\begin{equation*}
    \phi_L(t , x) \left( \phi , \left\{ B_t(x) \, : \, t \geq 0, \, x \in \mathbb{T}_L \right\} \right).
\end{equation*}
This will be useful in Section~\ref{fluctuationofgradphi}, as the dynamic can then be differentiated with respect to the increments of the Brownian motions.

The law of the dynamic $\phi_L$ is not exactly stationnary as the spatially averaged value of the dynamic is not constant: summing the first equation of~\eqref{def.Langevindynamics} over $x \in \mathbb{T}_L$ (and using a discrete integration by parts on the torus to cancel the term involving the nonlinear elliptic operator) shows the identity
\begin{equation*}
    \sum_{x \in \mathbb{T}_L} \phi_L(t , x) = \sum_{x \in \mathbb{T}_L} B_t(x).
\end{equation*}
In particular, the law of $\phi_L(t , \cdot)$ is not equal to $\mu_{\mathbb{T}_L}$ (if $t \neq 0$), as the sum would have to be equal to $0$. Nevertheless, this is the only obstruction and the process
\begin{equation*}
    \phi_L(t , \cdot) - \frac{1}{\left| \mathbb{T}_L \right|} \sum_{x \in \mathbb{T}_L} \phi_L(t , x) 
\end{equation*}
is stationnary both with respect to the space and time variables. It is also reversible. Note that, since the second term in the right-hand side is spatially constant, the discrete gradient $ \nabla \phi_L$ is a stationnary process.

We next state the Helffer-Sj\"{o}strand representation which allows to express the variance of linear functionals of a random surface distributed according to $\mu_{\mathbb{T}_L}$ in terms of the solution of a random parabolic equations defined in terms of the Langevin dynamic. The formula was initially introduced in~\cite{helffer1994correlation, NS, DGI00, GOS} and is stated below in the case of the torus for the specific observable $\phi(0)$.

\begin{proposition}[Helffer-Sj\"{o}strand representation formula on the torus]
Let $P_\a$ be the solution of the parabolic equation in the torus
\begin{equation*}
    \left\{ \begin{aligned}
     \partial_t P_\a(t,x) - \nabla \cdot \a \nabla P_\a(t,x)  &= 0 &~~\mbox{for} &~~ (t , x) \in (0 , \infty) \times \mathbb{T}_L, \\
    P_\a (0 , x) & = \delta_0(x) - \frac{1}{\left| \mathbb{T}_L \right|} & ~~\mbox{for} &~~ x \in \mathbb{T}_L,
    \end{aligned} \right.
\end{equation*}
where $\a : ( 0 , \infty) \times E \left( \mathbb{T}_L \right) \to [0, \infty)$ is the random dynamic environment defined by the formula, for any $(t , e) \in (0, \infty) \times E(\mathbb{T}_L)$
\begin{equation*}
    \a(t , e) := V''(\nabla \phi_L (t , e)).
\end{equation*}
Then one has the identity
\begin{equation*}
    \var_{\mathbb{T}_L} \left[  \phi(0) \right] = \E \left[ \int_0^\infty P_\a (t , 0) \, dt \right].
\end{equation*}
\end{proposition}

\subsection{Efron's monotonicity theorem for log concave measures}
In this section, we state the Efron's monotonicity theorem for a pair of independent log-concave random variables due to Efron~\cite{Efron1965}.

\begin{theorem}[Efron's monotonicity theorem~\cite{Efron1965}]
    Let $(X , Y)$ be a pair of independent, real-valued  and log-concave random variables and let $\Psi : \R^2 \to \R$ be a function which is nondecreasing in each of its arguments, then the conditional expectation
    \begin{equation*}
            \E \left[ \Psi(X , Y) \, | \, X + Y = s \right] ~\mbox{is nondecrasing in}~ s.
    \end{equation*}
\end{theorem}

\subsection{The discrete Gagliardo-Nirenberg-Sobolev inequality} \label{SectionDGNS}

We state below the discrete version of the standard Gagliardo-Nirenberg-Sobolev inequality in the torus. The proof can be deduced from the standard Gagliardo-Nirenberg inequality in bounded domain for which we refer to~\cite{Nirenberg59}.

\begin{proposition}[discrete Gagliardo-Nirenberg-Sobolev inequality on the torus] \label{propDGS}
Fix three exponents $\kappa , \lambda, \mu \in (1 , \infty)$ and let $\theta \in [0,1]$ be such that the relation
\begin{equation*}
    \frac{1}{\kappa} =  \theta \left( \frac{1}{\lambda} - \frac{1}{d} \right) + \frac{1-\theta}{\mu}
\end{equation*}
holds. Then there exists a constant $C := C(d , \kappa, \lambda, \mu, \theta) < \infty$ such that for any $L \geq 1$ and any function $f : \mathbb{T}_L \to \R$,
\begin{equation*}
    \left\| f \right\|_{\underline{L}^\kappa (\mathbb{T}_L)} \leq C L^\theta \left\| \nabla f \right\|_{\underline{L}^\lambda (\mathbb{T}_L)}^\theta \left\| f \right\|_{\underline{L}^\mu(\mathbb{T}_L)}^{1-\theta} + C  \left\| f \right\|_{\underline{L}^2(\mathbb{T}_L)}.
\end{equation*}
If $f : \mathbb{T}_L \to \R$ satisfies the additional assumption $\sum_{x \in \mathbb{T}_L} f(x) = 0$, then
\begin{equation*}
    \left\| f \right\|_{\underline{L}^\kappa(\mathbb{T}_L)} \leq C L^\theta \left\| \nabla f \right\|_{\underline{L}^\lambda(\mathbb{T}_L)}^\theta \left\| f \right\|_{\underline{L}^\mu(\mathbb{T}_L)}^{1-\theta}.
\end{equation*}
\end{proposition}

In the proofs below, we will apply the discrete Gagliardo-Nirenberg-Sobolev inequality and the H\"{o}lder inequality with the following collections of exponents:
\begin{equation} \label{explambdakappanu}
    \lambda_d := \frac{2d+2}{d+2},~~ \kappa_d := \frac{d \lambda_d}{d - \lambda_d}, ~~ \sigma_d := \frac{2\kappa_d}{\kappa_d - 2} ~~ \mbox{and}~~ \tau_d = \frac{2\lambda_d}{ 2-\lambda_d},
\end{equation}
and
\begin{equation} \label{explambdakappanubis}
 \lambda'_d := \frac{2d+3}{d+2}, ~~\tau_d' :=  \frac{2\lambda_d'}{2-\lambda_d'} ~~\mbox{and}~~ \theta_d := \frac{2}{3}\frac{2d+3}{2d+2}.
\end{equation}
They are chosen so as to satisfy the following properties:
\begin{enumerate}
\item For any dimension $d \geq 2$, $1 < \lambda_d < 2 < \kappa_d < \infty$, and the Gagliardo-Nirenberg-Sobolev inequality can be applied with the exponents $\kappa = \kappa_d$, $\lambda = \lambda_d$ and $\theta = 1$ (and arbitrary $\mu$).
\item The pair of exponents $(\lambda_d , \tau_d)$ and $(\kappa_d , \sigma_d)$ are chosen so as to satisfy H\"{o}lder inequalities, and we have
\begin{equation*}
     \frac{1}{\tau_d} + \frac{1}{2}= \frac{1}{\lambda_d} ~~\mbox{and} ~~ \frac{1}{\kappa_d} + \frac{1}{\sigma_d} = \frac{1}{2}.
\end{equation*}
We also note that the following identities hold
\begin{equation} \label{eq:sigmataud}
    \frac{1}{\sigma_d} + \frac{1}{\tau_d} = \frac{1}{d} ~~ \mbox{and}~~ \frac{1}{\tau_d'} + \frac{1}{2} = \frac{1}{\lambda_d'} .
\end{equation}
\item For any function $f : \mathbb{T}_L \to \R$, one has the inequality
\begin{equation*}
    \left\| f \right\|_{L^{\kappa_d} (\mathbb{T}_L)} \leq C L^{\theta_d} \left\| \nabla f \right\|_{\underline{L}^{\lambda_d'} (\mathbb{T}_L)}^{\theta_d} \left\| f \right\|_{\underline{L}^2(\mathbb{T}_L)}^{1-\theta_d} + C \left\| f \right\|_{\underline{L}^2(\mathbb{T}_L)}.
\end{equation*}
Applying Young's inequality for product, we deduce that, for any $\ep \in (0,1]$,
\begin{equation*}
\left\| f \right\|_{\underline{L}^{\kappa_d} (\mathbb{T}_L)} \leq \ep L \left\| \nabla f \right\|_{\underline{L}^{\lambda_d'} (\mathbb{T}_L)} + C \ep^{-\frac{\theta_d}{1-\theta_d}} \left\| f \right\|_{\underline{L}^2(\mathbb{T}_L)}.
\end{equation*}
\end{enumerate}

\begin{remark}
The same inequalities hold on more general subsets than the torus, we will use it below in annuli of the form $A_r := \Lambda_{2r} \setminus \Lambda_{r}$ with $r \in \{ 1 , \ldots, \frac{L}{2}\}$ which can be seen as a subset of the torus (using the identification mentioned in Section~\ref{secgeneralnotation}). In this setting, we have, for any $f : A_r \to \R$ and any $\ep \in (0,1]$,
\begin{equation*}
\left\| f \right\|_{\underline{L}^{\kappa_d} (A_r)} \leq \ep r \left\| \nabla f \right\|_{\underline{L}^{\lambda_d'} (A_r)} + C \ep^{-\frac{\theta_d}{1-\theta_d}} \left\| f \right\|_{\underline{L}^2(A_r)}.
\end{equation*}
\end{remark}

\subsection{Maximal inequalities}

In this section, we recall some classical properties of maximal functions. We let $(\Omega , \mathcal{F} , \mathbb{P})$ be a probability space, and let $(\theta_x)_{x \in \Zd}$ be a measure preserving action of $\Zd$ on this space. For every measurable function $f : \Omega \to \R$, we define the maximal function
\begin{equation} \label{def.euclideanballs}
    M (f) := \sup_{r \in \N} \frac{1}{|\Lambda_r|} \sum_{x \in \Lambda_r} f(\theta_x \omega).
\end{equation}
We next record the $L^p$ maximal inequality, which can be obtained as a consequence of the weak type~$(1,1)$ estimate~\cite[Theorem 3.2]{AK81} with the Marcinkiewicz interpolation theorem (see~\cite[Appendix D]{taylor2006measure}). The result is stated and used in~\cite[Appendix A]{MO16}.

\begin{proposition}[$L^p$ Maximal inequality] \label{propmaximalineq}
For any $p \in ( 1 , \infty]$, there exists a constant $C := C(p,d) < \infty$ such that, for any $f \in L^p(\Omega)$,
\begin{equation*}
    \left\| M (f) \right\|_{L^p(\Omega)} \leq C  \left\| f \right\|_{L^p(\Omega)}.
\end{equation*} 
\end{proposition}

\subsection{The anchored Nash estimate of Mourrat and Otto}

In this section, we record the anchored Nash estimate proved by Mourrat and Otto~\cite[Theorem 2.1]{MO16}. In the statement below and later, we will make use of the notation $|\cdot |_* = |\cdot | + 1.$

\begin{theorem}[Anchored Nash inequality, Theorem 2.1 of~\cite{MO16}] \label{theoremanchoredNash}
Let $p \in (d , \infty)$, $p' \in (d , \infty]$, and $\theta \in [\theta_c , 1]$, where $\theta_c \in [0,1)$ is defined by
\begin{equation} \label{def.thetac}
    \frac{1}{\theta_c} = 1 + \frac{dp+ 2p}{dp + 2d} \left( \frac{p'}{d} -1 \right).
\end{equation}
Define $\alpha, \beta , \gamma \in [0 , 1)$ by 
\begin{equation} \label{def.alphabetagamma}
    \alpha := (1-\theta) \frac{d}{d+2} + \theta \frac{p}{p+2}, ~~ \beta := (1 - \theta) \frac{2}{d+2}, ~~\mbox{and} ~~ \gamma := \theta \frac{2}{p+2}.
\end{equation}
There exists $C := C(d , p , q , \theta) < \infty$ such that, for any function $f : \Zd \to \R$, and $w : E(\Zd) \to (0 , \infty)$,
\begin{equation*}
    \left\| f \right\|_{L^2 \left( \Zd \right)} \leq C \left( M(w^{-p'})^{\frac{1}{p'}} \left\| w \nabla f \right\|_{L^2 \left( \Zd \right)} \right)^{\alpha} \left\| f \right\|_{L^1 \left( \Zd \right)}^\beta \| |x|^{p/2}_* f \|_{L^2 \left( \Zd \right)}^\gamma.
\end{equation*}
\end{theorem}

\begin{remark}
The statement of the maximal function $M(w^{-p'})$ is defined with respect to (Euclidean) balls in~\cite[Theorem 2.1]{MO16} and with boxes in~\eqref{def.euclideanballs}. The two statements are equivalent, but writing it with boxes will be convenient to state the periodic version of the result in Proposition~\ref{theoremanchoredNashtorus} below.
\end{remark}

\begin{remark}
    By~\cite[(3.7)]{MO16} (or explicit computations), we have $\alpha + \beta + \gamma = 1$ as well as the identity
    \begin{equation} \label{eq:3.7}
            \alpha - \frac{ (p-d)\gamma}{2} = \frac d2 \left( \beta + \gamma \right) = \frac d2 (1 - \alpha) ~~ \implies ~~ 1 - \frac{(p-d)\gamma }{2\alpha} =  -\frac d2 \left( 1 - \frac 1\alpha \right).
    \end{equation}
\end{remark}

\subsection{Stochastic integrability for random variables}

We collect the following elementary property regarding the stochastic integrability stochastic processes.

\begin{lemma} \label{lemmastochint}
Let $(X_t)_{t \geq 0}$ be a continuous stochastic process and assume that there exists two constants $C_0 < \infty$ and $c_0 > 0$ and an exponent $a \geq 1$ such that, for any $t \geq 0$ and any $K \geq 0$,
\begin{equation} \label{eq:Pastocest}
    \P \left( |X_t| \geq K \right) \leq C_0 \exp \left( - c_0 K^a \right).
\end{equation}
Then there exist two constants $c_1 := c_1(C_0 , c_0) >0$ and $C_1 := C_1(c_0 , C_0) < \infty$ such that for any nonnegative function $f : (0 , \infty) \to \R$ satisfying $\int_0^\infty f(x) \, dx = 1$, and for any $K \geq 0$,
\begin{equation} \label{eq:Pastocest2}
     \P \left( \int_0^\infty f(t) |X_t| \geq K \right) \leq C_1 \exp \left( - c_1 K^a  \right).
\end{equation}
\end{lemma}

\begin{proof}
The proof is based on an application of Jensen inequality. Assumption~\eqref{eq:Pastocest} implies that there exists two constants $c_1 := c_1(C_0 , c_0) >0$ and $C_1 := C_1(c_0 , C_0) < \infty$ such that, for any $t \geq 0$,
\begin{equation*}
    \E \left[ \exp \left( c_1 |X_t|^a \right) \right] \leq C_1.
\end{equation*}
Using the convexity and monotonicity of the map $x \mapsto \exp( c_1 x^a)$ on $[0 , \infty)$, we see that
\begin{equation*}
     \E \left[ \exp \left( c_1  \left(\int_0^\infty f(t) |X_t| \, dt \right)^a   \right) \right]  \leq \E \left[ \int_0^\infty f(t) \exp \left( c_1 |X_t|^a \right) \, dt  \right] \leq C_1,
\end{equation*}
from which we deduce the bound~\eqref{eq:Pastocest2}.
\end{proof}

\section{Fluctuation estimates for the Langevin dynamic} \label{Section3}

This section is devoted to the proofs of two properties of the Langevin dynamic. The first one provides a stochastic integrability estimate on the gradient of the Langevin dynamic, the second one provides a fluctuation estimate for the Langevin dynamic, arguing that it can only remain contained in a fixed interval for a long time with small probability.

\subsection{Stochastic integrability estimate for the discrete gradient of the field}

In this section, we establish stochastic integrability estimates on the gradient of the Langevin dynamic. We first prove in Section~\ref{sec3.1.1} that the tail of the distribution of the discrete gradient of a random surface distributed according to the measure $\mu_{\mathbb{T}_L}$ decays at least like $K \mapsto \exp (- c K^r)$ (where $r$ is the exponent of Assumption~\ref{assumption1} encoding the growth of $V$). We then transfer this stochastic integrability from the Gibbs measure to the Langevin dynamic in Section~\ref{fluctuationofgradphi}.

\subsubsection{Stochastic integrability for the Gibbs measure} \label{sec3.1.1}

\begin{proposition} \label{prop.stochintgradfield}
There exist two constants $c := c(d, V) > 0$ and $C := C(d , V) < \infty$ such that, for any $L \geq 1$ and any edge $e \in E \left( \mathbb{T}_L \right)$, if $\phi$ is a random surface sampled according to $\mu_{ \mathbb{T}_L}$, then
\begin{equation*}
    \P \left[ \left| \nabla \phi (e) \right| > K \right] \leq C \exp \left( - c K^{r} \right).
\end{equation*}
\end{proposition}

We present below a proof of this proposition based on the Efron's monotonicity theorem for log-concave measure and a coupling argument (originally due to Funaki and Spohn~\cite{FS}) for the Langevin dynamic. We mention that, in the case when the potential $V$ is symmetric (i.e., $V(x) = V(-x)$ for all $x \in \R$), an alternative approach, relying on reflection positivity in the form of the chessboard estimate (following~\cite{MilosPeled2015} and~\cite[Lemma 3.9]{magazinov2020concentration}), would yield the same result.

\begin{proof}
We first prove the upper bound: there exists a constant $C := C(d , V) <\infty$ such that, for any $L \in \N$, any $e \in E \left( \mathbb{T}_L\right)$, if we let $\phi$ be a random surface sampled according to $\mu_{\mathbb{T}_L}$, then
\begin{equation} \label{eq:26120827}
    \E \left[ \left| \nabla \phi(e) \right|^2 \right] + \E \left[ \left| V' \left(\nabla \phi(e) \right) \right|^2 \right]  \leq C.
\end{equation}
The proof of the inequality~\eqref{eq:26120827} is based on the following identity: for any $x \in \mathbb{T}_L$,
\begin{equation} \label{eq:08312612}
    \E \left[ \phi(x) \nabla \cdot V'(\nabla  \phi)(x) \right] = - \frac{\left| \mathbb{T}_L \right| -1 }{\left| \mathbb{T}_L \right|}.
\end{equation}
To prove the identity~\eqref{eq:08312612}, we use the following result: for any probability density $f : \R^n \to [0 , \infty)$ which is continuously differentiable, such that $|y|f(y)$ tends to $0$ at infinity and $y \to (1 + |y|) \nabla f(y)$ is integrable, and for any index $i \in \{ 1 , \ldots , n \}$,
\begin{equation*}
    \int_{\R^n} y_i \frac{d f}{d y_i} (y) \, dy = -1.
\end{equation*}
Applying this result when the underlying space is $\Omega^\circ_{ \mathbb{T}_L}$ and noting that the function $\delta_x - \frac{1}{\left| \mathbb{T}_L \right|} \in \Omega^\circ_{ \mathbb{T}_L}$ has an $L^2(\mathbb{T}_L)$-norm equal to $(\frac{\left| \mathbb{T}_L \right|-1}{\left| \mathbb{T}_L \right|})^{\frac 12}$, we obtain
\begin{equation*}
    \frac{\left| \mathbb{T}_L \right|}{\left| \mathbb{T}_L \right| -1} \int_{\Omega^\circ_{ \mathbb{T}_L}} \phi(x) \nabla \cdot V'(\nabla  \phi)(x) \mu_{\mathbb{T}_L} (d\phi) = -1,
\end{equation*}
which is the identity~\eqref{eq:08312612}. Summing the inequality~\eqref{eq:08312612} over the vertices $x \in \mathbb{T}_L$ and performing a discrete integration by parts, we deduce that
\begin{equation*}
    \E \left[  \sum_{e' \in E \left(\mathbb{T}_L\right)} V'(\nabla \phi(e')) \nabla \phi(e') \right] = \left| \mathbb{T}_L \right| -1.
\end{equation*}
Using Assumption~\ref{assumption1} on the potential $V$, we see that the previous inequality implies
\begin{equation*}
    \E \left[  \sum_{e' \in E \left(\mathbb{T}_L \right)} \left| \nabla \phi(e') \right|^2 \right] \leq C \left| \mathbb{T}_L \right|.
\end{equation*}
Using that the spatial stationarity of the distribution $\mu_{\mathbb{T}_L}$ (since we consider the Gibbs measure $\mu_{\mathbb{T}_L}$ in the torus), we deduce that, for any edge $e \in \mathbb{T}_L$,
\begin{equation*}
    \E \left[ \left| \nabla \phi(e) \right|^2 \right] \leq \frac{C}{\left| \mathbb{T}_L \right|}  \E \left[  \sum_{e' \in E \left(\mathbb{T}_L\right)} \left| \nabla \phi(e') \right|^2 \right] \leq C.
\end{equation*}
We next note that, since the Gibbs measure $\mu_{\mathbb{T}_L}$ is log-concave, the Pr\'ekopa-Leindler inequality~\cite{prekopa1971logarithmic, prekopa1973logarithmic, leindler1972certain} implies that the distribution of the random variable $\nabla \phi
(e)$ is also log-concave. This implies that the tail of its distribution decays exponentially fast on the scale of its standard deviation, and thus all the moments of $\nabla \phi(e)$ are bounded uniformly in $L$. In particular, since the map $V'$ grows at most like a polynomial, we obtain the bound~\eqref{eq:26120827}. We then fix an edge $e \in E \left( \mathbb{T}_L \right)$ and introduce the collection of potentials~$(V_{e'})_{e' \in E(\mathbb{T}_L)}$
\begin{equation*}
    V_{e'} (x) := \left\{ \begin{aligned}
    V(x) &~\mbox{if}~ e' \neq e, \\
    \frac{V(x)}{2}  &~\mbox{if}~ e' = e. \\
    \end{aligned} \right.
\end{equation*}
We then denote by $\phi^e : \mathbb{T}_L \to \R$ be a random surface distributed according to the Gibbs measure
\begin{equation} \label{eq:09310312}
    \mu^e_{\mathbb{T}_L}(d \phi) := \frac{1}{Z_{\mathbb{T}_L}^e} \exp \left( - \sum_{e' \in E \left( \mathbb{T}_L \right)} V_{e'} \left( \nabla \phi(e') \right) \right)  d \phi.
\end{equation}
Since the measure~\eqref{eq:09310312} is log-concave, the random variable $\nabla \phi^e(e)$ is also log-concave. We next prove the following estimate: there exists a constant $C := C(d , V) < \infty$ such that
\begin{equation} \label{eq:10110312}
    \E \left[\left| \nabla \phi^e(e) \right|^2  \right] \leq C.
\end{equation}
The proof of~\eqref{eq:10110312} is based on a coupling argument for Langevin dynamic. To this end, we introduce the Langevin dynamic associated with the measure $\mu^e_{\mathbb{T}_L}$, i.e.,
    \begin{equation} \label{eq:21252711}
    \left\{ \begin{aligned}
    d \phi_L^e (t , x) &= \nabla \cdot V'_{e'}(\nabla \phi_L^e) (t , x) dt + \sqrt{2} d B_t(x) &~\mbox{for}~& (t , x) \in [0 , \infty] \times \mathbb{T}_L, \\
    \phi_L^e(0 , x) &= \phi^e(x) &~\mbox{for}~& x \in \mathbb{T}_L,
    \end{aligned} \right.
\end{equation}
where the initial data $\phi^e$ is distributed according to the measure $\mu^e_{\mathbb{T}_L}$ and is independent of the Brownian motions. We note that, as it was the case for~\eqref{def.Langevindynamics}, the process $\nabla \phi^e_L$ is stationary with respect to the time translations. We next couple the dynamic~\eqref{eq:21252711} to the one of~\eqref{def.Langevindynamics} by assuming that they are driven by the same Brownian motions and that the initial conditions $\phi^e$ and $\phi$ are independent. Subtracting the two dynamics, we observe that the difference $u := \phi_L - \phi^e_L$ solves the parabolic equation
\begin{equation} \label{eq:12390312}
    \partial_t u (t , x)- \nabla \cdot \a_e \nabla u (t , x) = \nabla \cdot \left[ \left(V_{e'}' - V'\right)(\nabla \phi_L) \right] (t , x) \hspace{5mm} \mbox{for} ~ (t , x) \in [0 , \infty] \times \mathbb{T}_L,
\end{equation}
with the definition
\begin{equation*}
    \a_e(t , e') :=
    \int_0^1 V_{e'}''(s \nabla \phi_L(t , e') + (1-s) \nabla \phi_L^e(t , e')) \, ds. 
\end{equation*}
Noting that the potentials $V_{e'}$ and $V$ are only different at the edge $e$, we may use an energy estimate on the equation~\eqref{eq:12390312} and obtain, for any $T \geq 0$,
\begin{equation} \label{eq:13110312}
    \int_{0}^T \sum_{e' \in E \left( \mathbb{T}_L \right)} \a_e(t , e') \left| \nabla u(t ,e') \right|^2 \, dt \leq C \int_0^T \left| V' \left(\nabla \phi_L(t , e) \right) \nabla u(t,e) \right| \, dt + C \sum_{x \in \mathbb{T}_L}  \left| u(0 , x) \right|^2.
\end{equation}
The inequality~\eqref{eq:13110312} implies the following (weaker) estimate
\begin{equation*}
    \int_{0}^T \a_e(t , e) \left| \nabla u(t ,e) \right|^2 \leq \int_0^T \left| V' \left(\nabla \phi_L(t , e) \right) \nabla u(t,e) \right| \, dt + \sum_{x \in  \mathbb{T}_L}  \left| u(0 , x) \right|^2.
\end{equation*}
Assumption~\ref{assumption1} on the potential $V$ implies that there exists a constant $C := C(V) < \infty$ such that
\begin{equation} \label{eq:13100312}
    \a_e(t , e) \left| \nabla u(t ,e) \right|^2 \geq  \left| \nabla u(t ,e) \right|^2 - C.
\end{equation}
Substituting~\eqref{eq:13100312} into~\eqref{eq:13110312} and applying the Cauchy-Schwarz inequality, we deduce that
\begin{equation*}
    \int_{0}^T \left| \nabla u(t ,e) \right|^2 \, dt  \leq C T  + C \int_0^T V'( \nabla \phi_L(t,e))^2 \, dt + C \sum_{x \in \mathbb{T}_L }  \left| u(0 , x) \right|^2.
\end{equation*}
Using the definition $u := \phi_L - \phi^e_L$, we thus obtain
\begin{equation*}
    \int_{0}^T \left| \nabla \phi_L^e(t ,e) \right|^2 \, dt \leq  C T  + C \int_0^T \left( V'( \nabla \phi_L(t,e))^2 + \left|\nabla \phi_L(t,e)\right|^2\right) \, dt + C \sum_{x \in \mathbb{T}_L }  \left| u(0 , x) \right|^2.
\end{equation*}
Taking the expectation in both sides of the previous inequality, and using the stationarity of the gradients $\nabla \phi_L$ and $\nabla \phi_L^e$, we deduce that, for any $T > 0$,
\begin{equation*}
    \E \left[\left| \nabla \phi^e_L(0 ,e) \right|^2  \right] \leq C + C \E \left[ V'( \nabla \phi_L(0 , e) )^2  + \left|\nabla \phi_L(0,e)\right|^2 \right] + \frac{C}{T} \sum_{x \in \mathbb{T}_L}  \E \left[ \left| u(0 , x) \right|^2 \right].
\end{equation*}
Taking the limit $T \to \infty$ and using the bound~\eqref{eq:26120827} completes the proof of~\eqref{eq:10110312}.

We next let $Y$ be a real-valued random variable whose law is given by
\begin{equation*}
    \mu_Y := \frac{1}{Z_Y} \exp \left( - \frac{1}{2} V \left( y \right) \right) dy \hspace{5mm} \mbox{with} \hspace{5mm} Z_Y := \int_\R \exp \left( - \frac{1}{2} V \left( y \right) \right) dy.
\end{equation*}
We couple the random variables $Y$ and $\phi^e$ by assuming that they are independent. Using that the law of the random variable $Y$ is explicit, the independence of $Y$ and $\nabla \phi^e(e)$ and the bound~\eqref{eq:10110312}, we deduce that there exists a constant $c := c(d , V) > 0$ such that
\begin{align} \label{eq:15210312}
    \P \left[ Y \geq \nabla \phi^e(e) \right] & \geq \P \left[ \left\{ Y \geq 2 \E \left[\left| \nabla \phi^e(e) \right|  \right]  \right\} \cap \left\{ \nabla \phi^e(e) \leq 2 \E \left[\left| \nabla \phi^e(e) \right|  \right]   \right\} \right] \\
    & = \P \left[ \left\{ Y \geq 2 \E \left[\left| \nabla \phi^e(e) \right|  \right]  \right\} \right)  \P \left(\left\{ \nabla \phi^e(e) \leq 2 \E \left[\left| \nabla \phi^e(e) \right|  \right]   \right\} \right] \notag \\
    & \geq c. \notag
\end{align}
We next rely on the observation that the law of $\nabla \phi(e)$ (where $\phi$ is distributed according to the measure~$\mu_{\mathbb{T}_L}$) is equal to the law of the random variable $Y$ conditionally on the event $\left\{ Y - \nabla \phi^e(e) = 0 \right\}$. This property is a consequence of the following observation: if $X$ and $Z$ are two independent real-valued random variables with bounded continuous densities $f$ and $g$ then the law of $X$ conditionally on the event $\{ X - Z = 0\}$ has a density proportional to the function $fg$. In particular, for any non-negative function $F : \R \to [0 , \infty)$, one has the identity
\begin{equation} \label{eq:03121520}
    \E \left[ F(\nabla \phi(e)) \right] = \E \left[ F(Y) \,  | \,  Y - \nabla \phi^e(e) = 0 \right].
\end{equation}
We then introduce the constant $c_3 := \frac{c_- }{4r (r-1)} > 0$ (where $c_-$ is the constant appearing in Assumption~\ref{assumption1}) and the function
\begin{equation*}
    F(x) := \left\{ \begin{aligned}
    0 &~\mbox{if}~ x \leq  0, \\
    \exp \left( c_3 x^r \right) &~\mbox{if}~ x \geq 0.
    \end{aligned} \right.
\end{equation*}
Assumption~\ref{assumption1} on the potential $V$ implies that there exists a constant $C := C(V) < \infty$ such that
\begin{equation} \label{eq:15190312}
    \E \left[ F(Y) \right] = \frac{1}{Z_Y} \int_\R F(y) \exp \left( - \frac{1}{2} V(y)\right) \, dy \leq C.
\end{equation}
We then note that the Efron's monotonicity theorem applied to the pair of independent random variables $(Y , \nabla \phi^e)$, the nonnegativity and monotonicity of the function $F$ imply the almost sure inequality
\begin{equation} \label{eq:15400312}
    \E \left[ F(Y) \,  | \,  Y - \nabla \phi^e(e) = 0 \right] \indc_{\left\{ Y - \nabla \phi^e(e) \geq 0 \right\}} \leq \E \left[ F(Y) \,  | \,  Y - \nabla \phi^e(e) \right].
\end{equation}
Combining the bound~\eqref{eq:15190312} with the identity~\eqref{eq:15190312}, the lower bound~\eqref{eq:15210312} and the inequality~\eqref{eq:15400312} yields the existence of a constant $C := C(d , V) < \infty$ such that
\begin{align} \label{eq:15420312}
    \E \left[ F(\nabla \phi(e) ) \right] & = \E \left[ F(Y) \,  | \,  Y - \nabla \phi^e(e) = 0 \right] \\
    & \leq \frac{1}{\P \left( Y - \nabla \phi^e(e) \geq 0 \right)} \E \left[ \E \left[ F(Y) \,  | \,  Y - \nabla \phi^e(e) \right] \right] \notag \\
    & \leq \frac{1}{\P \left( Y - \nabla \phi^e(e) \geq 0 \right)} \E \left[ F(Y)\right] \notag \\
    & \leq C. \notag
\end{align}
The inequality~\eqref{eq:15420312} implies that there exist two constants $C := C(d , V) < \infty$ and $c := c(d , V) > 0$ such that, for any $K \geq 1$,
\begin{equation} \label{eq:15490312}
    \P \left[  \nabla \phi (e) > K \right] \leq C \exp \left( - c K^{r} \right).
\end{equation}
The same argument can be applied with the potential $\tilde V(x) :=  V(-x)$ to obtain the upper bound, for any $K \geq 1$,
\begin{equation} \label{eq:15500312}
    \P \left[  \nabla \phi (e) < - K \right] \leq C \exp \left( - c K^{r} \right).
\end{equation}
Combining~\eqref{eq:15490312} and~\eqref{eq:15500312} completes the proof of Proposition~\ref{prop.stochintgradfield}.
\end{proof}

\subsubsection{Stochastic integrability for the Langevin dynamic} \label{fluctuationofgradphi}

In this section, we extend the result of the previous section to the Langevin dynamic (using essentially a union bound and the stationarity of the dynamic).

\begin{proposition} \label{propositionsubdynamic}
There exist two constants $c := c(d , V) > 0$ and $C := C(d , V) < \infty$ such that, for any $T \geq 1$ and any $K \geq 1$,
\begin{equation} \label{eq:13510512}
    \P \left[ \sup_{t \in [0 , T]}  \left| \nabla \phi_L (t , e) \right| \geq K \right] \leq C T \exp \left( - c K^r \right).
\end{equation}
\end{proposition}

\begin{proof}
Fix $K \geq 1$ and let $N :=  K^{r}$. We have the inclusion of events
\begin{multline} \label{eq:11580512}
    \left\{ \sup_{t \in [0 , T]}  \left| \nabla \phi_L (t , e) \right| \geq K  \right\} \subseteq \left\{ \sup_{n \in \left\{ 0 , \ldots, \lfloor T N \rfloor \right\}} \left| \nabla \phi_L \left(\frac{n}{N} , e \right) \right| \geq \frac{K}{2} \right\} \\ 
    \bigcup \left\{ \sup_{n \in \left\{ 0 , \ldots, \lfloor T N \rfloor \right\}} \sup_{t \in \left[ \frac{n}{N} , \frac{n+1}{N} \right]} \left| \nabla \phi_L \left(t , e \right) - \nabla \phi_L \left(\frac{n}{N} , e \right) \right| \geq \frac{K}{2} \right\}.
\end{multline}
We then bound the probabilities of the two terms in the right-hand side separately. For the first one, a union bound, Proposition~\ref{prop.stochintgradfield}, and the identity $N :=   K^{r} $ yield
\begin{align} \label{eq:13540512}
    \P \left[ \sup_{n \in \left\{ 0 , \ldots, \lfloor T N \rfloor \right\}} \left| \nabla \phi_L \left(\frac{n}{N} , e \right) \right| \geq \frac{K}{2} \right] & \leq  \sum_{n = 0}^{\lfloor T N \rfloor} \P \left[  \left| \nabla \phi_L \left(\frac{n}{N} , e \right) \right| \geq \frac{K}{2} \right] \\
    & \leq C K^r T \exp \left( - c K^r \right) \notag \\
    & \leq  C T \exp \left( - c K^r \right), \notag
\end{align}
where we reduced the value of the constant $c$ in the third line to absorb the polynomial factor $K^r$. For the second term in the right-hand side of~\eqref{eq:11580512}, we first fix $n \in \left\{ 0 , \ldots, \lfloor T N \rfloor \right\}$ and use the definition of the Langevin dynamic~\eqref{def.Langevindynamics} to write
\begin{equation*}
     \nabla \phi_L \left(t , e \right) - \nabla \phi_L \left(\frac{n}{N} , e \right) = \int_{\frac{n}{N}}^t \nabla \left( \nabla \cdot V'(\nabla \phi_L) \right) (s , e) \, ds + \nabla B_{t}(e) - \nabla B_{\frac{n}{N}}(e).
\end{equation*}
This implies
\begin{multline} \label{eq:13340512}
    \sup_{t \in \left[ \frac{n}{N} , \frac{n+1}{N} \right]} \left| \nabla \phi_L \left(t , e \right) - \nabla \phi_L \left(\frac{n}{N} , e \right) \right| \\ \leq \int_{\frac{n}{N}}^{\frac{n+1}{N}} \left| \nabla \left( \nabla \cdot V'(\nabla \phi_L) \right) (s , e) \right| \, ds +  \sup_{t \in \left[ \frac{n}{N} , \frac{n+1}{N} \right]} \left| \nabla B_{t}(e) - \nabla B_{\frac{n}{N}}(e) \right|.
\end{multline}
Using the definition of the discrete gradient and Assumption~\ref{assumption1} on the potential $V$, we see that
\begin{equation*}
    \left| \nabla \left( \nabla \cdot V'(\nabla \phi_L) \right) (s , e) \right|  \leq \sum_{e' \cap e \neq \emptyset} \left| V'(\nabla \phi_L)(t , e') \right|  \leq C + \sum_{e' \cap e \neq \emptyset} \left|\nabla \phi_L (t , e') \right|^{r-1}.
\end{equation*}
Using Lemma~\ref{lemmastochint} (with $f = N \indc_{\left[\frac{n}{N} , \frac{n+1}{N}\right]}$), we deduce that
\begin{align} \label{eq:13350512}
    \P \left[ \int_{\frac{n}{N}}^{\frac{n+1}{N}} \left| \nabla \left( \nabla \cdot V'(\nabla \phi_L) \right) (s , e) \right| \, ds \geq \frac{K}{4} \right] & \leq C \exp \left( - c (N K )^{\frac{r}{r-1}} \right) \\
    & \leq C \exp \left( - c K^{r}\right). \notag
\end{align}
Additionally, the supremum of the Brownian motions can be estimated by noting that the difference of two independent Brownian motions is equal in law (up to a multiplicative constant equal to $\sqrt{2}$) to a Brownian motion. We obtain
\begin{align} \label{eq:13360512}
    \P \left[  \sup_{t \in \left[ \frac{n}{N} , \frac{n+1}{N} \right]} \left| \nabla B_{t}(e) - \nabla B_{\frac{n}{N}}(e) \right| \geq \frac{K}{4} \right] & = 
    \P \left[ \sup_{t \in \left[ 0 , 1 \right]} B_{t} \geq \frac{\sqrt{N}K}{4 \sqrt{2} }  \right] \\ & \leq C \exp \left( - c N K^2 \right) \notag \\
    & \leq C \exp \left( - c K^r \right). \notag
\end{align}
Combining~\eqref{eq:13340512},~\eqref{eq:13350512} and~\eqref{eq:13360512}, with a union bound, we have obtained
\begin{align} \label{eq:13570512}
    \lefteqn{\P \left[  \sup_{n \in \left\{ 0 , \ldots, \lfloor T N \rfloor \right\}} \sup_{t \in \left[ \frac{n}{N} , \frac{n+1}{N} \right]} \left| \nabla \phi_L \left(t , e \right) - \nabla \phi_L \left(\frac{n}{N} , e \right) \right| \geq \frac{K}{2} \right] } \qquad & \\ & \leq \sum_{n = 0}^{ \lceil T N \rceil} \P \left[ \sup_{t \in \left[ \frac{n}{N} , \frac{n+1}{N} \right]} \left| \nabla \phi_L \left(t , e \right) - \nabla \phi_L \left(\frac{n}{N} , e \right) \right| \geq \frac{K}{2} \right] \notag \\
    & \leq C NT \exp \left( - c K^r \right)\notag  \\
    & \leq C K^r T \exp \left( - c K^r \right) \notag \\
    & \leq  C T \exp \left( - c K^r \right). \notag
\end{align}
Combining~\eqref{eq:11580512},~\eqref{eq:13540512} and~\eqref{eq:13570512} completes the proof of~\eqref{eq:13510512}.
\end{proof}

\subsection{A fluctuation estimate for the Langevin dynamic}

Building upon the stochastic integrability estimate for the dynamic established in Proposition~\ref{propositionsubdynamic}, we prove that the dynamic cannot remain contained in a deterministic interval for a long time. The argument follows the one outline in Section~\ref{eq:18111101}, with additional technicalities to take into account that the second derivative of the potential $V$ is assumed to be unbounded from above. We recall the definition~\eqref{def.RV} of the constant $R_V$.

\begin{proposition}[Fluctuation for the Langevin dynamic] \label{prop3.4}
There exist two constants $C := C(d , V) < \infty$ and $c := c(d , V) >0$ such that, for any $T \geq 1$ and any edge $e \in \mathbb{T}_L$, 
\begin{equation} \label{upperboundfluctuation}
    \P \left[ \, \forall t \in [0 , T], \, \left| \nabla \phi_L(t , e) \right| \leq R_V \, \right] \leq C \exp \left( - c \left( \ln T \right)^{\frac{r}{r-2}} \right).
\end{equation}
\end{proposition}

\begin{proof}
We fix an edge $e \in E(\mathbb{T}_L)$ and will prove the following estimate: there exist two constants $C := C(d,V) < \infty$ and $c := c(d , V) > 0$ and a time $T_0 := T_0(d , V) < \infty$ such that, for any $T \geq T_0$,
\begin{equation} \label{upperboundfluctuation2}
    \P \left[ \, \forall t \in [0 , T], \, \left| \nabla \phi_L(t , e) \right| \leq R_V \, \right] \leq C \exp \left( - c \left( \ln T \right)^{\frac{r}{r-2}} \right).
\end{equation}
The bound~\eqref{upperboundfluctuation} can be deduced from~\eqref{upperboundfluctuation2} by increasing the value of the constant $C$. 
Let us fix a time $T \geq 1$ and let $N := (\ln T)/R_V^2$. The definition of the parameter $N$ is motivated by the following inequality: for any $T$ chosen sufficiently large (universally),
\begin{align} \label{eq:19132612}
    \mathbb{P} \left[ \left| B_{1/N} \right| \geq \frac{4}{3} R_V \right] & = \mathbb{P} \left[ \left| B_{1} \right| \geq \frac{4}{3} \sqrt{\ln T} \right] \\
    & = \frac{2}{\sqrt{2\pi}} \int_{\frac{4}{3}\sqrt{\ln T}}^\infty e^{-\frac{x^2}{2}} \, dx \notag  \\
    & \geq \frac{2}{\sqrt{2\pi}} \int_{\frac{4}{3}\sqrt{\ln T}}^{\frac{4}{3}\sqrt{\ln T} + 1} e^{-\frac{x^2}{2}} \, dx \notag \\
    & \geq  \frac{2}{\sqrt{2\pi}}  \exp \left( - \frac{1}{2} \left( \frac{4}{3} \sqrt{\ln T} + 1 \right)^2  \right) \notag \\
    & \geq \frac{1}{T^{\sfrac{9}{10}}}. \notag
\end{align}
The proof relies on the observation that a Brownian motion can be decomposed into mutually independent Brownian bridges and increments. To be more specific, we introduce the following sets and notation:
\begin{itemize}
    \item For each $k \in \N$ and each $x \in \mathbb{T}_L$, we let $W_k(\cdot ; x)$ be the Brownian bridge defined by the formula
    \begin{equation*}
        \forall t \in \left[0, \frac{1}{N}\right], \hspace{5mm} W_k(t ; x) := B_{t + \frac{k}{N}}(x) - B_{\frac{k}{N}}(x) - N t (B_{\frac{k+1}{N}}(x) - B_{\frac{k}{N}}(x)).
    \end{equation*}
    We will denote by $\mathcal{W} := \left\{  W_k (\cdot ; x) \, : \, k \in \N, x \in \mathbb{T}_L \right\}$ the collection of Brownian bridges.
    \item For each $k \in \N$ and each $y \in \mathbb{T}_L$, we denote by $X_k(y)$ the increment
    \begin{equation*}
            X_k(y) := B_{\frac{k+1}{N}}(y) - B_{\frac{k}{N}}(y).
    \end{equation*}
    We will denote by $\mathcal{X} := \left\{ X_k(x) \, : \, k \in \N, \, x \in \mathbb{T}_L\right\}$ the set of all the increments. For $(l,y) \in \N \times \mathbb{T}_L$, the set $\mathcal{X}_{l , y} := \left\{ X_k(x) \, : \, k \in \N, \, x \in \mathbb{T}_L, k \neq l, x \neq y\right\}$ denotes the collection of all the increments except~$X_l(y)$.
\end{itemize}
In particular, the Brownian bridges $\left\{ B_t(x) \, : \, x \in \mathbb{T}_L, \, t \geq 0 \right\}$ are fully determined by the Brownian bridges of $\mathcal{W}$ and the increments of $\mathcal{X}$. This implies, using the discussion of Section~\ref{section2.2}, that the dynamic $\phi_L$ is fully determined by the initial condition $\phi$, the Brownian bridges of $\mathcal{W}$ and the increments of $\mathcal{X}$. We thus introduce the notation 
$$\mathcal{R} := (\phi , \mathcal{X}, \mathcal{W}).$$
The set of all possible triplets $\mathcal{R}$ will be denoted by 
$$\Omega := \Omega_{\mathbb{T}_L}^{\circ} \times \R^{\N \times \mathbb{T}_L} \times C\left( \left[0,\frac 1N\right] , \R\right)^{\N \times \mathbb{T}_L}.$$
Since the dynamic $\left\{ \phi_L(t , x) \, : \, t \geq 0, \, x \in \mathbb{T}_L \right\}$ can interpreted as deterministic functions of $\mathcal{R} \in \Omega$, we  will write
\begin{equation*}
    \phi_L(t , x) := \phi_L(t , x) \left(  \mathcal{R} \right).
\end{equation*}
For $(l,y) \in \N \times \mathbb{T}_L$, we denote by $\mathcal{R}_{l , y} := ( \phi, \mathcal{X}_{l , y}, \mathcal{W})$ and by $\Omega_{l,y}$ the set of possible values for $\mathcal{R}_{l , y}$. We have the identities $\mathcal{R} = (X_l(y), \mathcal{R}_{l , y})$ and $\Omega = \R \times \Omega_{l,y}$. To emphasize the dependency of the dynamic on the increment $X_l(y)$, we will write
\begin{equation} \label{eq:15342911}
    \phi_L(t , x) = \phi_L(t , x) \left( X_l(y), \mathcal{R}_{l , y} \right).
\end{equation}
We denote by $\mathcal{F}_{\mathcal{R}, l,y}$ the $\sigma$-algebra generated by $\mathcal{R}_{l,y}$ and note that the increment $X_{l}(y)$ is independent of the $\sigma$-algebra $\mathcal{F}_{\mathcal{R}, l,y}$. For later use, we note that the dynamic $\phi_L(t , x)$ depends only on the increments $X_k(y)$ and the Brownian bridges $W_k(\cdot; y)$ such that $t \geq \frac{k}{N}$. This reflects the fact that the dynamic $\phi_L$ evaluated at the time $t$ depends only on the realization of the Brownian motions before the time $t$. 

We now fix a positively oriented edge $e \in E \left( \mathbb{T}_L \right)$ and let $y$ be the second endpoint of $e$. For any~$l \in \N$, we introduce the following random subset of $\R$ (depending on the collection $\mathcal{R}_{l , y}$),
\begin{equation} \label{eq:11200112}
    \mathcal{A}_l (\mathcal{R}_{l , y}) := \left\{  X \in \R \, : \, \left| \nabla \phi_L \left(\frac{l+1}{N} , e \right) \left( X, \mathcal{R}_{l , y} \right) \right| \leq R_V  \right\} \subseteq \R,
\end{equation}
where we used the notation introduced in~\eqref{eq:15342911}. In words, the set $\mathcal{A}_l (\mathcal{R}_{l , y})$ is the set of all possible values for the increment $X_l(y)$ such that the gradient of the dynamic $\phi_L$ computed at time $(l+1)/N$ at the edge $e$ with initial condition, Brownian bridges and increments given by $\mathcal{R} = (X_l(y) , \mathcal{R}_{l , y})$ belongs to the interval $[-R_V , R_V]$. 

We next introduce the event $A_l \subseteq \Omega$ defined as follows
\begin{equation} \label{eq:11210112}
    A_l := \left\{ \mathcal{R} := (X_l(y),  \mathcal{R}_{l , y}) \in \Omega \, : \,  X_l(y) \in \mathcal{A}_l (\mathcal{R}_{l , y}) ~~\mbox{and}~~ \frac{1}{\sqrt{2 \pi N}}\int_{\mathcal{A}_l (\mathcal{R}_{l , y})} e^{- \frac{x^2}{2N}} \, dx \leq 1- \frac{1}{T^{\sfrac{9}{10}}} \right\}.
\end{equation}
Since the law of the increment $X_l(y)$ is Gaussian of variance $1/N$ and since $X_l(y)$ is independent of the set $\mathcal{R}_{l , y}$, we have the almost sure upper bound
\begin{equation} \label{eq:28111220}
    \E \left[ \indc_{A_{l}}  \big| \mathcal{F}_{\mathcal{R}, l,y} \right] \leq 1- \frac{1}{T^{\sfrac{9}{10}}}.
\end{equation}
We next estimate the probability for the intersection of all the events $A_{l}$ for $l \in \left\{ 0 , \ldots, \lfloor NT \rfloor \right\}$ and prove the following stretched exponential decay in the time $T$,
\begin{equation} \label{eq:11492811}
    \P \left[ \bigcap_{l = 0}^{ \lfloor N T \rfloor} A_l \right] \leq \exp \left( - T^{\sfrac{1}{10}} \right).
\end{equation}
The proof of~\eqref{eq:11492811} is obtained by consecutive conditioning. We first note that, since the dynamic~$\phi_L(t , x)$ depends only on the increments $X_l(y)$ and the Brownian bridges $W_l(\cdot; y)$ such that $t \geq \frac{l}{N}$, the events $(A_0, , \ldots, A_{\lfloor NT \rfloor - 1})$ do not depend on the increment $X_{\lfloor NT \rfloor}(y)$, and are thus measurable with respect to the $\sigma$-algebra $\mathcal{F}_{\mathcal{R}, \lfloor NT \rfloor,y}$. Combining this observation with the upper bound~\eqref{eq:28111220}, we obtain
\begin{align*}
    \P \left[ \bigcap_{l = 0}^{\lfloor NT \rfloor } A_l \right] & = \E \left[ \prod_{l = 0}^{\lfloor NT \rfloor } \indc_{A_l} \right] \\
    & = \E \left[ \E \left[ \prod_{l = 0}^{\lfloor NT \rfloor } \indc_{A_l}  \bigg| \mathcal{F}_{\mathcal{R}, \lfloor NT \rfloor ,y}\right]\right] \\
    & =  \E \left[\left(  \prod_{l = 0}^{\lfloor NT \rfloor - 1}  \indc_{A_l} \right) \times \E \left[ \indc_{A_{\lfloor NT \rfloor }}  | \mathcal{F}_{\mathcal{R}, \lfloor NT \rfloor ,y} \right]\right] \\
    & \leq \left( 1- \frac{1}{T^{\sfrac{9}{10}}} \right)\P \left[ \bigcap_{l = 0}^{\lfloor NT \rfloor -1} A_l \right].
\end{align*}
We may then iterate the previous computation, noting that, for any $l \in \left\{0 , \ldots, \lfloor NT \rfloor-1 \right\}$, the events $(A_1 , \ldots, A_{l})$ are measurable with respect to the $\sigma$-algebra $\mathcal{F}_{\mathcal{R} , l +1 ,y}$. This leads to the upper bound, for~$T$ sufficiently large (depending only on $V$) so that $\lfloor NT \rfloor +1 \geq T (\ln T)/R_V^2 \geq T$,
\begin{equation*}
    \P \left[ \bigcap_{l = 0}^{ \lfloor N T \rfloor} A_l \right] \leq \left( 1- \frac{1}{T^{\sfrac{9}{10}}} \right)^{\lfloor NT \rfloor +1}
     \leq  \exp \left( - \frac{\lfloor NT \rfloor +1}{T^{\sfrac{9}{10}}} \right) 
     \leq \exp \left( - T^{\sfrac{1}{10}} \right).
\end{equation*}
We next select a time $T_G := T_G(d , V) < \infty$ and a constant $C_G := C_G(d , V) < \infty$ such that the following implication holds: for any $T \geq T_G$,
\begin{equation} \label{eq:8.570112}
    \sum_{e' \cap e \neq \emptyset} \left| \nabla \phi_L (t , e') \right| \leq \frac{(\ln T)^{\frac{1}{r-2}}}{C_G} \implies \sum_{e' \cap e \neq \emptyset} \left| \a (t , e') \right| \leq \frac{N}{2}.
\end{equation}
The identity $N := \ln T /R_V^2$ and Assumption~\ref{assumption1} ensure that the constant $C_G$ and the time $T_G$ exist and are finite.
We then define the interval $I_T$
\begin{equation*}
    I_T := \left[ - \frac{(\ln T)^{\frac{1}{r-2}}}{(16 \sqrt{2}d)C_G} , \frac{(\ln T)^{\frac{1}{r-2}}}{(16 \sqrt{2}d)C_G}  \right]
\end{equation*}
as well as the good event
\begin{equation*}
    G_T := \left\{ \mathcal{R} \in \Omega \, : \, \sup_{t \in [0,T]} \sum_{e' \cap e \neq \emptyset} \left| \nabla \phi_L (t , e')(\mathcal{R}) \right| \leq \frac{(\ln T)^{\frac{1}{r-2}}}{2 C_G} \right\} \bigcap \bigcap_{k=0}^{\lfloor NT \rfloor} \left\{  X_k(y) \in I_T    \right\}.
\end{equation*}
We first show that the probability of the event $G_T$ is close to $1$. Using Proposition~\ref{propositionsubdynamic}, that the law of the increments $\left\{ X_k(y) \, : \, 1 \leq k \leq \lfloor NT \rfloor \right\}$ is Gaussian of variance $1/N = R_V^2/\ln T$ and a union bound on the complement of the event $G_T$, we obtain
\begin{equation} \label{eq:18092811}
    \P \left[ G_T^c \right] \leq C T \exp \left( - c \left( \ln T \right)^{\frac{r}{r-2}} \right) + C NT \exp \left( - c \left( \ln T \right)^{\frac{r}{r-2}} \right) \leq  C \exp \left( - c \left( \ln T \right)^{\frac{r}{r-2}} \right).
\end{equation}
We will now prove the inclusion of events
\begin{equation} \label{eq:18082811}
    \left\{ \mathcal{R} \in \Omega \, : \, \forall t \in [0 , T], \, \left| \nabla \phi_L(t , e) \right| \leq R_V  \right\} \subseteq \bigcap_{l = 0}^{ \lfloor N T \rfloor} A_l \cup G_T^c.
\end{equation}
Proposition~\ref{upperboundfluctuation} is then obtained by combining~\eqref{eq:11492811},~\eqref{eq:18092811},~\eqref{eq:18082811} and a union bound.
The rest of the argument is devoted to the proof of~\eqref{eq:18082811}. As mentioned in Section~\ref{eq:18111101}, we first observe from the definition of the Langevin dynamic that the function 
$$(X_l(y), \mathcal{R}_{l,y}) \mapsto \phi_L(t , x) (X_l(y), \mathcal{R}_{l,y})$$ 
is differentiable with respect to the increment $X_l(y)$, and its derivative can be computed in terms of a solution of a parabolic equation. To be more specific, let us introduce the notation
\begin{equation} \label{eq:08243011}
    w(t , y)(X_l(y), \mathcal{R}_{l,y}) := \frac{\partial \phi_L(t , y)}{\partial X_l(y)} (X_l(y), \mathcal{R}_{l,y}),
\end{equation}
and note that, for any $l \in \N$ and any $t \in \left[\frac{l}{N}, \frac{l+1}{N}\right]$,
\begin{equation*}
    B_t(y) = \sum_{k = 0}^{l-1} X_k (y) + N \left( t - \frac{k}{N}\right) X_{l}(y) + W_l\left(t - \frac{l}{N} ; y\right),
\end{equation*}
which implies the identity, for any $t \in \left[\frac{l}{N}, \frac{l+1}{N}\right]$,
\begin{equation*}
    d B_t(y) = N X_{l}(y) dt + d W_l\left(t - \frac{l}{N} ; y\right).
\end{equation*}
Substituting the previous identity in the definition~\eqref{def.Langevindynamics} of the Langevin dynamic and differentiating both sides of the identity by $X_l(y)$, we see that the function $w$ solves the parabolic equation
\begin{equation*}
     \left\{ \begin{aligned}
    \partial_t w (t , x) &= \nabla \cdot \a \nabla u (t , x) + \sqrt{2} N \indc_{\left[ \frac{l}{N}, \frac{l+1}{N} \right]}(t) \delta_{y}(x) &~\mbox{for}~& (t , x) \in [0 , \infty] \times \mathbb{T}_L, \\
    w (0 , x) &= 0 &~\mbox{for}~& x \in \mathbb{T}_L,
    \end{aligned} \right.
\end{equation*}
with the environment $\a(t , e') := V''(\nabla \phi_L(t , e'))$.
Applying Duhamel's principle with the definition of the heat kernel stated in~\eqref{eq:11012312}, we obtain the identity, for any $t \geq \frac{l}{N}$,
\begin{equation} \label{eq:8333011}
    w(t , x) = \sqrt{2} N \int_{\frac{l}{N}}^{\min \left( \frac{l+1}{N}, t \right)} \left( P_\a (t , x ; s , y) + \frac{1}{\left| \mathbb{T}_L \right| } \right) \, ds. 
\end{equation}
Additionally, the upper and lower bounds~\eqref{upperbound.PAmaxprinc} imply the following estimate on the gradient of the heat kernel, for any edge $e' \in E \left( \mathbb{T}_L \right)$ and any pair of times $(t , s) \in (0 , \infty)^2$ with $t \geq s$,
\begin{equation} \label{eq:13462811}
    \left| \nabla P_\a \left( t , e' ; s , y  \right) \right| \leq 1.
\end{equation}
A combination of the previous displays implies the following bound, for any $\mathcal{R} = (X_l(y), \mathcal{R}_{l,y}) \in \Omega$ and any $(t , e') \in (0 , \infty) \times E \left( \mathbb{T}_L \right)$,
\begin{equation} \label{eq:08250112}
    \left| \frac{\partial \nabla \phi_L(t , e')}{\partial X_l(y)} (X_l(y), \mathcal{R}_{l,y}) \right| \leq \sqrt{2}.
\end{equation}
We then fix a realization of the randomness $\mathcal{R} := (X_{l}(y), \mathcal{R}_{l , y}) \in \Omega$ and assume that $\mathcal{R} \in G_T$. We first claim that, for any increment $X \in I_T$,
\begin{equation} \label{eq:08260112}
   \sup_{t \in [0 , T]}\sum_{e'\cap e \neq \emptyset } \left| \nabla \phi_L \left( t, e' \right) \left( X , \mathcal{R}_{l , y} \right) \right| \leq \frac{(\ln T)^{\frac{1}{r-2}}}{C_G}.
\end{equation}
To prove~\eqref{eq:08260112}, we first use~\eqref{eq:08250112} and deduce that
\begin{equation*}
    \sup_{t \in [0 , T]}\sum_{e'\cap e \neq \emptyset} \left| \nabla \phi_L \left( t, e' \right)(X , \mathcal{R}_{l,y}) -  \nabla \phi_L \left( t, e' \right)(X_{l}(y) , \mathcal{R}_{l,y}) \right|  \leq \sqrt{2} (4 d) \left| X - X_{l}(y) \right|
     \leq  \frac{(\ln T)^{\frac{1}{r-2}}}{2C_G}.
\end{equation*}
By the assumption $(X_{l}(y), \mathcal{R}_{l , y}) \in G_T$, we have that
\begin{equation*}
     \sup_{t \in [0 , T]} \sum_{e'\cap e \neq \emptyset} \left| \nabla \phi_L \left( t, e' \right)(X_{l}(y) , \mathcal{R}_{l,y}) \right| \leq \frac{(\ln T)^{\frac{1}{r-2}}}{2 C_G}.
\end{equation*}
A combination of the two previous displays with the triangle inequality yields, for any $X \in I_T$,
\begin{equation*}
     \sup_{t \in [0 , T]} \sum_{e'\cap e \neq \emptyset} \left| \nabla \phi_L \left( t, e' \right)(X , \mathcal{R}_{l,y})\right| \leq  \frac{(\ln T)^{\frac{1}{r-2}}}{C_G}.
\end{equation*}
Using the definition of the constant $C_G$ and the implication~\eqref{eq:8.570112}, we have proved the following result: for any $T \geq T_G$, any $\mathcal{R} := (X_l(y), \mathcal{R}_{l,y}) \in G_T$, any increment $X \in I_T$, one has the upper bound
\begin{equation*}
    \sup_{t \in [0 , T]} \sum_{e'\cap e \neq \emptyset} \left| \a (t , e')(X , \mathcal{R}_{l,y}) \right| \leq \frac{N}{2}.
\end{equation*}
The previous upper bound is useful as it can be used to control the derivative in time of the heat kernel. Indeed, using the identity $\partial_t P_\a = \nabla \cdot \a \nabla P_\a$ together with the bound~\eqref{eq:13462811}, we obtain the estimate, for any pair of times $(s , t)  \in [0 , \infty)^2 $,
\begin{equation*}
    \left| \partial_t \nabla P_\a (t , e ; s , y) \right|
    \leq \sum_{e' \cap e \neq \emptyset}  \a(t , e') \left| \nabla P_\a (t , e' ; s , y) \right| \leq \sum_{e' \cap e \neq \emptyset} \a (t, e).
\end{equation*}
Combining the two previous displays with the identity $\nabla P_\a (s , e ; s , y) = 1$ (since $y$ is the second endpoint of $e$), we obtain that, for any $\mathcal{R} := (X_l(y), \mathcal{R}_{l,y}) \in G_T$ and any increment $X \in I_T$,
\begin{align*}
    \frac{\partial \nabla \phi_L\left( \frac{l+1}{N} , e\right)}{\partial X_l(y)} (X, \mathcal{R}_{l,y}) & = \sqrt{2} N \int_{\frac{l}{N}}^{\frac{l+1}{N}} \nabla  P_\a \left(\frac{l+1}{N}, e ; s , y \right) (X, \mathcal{R}_{l,y}) \, ds  \\
    &  \geq \sqrt{2} N \int_{\frac{k}{N}}^{\frac{k+1}{N}} 1 - \frac{N}{2} \left(\frac{k+1}{N} - s \right) \, ds \\
    & \geq \frac{3}{4}.
\end{align*}
This lower bound on the derivative of the gradient of the dynamic implies that, for any $(X_l(y) , \mathcal{R}_{l,y}) \in G_T$, the function
\begin{equation*}
     X \mapsto \nabla \phi_L\left( \frac{l+1}{N} , e \right)(X , \mathcal{R}_{l,y}) - \frac{3}{4} X ~~\mbox{is increasing on the interval}~~I_T.
\end{equation*}
This implies the following upper bound on the Lebesgue measure of the set $\mathcal{A}_l (\mathcal{R}_{l , y}) \cap I_T$,
\begin{equation*}
    \left| \mathcal{A}_l (\mathcal{R}_{l , y}) \cap I_T \right| \leq \frac{8}{3} R_V,
\end{equation*}
which then yields the estimate, for any $T$ sufficiently large (depending on $d$ and $V$) so that the computation~\eqref{eq:19132612} applies
\begin{equation*}
    \frac{1}{\sqrt{2 \pi N}}\int_{\mathcal{A}_l (\mathcal{R}_{l , y})} e^{- \frac{x^2}{2N}} \, dx  \leq 1 -  \frac{1}{\sqrt{2 \pi N}}\int_{I_T \setminus [-\frac{4}{3}R_V , \frac{4}{3} R_V]} e^{- \frac{x^2}{2N}} \, dx \leq 1 -  \frac{1}{T^{\sfrac{9}{10}}}.
\end{equation*}
From the definitions~\eqref{eq:11200112} and~\eqref{eq:11210112}, the previous inequality implies the identity, for any $l \in \{1 , \ldots , \lfloor NT \rfloor \}$,
\begin{equation*}
    G_T \cap A_l = G_T \cap \left\{ \mathcal{R} \in \Omega \, : \,  \left| \nabla \phi_L \left(\frac{l+1}{N} , e \right) \left(\mathcal{R} \right) \right| \leq R_V \right\}.
\end{equation*}
Taking the intersection over $l \in \{1 , \ldots , \lfloor NT \rfloor \}$ completes the proof of~\eqref{eq:18082811}.
\end{proof}

\section{On diagonal upper bound for the heat kernel} \label{section4ondiagupp}

In this section, we combine the result of Section~\ref{Section3} with the techniques developed by Mourrat and Otto~\cite{MO16} to obtain an on-diagonal upper bound for the heat kernel appearing in the Helffer-Sj\"{o}strand representation formula. The section is organized as follows. In Section~\ref{sec.moderatedevtandmaxquant}, we collect some preliminary definitions and results and state the main technical result of the section (pertaining to the decay rate of the $L^2$-norm of the heat kernel) in Theorem~\ref{thm3.13}. Section~\ref{sec:moderationenvt}, Section~\ref{sec:sec4.4} and Section~\ref{section:4.5} are devoted to the proof of Theorem~\ref{thm3.13} following the techniques of Mourrat and Otto~\cite{MO16}. The on-diagonal upper bound on the heat kernel is deduced from Theorem~\ref{thm3.13} in Section~\ref{section:ondiagonalheat}. Finally, Section~\ref{section:localization} completes the proof of Theorem~\ref{main.theorem} by combining the on-diagonal heat kernel estimate with the Helffer-Sj\"{o}strand representation formula.

\subsection{Preliminaries} \label{sec.moderatedevtandmaxquant}

We select two exponents $p , p' \in (d , \infty)$ depending only on the dimension $d$. These exponents will be used to define the moderated environment and apply the anchored Nash inequality, any specific values are admissible (for instance, one can choose $p = p' = d+1$). We let $\phi_L$ be the Langevin dynamic in the torus and let $\a := V''(\nabla \phi_L)$ be the environment appearing in the Helffer-Sj\"{o}strand representation formula. Using the stationarity of the gradient of the Langevin dynamic, Proposition~\ref{prop.stochintgradfield} and the growth condition assumed on the second derivative~$V''$, we know that all the moments of the random environment $\a$ are finite: for any $q \in [1 , \infty)$, and any $(t , e) \in (0 , \infty) \times E \left( \mathbb{T}_L \right)$,
\begin{equation} \label{ahasallmoments}
    \E \left[ \a(t , e)^q \right] < \infty. 
\end{equation}
Following the insight of Mourrat and Otto~\cite{MO16}, we introduce in this section the following moderated environment $w$. We first introduce the two functions
\begin{equation} \label{def.kt}
    k_t := \frac{\delta}{(1 + t)^{p+3}} \hspace{5mm} \mbox{and} \hspace{5mm}  K_t :=  k_t + \int_{t}^{\infty} s k_s \, ds,
\end{equation}
where $\delta := \delta(d)> 0$ is chosen sufficiently small so that, for any $t , s' \in (0 , \infty)$ with $s' \geq t$,
\begin{equation} \label{K*KsmallerthanK}
    \int_{t}^{s'} K_{s - t} K_{s' - s} \, ds \leq K_{s' - t} ~~ \mbox{and}~~ \int_{0}^{\infty} K_{s} \, ds \leq 1.
\end{equation}
Using the function $k$, we define the moderated environment $w$ as follows.

\begin{definition}[Moderated environment for the Langevin dynamic] \label{def:moderatedLangevin}
    We define the \emph{moderated environment} according to the formula, for any $(t , e) \in [0 , \infty) \times E\left( \mathbb{T}_L\right)$,
    \begin{equation} \label{eq:17580612}
    w(t , e)^2 = \int_t^{\infty} k_{s-t} \frac{\a(s, e) \wedge 1}{( s-t)^{-1} \sum_{e' \cap e \neq \emptyset}\int_t^s \a(s',e') \vee 1 \, ds'} \, ds.
\end{equation}
\end{definition}
Compared to the environment $\a$, the moderated environment $w$ satisfies the property that all the moments of $w$ and of $w^{-1}$ are finite, and we will prove in Proposition~\ref{propmoderatedenvtstochint} that, for any $q \in [1 , \infty)$, and any $(t , e) \in (0 , \infty) \times E \left( \mathbb{T}_L \right)$,
\begin{equation} \label{aandwallmoments}
    \E \left[ w(t , e)^q \right] +  \E \left[ w(t , e)^{-q} \right] < \infty.
\end{equation}
This result is proved in Proposition~\ref{propmoderatedenvtstochint}, and the proof builds upon the fluctuation estimate for the Langevin dynamic proved in Proposition~\ref{prop3.4}.

Various functionals of the environments $\a$ and $w$ will appear in the proof of the heat kernel estimate. They are collected below. Their formulae are technical, and we incite the reader to consult as a reference. They all possess the property they have finite moments of all order (see~\eqref{Mhaveallmoments} and~\eqref{Mscrhaveallmoments}). 

Before stating their definition, we recall the definitions of the exponents introduced in~\eqref{explambdakappanu} and~\eqref{explambdakappanubis}, let $\theta_c$ be the exponent given by~\eqref{def.thetac} of Proposition~\ref{theoremanchoredNash} (with the values of $p,p' \in (d , \infty)$ selected at the beginning of this section), and let $\alpha, \beta, \gamma$ be the exponents defined in~\eqref{def.alphabetagamma} (with $\theta = \theta_c$). For any time $t \geq 0$, we introduce the six random variables
\begin{equation*}
    \left\{ \begin{aligned}
       \mathcal{M}_{p'}(t) & := 1+  \left(1 + \| w (t , \cdot)\|_{\underline{L}^{\sigma_d}\left( \mathbb{T}_L \right)} \| w^{-1}(t , \cdot) \|_{\underline{L}^{\tau_d}\left( \mathbb{T}_L \right)} \right)^2 \sup_{r \in \{ 1 , \ldots, L\}} \| w^{-1} (t , \cdot)\|_{\underline{L}^{p'} (\Lambda_r)}^2,   \\
       \mathcal{M}_0(t) & := 1+ \sup_{r \in \{ 0 , \ldots , L \}} \| \a(t , \cdot)^{1/2} \|_{\underline{L}^{\sigma_d}\left( \Lambda_r \right)}^2 \left( 1 + \| w^{-1}(t , \cdot) \|_{\underline{L}^{\tau_d'}\left( \Lambda_r \right)}^2\right), \\
        \mathcal{M}_1(t) & := 1+ \int_t^\infty K_{s-t} \mathcal{M}_0(s)^{\frac{p }{2(1 - \theta_d)}} \, ds, \\
        \mathcal{M}_2(t)&  := 1+\sup_{x \in \mathbb{T}_L} \frac{ \sum_{e \ni x} \a(t,e)}{ |x|_*^{(p-2)/(p-1)}}, \\
        \mathcal{M}_3(t) & :=  1+\left( \int_{t}^\infty K_{s-t} \mathcal{M}_{p'}(s)^\frac{\alpha}{\beta} \, ds \right)^\beta,  \\
        \mathcal{M}_4(t) & := 1+ \sup_{s \in [t , t+1]} \left\| w(s , \cdot)^{-1}\right\|_{\underline{L}^d(\mathbb{T}_L)}^2.
    \end{aligned} \right.
\end{equation*} 
These six random variables appear at different stages of the proof. The term ``+1" is added to the definition to ensure that they are always larger than $1$. Their main key property is that they have finite moments of every order: for any $i \in \{ 0 , 1,2,3, 4 \}$, any $q \in [1 , \infty]$ and any $t \geq 0$,
\begin{equation} \label{Mhaveallmoments}
    \E \left[ \mathcal{M}_{p'}(t)^q \right] + \E \left[ \mathcal{M}_i(t)^q \right] < \infty.
\end{equation}
The proof of~\eqref{Mhaveallmoments} is a consequence of the bounds~\eqref{ahasallmoments} and~\eqref{aandwallmoments}, the Jensen inequality and the $L^p$-maximal inequality stated in Proposition~\ref{propmaximalineq}. Building upon these definitions, we consider the maximal functions
\begin{equation} \label{eq:20382112}
    \left\{ \begin{aligned}
    \mathscr{M}_1 & :=   \left( \sup_{t \geq 1} \frac{1}{t} \int_0^t \mathcal{M}_1(t)^{\frac 2p} \, dt \right)^{\frac p2} ,  \\
    \mathscr{M}_2 & := \left(\sup_{t \geq 1} \frac{1}{t} \int_0^t \mathcal{M}_2(t) \, dt \right)^{p-1}, \\
    \mathscr{M}_3 & :=  \left( \inf_{t \geq 1} \frac{1}{t} \int_0^t \mathcal{M}_3^{-\frac{1}{\alpha}}(s) \, ds \right)^{-\frac{\alpha}{\gamma}}, \\
     \mathscr{M}_4 & :=  \left( \inf_{t \geq 1} \frac{1}{t} \int_0^t \mathcal{M}_4^{-1}(s) \, ds \right)^{-1}.
    \end{aligned} \right.
\end{equation}
From the bound~\eqref{Mhaveallmoments} and the maximal inequality (with respect to the time variable) stated in Proposition~\ref{propmaximalineq} and the Jensen inequality, we know that all the moments of the random variables listed in~\eqref{eq:20382112} are finite, i.e., for any $i \in \{ 1 , 2 , 3,4\}$ and any $q \in [1 , \infty)$,
\begin{equation} \label{Mscrhaveallmoments}
    \E \left[ \mathscr{M}_i^q \right] < \infty.
\end{equation}
Finally, building on these definitions, we may define the random variables $\mathscr{M}$ and $\mathscr{M}'$ appearing in the definition of Theorem~\ref{thm3.13} above according to the formulae
\begin{equation*}
    \mathscr{M}  := \left( (\mathscr{M}_{1} + \mathscr{M}_{2}) \mathscr{M}_{3} \right)^{\frac{\gamma}{1 - \alpha - \gamma}} ~~\mbox{and} ~~
    \mathscr{M}' := \mathscr{M}_3^{\frac{2\gamma}{(d \beta + p \gamma)} + \frac{\gamma}{\alpha}} \mathscr{M}_4.
\end{equation*}
The inequality~\eqref{Mscrhaveallmoments} implies that all the moments of $\mathscr{M}$ and $\mathscr{M}'$ are finite. The main theorem of this section investigates the decay of the $L^2(\mathbb{T}_L)$-norm of the heat kernel. It can be compared to~\cite[Theorem 4.2]{MO16}

\begin{theorem}[Energy upper bound for dynamic environment] \label{thm3.13}
There exists a constant $C := C(d) < \infty$ such that, for any $t \geq 1$,
\begin{equation*}
    \sum_{x \in \mathbb{T}_L} P_\a \left( t , x \right)^2 \leq \frac{C \mathscr{M}}{(1+t)^{\frac d2}} \exp \left( - \frac{t}{C \mathscr{M}' L^2} \right).
\end{equation*}
\end{theorem}

\subsection{Moderation of the environment} \label{sec:moderationenvt}

In this section, we adapt the arguments of Mourrat and Otto~\cite[Proposition 4.6]{MO16} to environments which are not bounded from above. Using the terminology introduced in~\cite[Definition 3.1]{MO16}, we show that the environment $\a$ is $(w , C K)$-moderate. The proof of Proposition~\ref{prop4.3} is a notational modification of~\cite{MO16} and is written below for completeness.

\begin{proposition}[$(w, C K)$-moderation] \label{prop4.3}
There exists a constant $C := C(d) > 0$ such that, for every $t \geq 0$ and every solution $u : (0 , \infty) \times \mathbb{T}_L \to \R$ of the parabolic equation\begin{equation*}
    \partial_t u - \nabla \cdot \a \nabla u = 0 ~~~\mbox{in} ~~~ (0 , \infty) \times \mathbb{T}_L,
\end{equation*}
one has the inequality, for any edge $e \in E \left(\mathbb{T}_L\right)$,
\begin{equation} \label{eq:11140912}
    w(t , e)^2 (\nabla u(t , e))^2 \leq C \sum_{e' \cap e \neq \emptyset}  \int_t^\infty K_{t - s} \a(s , e') (\nabla u(s , e'))^2 \, ds.
\end{equation}
\end{proposition}

\begin{proof}
Following the proof of~\cite[Proposition 4.6]{MO16}, we fix an edge $e \in E\left( \mathbb{T}_L \right)$ and first estimate
\begin{align} \label{eq:16110612}
w(t , e)^2 (\nabla u(t , e))^2 & =  \int_t^{\infty} k_{s-t} \frac{\a(s,e) \wedge 1}{(s - t)^{-1} \sum_{e' \cap e \neq \emptyset} \int_t^s \a(s',e') \vee 1 \, ds'} (\nabla u(t , e))^2 \, ds \\
                        & \leq 2  \int_t^{\infty} k_{s-t} \frac{\a(s,e) \wedge 1}{(s - t)^{-1} \sum_{e' \cap e \neq \emptyset} \int_t^s \a(s',e') \vee 1 \, ds'} (\nabla u(s , e))^2 \, ds \notag \\
                        & \quad + 2 \int_t^{\infty} k_{s-t} \frac{\a(s,e) \wedge 1}{(s - t)^{-1} \sum_{e' \cap e \neq \emptyset} \int_t^s \a(s',e') \vee 1 \, ds'} (\nabla u(s , e) - \nabla u(t , e) )^2 \, dt. \notag
\end{align}
The first term in the right-hand side can be estimated as follows
\begin{equation} \label{eq:16110312}
\int_t^{\infty} k_{s-t} \frac{\a(s,e) \wedge 1}{(t - s)^{-1} \sum_{e' \cap e \neq \emptyset} \int_t^s \a(s',e') \vee 1 \, ds'} (\nabla u(s , e))^2 \, ds \\ \leq  \int_t^{\infty} k_{s-t} \a(s,e) (\nabla u(s , e))^2 \, ds.
\end{equation}
We next use the identity $\partial_t u = \nabla \cdot \a \nabla u$ and denote by $x$ and $y$ the two endpoints of $e$. We then write
\begin{align*}
    (\nabla u(s , e) - \nabla u(t , e) )^2 & \leq 2 ( u(s , x) - u(t , x) )^2 + 2 ( u(s , y) - u(t , y) )^2  \\
    & \leq 2 \left(  \int_t^s \nabla \cdot \a \nabla u(s' , x)\, ds' \right)^2 + 2 \left(  \int_t^s \nabla \cdot \a \nabla u(s' , y)\, ds' \right)^2.
\end{align*}
We next observe that, by the Cauchy-Schwarz inequality,
\begin{align*}
    \left( \int_t^s \nabla \cdot \a \nabla u(s' , x) \, ds' \right)^2 & \leq C \sum_{e' \ni x} \left( \int_t^s \left| \a(s' , e') \nabla u(s' , e') \right| \, ds' \right)^2 \\
    & \leq C \sum_{e' \ni x} \left( \int_t^s \a(s' , e') \, ds' \right) \left( \int_t^s \a(s' , e') (\nabla u(s' , e'))^2  \, ds' \right).
\end{align*}
A combination of the two previous displays yields
\begin{equation*}
    (\nabla u(s , e) - \nabla u(t , e) )^2  \leq C \sum_{e' \cap e \neq \emptyset} \left( \int_t^s \a(s' , e') \, ds' \right) \left( \int_t^s \a(s', e') (\nabla u(s' , e') )^2 \, ds' \right).
\end{equation*}
We thus obtain
\begin{equation*}
    \frac{\a(s,e) \wedge 1}{(s - t)^{-1} \sum_{e' \cap e \neq \emptyset} \int_t^s \a(s',e') \vee 1 \, ds'} (\nabla u(s , e) - \nabla u(t , e) )^2 \leq C (s - t) \sum_{e' \cap e \neq \emptyset} \int_t^s \a(s', e') (\nabla u(s' , e') )^2 \, ds'.
\end{equation*}
Combining the previous estimate with~\eqref{eq:16110612} and~\eqref{eq:16110312}, we deduce that
\begin{align*}
    w(t , e)^2 (\nabla u(t , e))^2 &
    \leq C \int_t^{\infty} k_{s-t} \a(s,e) (\nabla u(s , e))^2 \, ds \\
    & \qquad + C \sum_{e' \cap e \neq \emptyset} \int_{t}^\infty k_{s-t} (s - t) \int_t^s \a(s', e') (\nabla u(s' , e') )^2 \, ds' \\
    &  \leq C  \sum_{e' \cap e \neq \emptyset}\int_t^\infty K_{s - t} \a(s , e') (\nabla u(s , e'))^2 \, ds.
\end{align*}
The proof of Proposition~\ref{prop4.3} is complete.
\end{proof}

\subsection{Stochastic integrability for the moderated environment}

In this section, we establish stochastic integrability estimates for the moderated environment $w$, and prove that all the moments of $w$ and $w^{-1}$ are finite.

\begin{proposition}[Stochastic integrability for the moderated environment] \label{propmoderatedenvtstochint}
There exist two constants $c := c(d , V) > 0$ and $C := C(d , V) < \infty$ such that, for any $T \geq 1$, any time $t \geq 0$ and any edge $e \in E \left( \mathbb{T}_L \right)$,
\begin{equation} \label{eq:18500612}
    \P \left[ w(t , e) \leq \frac{1}{T} \right] \leq C \exp \left( - c (\ln T)^{\frac{r}{r-2}} \right)
\end{equation}
and
\begin{equation} \label{eq:18510612}
    \P \left[ w(t , e) \geq T \right] \leq C \exp \left( - c T^{\frac{2r}{r-2}} \right).
\end{equation}
\end{proposition}

\begin{remark}
Proposition~\ref{propmoderatedenvtstochint} implies that, for any exponent $q > 0$, any time $t \geq 0$ and any edge $e \in E \left( \mathbb{T}_L \right)$,
\begin{equation*}
    \E \left[ w(t , e)^q \right] + \E \left[ w(t , e)^{-q} \right] < \infty.
\end{equation*}
\end{remark}

\begin{proof}
We first prove~\eqref{eq:18510612}. By the stationarity of the gradient of the Langevin dynamic $\phi_L$, it is sufficient to prove the result for $t= 0$. We first prove the following inclusion of events: there exists $c := c(d , V) > 0$ such that, for any $T \geq 1$,
\begin{multline} \label{eq:11030712}
    \left\{ w(0 , e)^2 \leq \frac{c}{T^{p+5}} \right\} \subseteq \left\{ \sup_{t \in [0,T]}  \left| \nabla \phi_L\left( t, e \right) \right| \leq  R_V \right\} \\ \bigcup \left\{ \sup_{t \in [0,T]}  V'' \left( \nabla \phi_L\left( t , e \right)\right) + \sum_{e' \cap e \neq \emptyset} \left| V' \left( \nabla \phi_L\left( t , e' \right) \right) \right|  \geq \frac{T}{2} \right\} \\ \bigcup \left\{ \sup_{\substack{t , t' \in [0,T]\\ |t - t'| \leq \frac{1}{T}}} \left| \nabla B_{t'} \left( e \right) - \nabla B_{t} \left( e \right) \right| \geq \frac{1}{2} \right\}.
\end{multline}
The inclusion~\eqref{eq:11030712} states that, in order for $w(0 ,e )$ to be small, the dynamic $\nabla \phi_L(\cdot , e)$ has to stay in the interval $[-R_V , R_V]$ for a long time (this behavior is ruled out by Proposition~\ref{prop3.4}), or must behave very irregularly, this condition is represented by the second and third events in the right-hand side of~\eqref{eq:11030712}, and can only happen with small probability.

We first prove~\eqref{eq:11030712}. To this end, we will prove the following implication: there exists $c := c(d , V) > 0$ such that, for any $T \geq 1$,
\begin{multline} \label{eq:17260712}
    \sup_{t \in [0,T]}  \left| \nabla \phi_L\left( t, e \right) \right| \geq  R_V, ~ \sup_{t \in [0,T]} V'' \left( \nabla \phi_L\left( t , e \right)\right)  +  \sum_{e' \cap e \neq \emptyset} \left| V' \left( \nabla \phi_L\left( t , e' \right) \right) \right| \leq \frac{T}{2} \\ ~ \mbox{and} ~ \sup_{\substack{t , t' \in [0,T]\\ |t - t'| \leq \frac{1}{T}}} \left| \nabla B_{t'} \left( e \right) - \nabla B_{t} \left( e \right) \right| \leq \frac12 \implies w(0 , e)^2 \geq \frac{c}{T^{p+5}}.
\end{multline}
We assume that the event in the left-hand side is satisfied and let $t \in [0 , T]$ be such that $\left| \nabla \phi_L\left( t, e \right) \right| \geq  R_V$. Using the definition of the Langevin dynamic~\eqref{def.Langevindynamics}, we see that, for any time $s \in \left[t - \frac{1}{2T} , t + \frac{1}{2T} \right]$,
\begin{align*}
    \left| \nabla \phi_L(s , e) - \nabla \phi_L(t , e) \right| &
    \leq \left| \int_{t}^s \nabla \left( \nabla \cdot V'(\nabla \phi_L) \right)(s',e)  \, ds' \right| + \left| \nabla B_t(e) - \nabla B_s(e)\right| \\
    & \leq \int_{t - \frac{1}{2T}}^{t + \frac{1}{2T}} \sum_{e' \cap e \neq \emptyset}\left| V' \left( \nabla \phi_L\left( s' , e' \right) \right) \right| \, ds' + \frac12 \\
    & \leq 1.
\end{align*}
Using the assumption $R_V \geq 2$ which follows from its definition~\eqref{def.RV}, we deduce that, for any $s \in \left[t - \frac{1}{2T} , t + \frac{1}{2T} \right]$, $\left| \nabla \phi_L(s , e) \right| \geq \frac{R_V}{2}$. This implies, for any  $s \in \left[t - \frac{1}{2T} , t + \frac{1}{2T} \right]$,
\begin{equation*}
    \a(s , e) = V''(\nabla \phi_L(s , e)) \geq 1.
\end{equation*}
The left-hand side of~\eqref{eq:17260712} yields the upper bound, for any $s \in [0,T]$,
\begin{equation*}
     \a(s , e) \leq \frac{T}{2}.
\end{equation*}
A combination of the two previous displays with the definition of $w$ stated in~\eqref{eq:17580612} and the definition of~$k$ stated in~\eqref{def.kt} implies, for any $T \geq 1$,
\begin{align*}
    w(0, e)^2 & = \int_0^{\infty} k_{s} \frac{\a(s, e) \wedge 1}{s^{-1} \sum_{e' \cap e \neq \emptyset}\int_0^s \a(s',e') \vee 1 \, ds'} \, ds \\
    & \geq \int_{t - \frac{1}{2T}}^{t + \frac{1}{2T}} k_{s} \frac{\a(s, e) \wedge 1}{s^{-1} \sum_{e' \cap e \neq \emptyset}\int_0^s \a(s',e') \vee 1 \, ds'} \, ds \\
    & \geq \frac{2}{T}  \int_{t - \frac{1}{2T}}^{t + \frac{1}{2T}} k_{s} \\
    & \geq \frac{c}{T^{p+5}}.
\end{align*}
The proof of~\eqref{eq:17260712}, and thus of~\eqref{eq:11030712} is complete. We next estimate the probabilities of the three events in the right-hand side of~\eqref{eq:11030712}. For the first one, we use Proposition~\ref{prop3.4} and write, for any $T \geq 1$,
\begin{equation*}
    \P \left(  \sup_{t \in [0,T]}  \left| \nabla \phi_L\left( t, e \right) \right| \leq  R_V \right) \leq C \exp \left( - c \left( \ln T \right)^{\frac{r}{r-2}} \right).
\end{equation*}
For the second term, we use Assumption~\ref{assumption1} on the potential $V$ and Proposition~\ref{propositionsubdynamic} to obtain that, for any $T \geq 1$,
\begin{align*}
    \P \left( \sup_{t \in [0,T]} \sum_{e' \cap e \neq \emptyset} \left| V' \left( \nabla \phi_L\left( t , e' \right) \right) \right|  + V'' \left( \nabla \phi_L\left( t , e \right)\right)\geq \frac{T}{2}  \right) & \leq C T \exp \left( - c  T^{\frac{r}{r-1}}\right) + C T \exp \left( - c  T^{\frac{r}{r-2}}\right) \\ &
    \leq C \exp \left( - c  T^{\frac{r}{r-1}}\right).
\end{align*}
For the third term, we note that, for any $T \geq 1$,
\begin{align*}
    \P \left( \sup_{\substack{t , t' \in [0,T]\\ |t - t'| \leq \frac{1}{T}}} \left| \nabla B_{t'} \left( e \right) - \nabla B_{t} \left( e \right) \right| \geq \frac{1}{2} \right) & 
    \leq  \sum_{l = 0}^{\lceil T^2 \rceil} \P \left( \sup_{t \in \left[ \frac{l-1}{T} , \frac{l+1}{T} \right] } \left| \nabla B_t(e) - \nabla B_{\frac{l}{T}}(e) \right| \geq \frac{1}{4} \right) \\
    & \leq (T^2+2) \P \left( \sup_{t \in \left[ 0 , \frac{2}{T} \right] } \left| \nabla B_t(e) - \nabla B_{\frac{1}{T}}(e) \right| \geq \frac{1}{4} \right) \\
    & \leq (T^2+2) \P \left( \sup_{t \in \left[ 0 , 2 \right] } \left| \nabla B_t(e) - \nabla B_{1}(e) \right| \geq \frac{\sqrt{T}}{4} \right) \\
    & \leq C (T^2+2) \exp \left( - c T \right).
\end{align*}
Combining the three previous displays with~\eqref{eq:11030712} yields, for any $T \geq 1$,
\begin{align}
    \P \left( w(0 , e) \leq \frac{c}{T^{p+5}} \right) & \leq  C \exp \left( - c \left( \ln T \right)^{\frac{r}{r-2}} \right) + C \exp \left( - c  T^{\frac{r}{r-1}}\right) + C (T^2+1) \exp \left( - c T \right) \\
    & \leq C \exp \left( - c \left( \ln T \right)^{\frac{r}{r-2}} \right). \notag
\end{align}
This implies~\eqref{eq:18500612}. To prove~\eqref{eq:18510612}, we note that, using the stationarity of the gradient of the Langevin dynamic and Assumption~\ref{assumption1}, for any $t \geq 0$ and any $e \in E \left( \mathbb{T}_L\right)$,
\begin{equation*}
    \P \left[ \a(t , e) \geq T \right] \leq C \exp \left( - c T^{\frac{r}{r-2}} \right).
\end{equation*}
We next observe that, from the definitions~\eqref{def.kt} and~\eqref{eq:17580612}, 
\begin{equation*}
    w(t , e)^2 \leq \int_t^{\infty} k_{s-t}\a(s, e)\, ds ~~\mbox{and}~~ \int_t^{\infty} k_{s-t}\, ds = \int_0^{\infty} k_{s}\, ds < \infty.
\end{equation*}
Combining the two previous displays with Lemma~\ref{lemmastochint} (and $f(t) = k_t / \int_0^\infty k_s \, ds$) completes the proof of~\eqref{eq:18510612}.
\end{proof}

\subsection{Anchored Nash estimate in the torus} \label{sec:sec4.4}

In this section, we prove a finite-volume version of the anchored Nash estimate of~\cite[Theorem 2.1]{MO16} (see Theorem~\ref{theoremanchoredNash}). The result is stated below and we emphasize that it only requires a minor adaptation of the proof of~\cite[Theorem 2.1]{MO16}.

\begin{proposition}[Anchored Nash estimate on the torus]\label{theoremanchoredNashtorus}
There exists $C := C(d) < \infty$ such that, for any function $u : \mathbb{T}_L \to \R$ satisfying $\sum_{x \in \mathbb{T}_L} u(x) = 0$ and any time $t \geq 0$,
\begin{equation*}
    \left\| u \right\|_{L^2 \left( \mathbb{T}_L \right)} \leq C \left( \mathcal{M}_{p'}(t)^{\frac 12}  \left\| w(t , \cdot) \nabla u \right\|_{L^2 \left( \mathbb{T}_L \right)} \right)^{\alpha} \left\| u \right\|_{L^1 \left( \mathbb{T}_L \right)}^\beta \| |x|^{p/2}_* u \|_{L^2 \left( \mathbb{T}_L \right)}^\gamma,
\end{equation*}
\end{proposition}

\begin{proof}
In this proof, we fix a time $t \geq 0$, identify the torus $\mathbb{T}_L$ with the box $\Lambda_L$ and extend the functions $u$ and the moderated environments $w$ periodically to the lattice $\Zd$. The periodicity of $w$ implies that there exists a constant $c := c(d) > 0$ such that, using the notation~\eqref{def.euclideanballs} for the maximal function,
\begin{equation} \label{eq:11460812}
    c M (w^{-p'}(t , \cdot))^{\frac1{p'}} \leq \sup_{r \in \{ 1 , \ldots, L\}} \| w^{-1} (t , \cdot)\|_{\underline{L}^{p'} (\Lambda_r)} \leq M (w^{-p'}(t , \cdot))^{\frac1{p'}}.
\end{equation}
We then let $\eta : \Zd \to \R$ be a cutoff function satisfying
\begin{equation} \label{def.eta1154}
    \indc_{\Lambda_L} \leq \eta \leq \indc_{\Lambda_{\frac{4}{3}L}} ~~\mbox{and} ~~ \left| \nabla \eta \right| \leq \frac{C}{L}.
\end{equation}
We next apply Theorem~\ref{theoremanchoredNash} with the finitely supported function $\eta u : \Zd \to \R$ and use the lower bound~\eqref{eq:11460812} to deduce that
\begin{equation} \label{eq:17010812}
    \left\| \eta u \right\|_{L^2 \left( \Zd \right)} \leq C \left( \left( \sup_{r \in \{ 1 , \ldots, L\}} \| w^{-1}(t , \cdot) \|_{\underline{L}^{p'} (\Lambda_r)} \right)\left\| w(t , \cdot) \nabla ( \eta u ) \right\|_{L^2 \left( \Zd \right)} \right)^{\alpha} \left\| \eta u \right\|_{L^1 \left( \Zd \right)}^\beta \| |x|^{p/2}_* \eta u \|_{L^2 \left( \Zd \right)}^\gamma.
\end{equation}
Using the periodicity of the function $u$ and the definition of the cutoff function $\eta$, we have the upper bounds, for some $C := C(d) < \infty$,
\begin{equation} \label{eq:170208122}
    \left\| u \right\|_{L^2 \left( \mathbb{T}_L \right)}  \leq \left\| \eta u \right\|_{L^2 \left( \Zd \right)}, ~~ \left\| \eta u \right\|_{L^1(\Zd)} \leq C \left\| u \right\|_{L^1(\mathbb{T}_L)}, ~\mbox{and} ~  \| |x|^{p/2}_* \eta u \|_{L^2 \left( \Zd \right)} \leq  C \| |x|^{p/2}_* u \|_{L^2 \left( \mathbb{T}_L \right)}.
\end{equation}
So that there only remains to treat the term $\left\| w \nabla ( \eta u ) \right\|_{L^2 \left( \Zd \right)}$. Expanding the discrete gradient and using the properties of the cutoff function $\eta$ stated in~\eqref{def.eta1154}, we obtain
\begin{equation} \label{eq:16180812}
    \sum_{e \in E \left( \Zd \right)} w(t,e)^2 \nabla ( \eta u )(e)^2 \leq C \sum_{e \in E \left( \Zd \right)} (w(t,e) \eta(e) \nabla u(e))^2 + \frac{C}{L^2} \sum_{e \in E ( \Lambda_{\frac{4}{3}L} )} (w(t,e) u(e))^2.
\end{equation}
where we recall the notation $\eta (e) = (\eta(x) + \eta(y))/2$ and $u(e) = (u(x) + u(y))/2$ for $e = ( x , y) \in E\left( \Zd \right)$.
Using the periodicity of the functions $u$ and $w$, we may rewrite the previous inequality as follows
\begin{equation*}
    \sum_{e \in E \left( \mathbb{T}_L\right)} w(t,e)^2 \nabla ( \eta u )(e)^2 \leq C \sum_{e \in E \left( \mathbb{T}_L\right)} (w(t,e) \nabla u(e))^2 + \frac{C}{L^2} \sum_{e \in E \left( \mathbb{T}_L\right)} (w(t,e) u(e))^2.
\end{equation*}
We then estimate the second term in the right-hand side. To this end, we will use the H\"{o}lder inequality and the Gagliardo-Nirenberg-Sobolev inequality. We recall the definitions of the four exponents $\lambda_d , \kappa_d, \sigma_d$ and $\tau_d$ introduced in~\eqref{explambdakappanu}, and apply first the H\"{o}lder inequality and then the Gagliardo-Nirenberg-Sobolev inequality (Proposition~\ref{propDGS} using that $\sum_{x \in \mathbb{T}_L} u(x) = 0$), then the H\"{o}lder inequality. We deduce that
\begin{align} \label{eq:16100812}
    \left( \frac{1}{\left|\mathbb{T}_L\right|}\sum_{e \in E \left( \mathbb{T}_L\right)} (w(t,e) u(e))^2\right)^{\frac 12} & \leq \left\| w(t , \cdot) \right\|_{\underline{L}^{\sigma_d}\left( \mathbb{T}_L \right)} \left\| u  \right\|_{\underline{L}^{\kappa_d} \left( \mathbb{T}_L \right)} \\
    & \leq CL\left\| w(t,\cdot) \right\|_{\underline{L}^{\sigma_d}\left( \mathbb{T}_L \right)} \left\| \nabla u  \right\|_{\underline{L}^{\lambda_d} \left( \mathbb{T}_L \right)} \notag \\
    & \leq CL \left\| w(t , \cdot) \right\|_{\underline{L}^{\sigma_d}\left( \mathbb{T}_L \right)} \left\| w^{-1}(t , \cdot) \right\|_{\underline{L}^{\tau_d}\left( \mathbb{T}_L \right)} \left\| w(t , \cdot) \nabla u  \right\|_{\underline{L}^{2} \left( \mathbb{T}_L \right)} . \notag 
\end{align}
Combining~\eqref{eq:16100812} with~\eqref{eq:16180812}, we deduce that
\begin{equation*}
    \left\| w(t , \cdot) \nabla ( \eta u ) \right\|_{L^2 \left( \Zd \right)}  \leq C (1 + \left\| w \right\|_{\underline{L}^{\sigma_d}\left( \mathbb{T}_L \right)} \| w^{-1} (t , \cdot)\|_{\underline{L}^{\tau_d}\left( \mathbb{T}_L \right)} ) \left\| w(t , \cdot) \nabla u  \right\|_{L^{2} \left( \mathbb{T}_L \right)}.
\end{equation*}
Combining the previous display with~\eqref{eq:17010812} and~\eqref{eq:170208122} completes the proof of Proposition~\ref{theoremanchoredNashtorus}.
\end{proof}

\subsection{Estimate on the $L^2$-norm of the heat kernel} \label{section:4.5}

Following~\cite[Proof of Theorem 3.2]{MO16}, we introduce the notation
\begin{equation*}
    \mathcal{E}_t := \sum_{x \in \mathbb{T}_L} P_\a (t , x)^2,~~ \mathcal{D}_t = \sum_{e \in E \left( \mathbb{T}_L \right)} \a(t , e) (\nabla P_\a(t , e))^2 ~~ \mbox{and}~~ \mathcal{N}_t := \sum_{x \in \mathbb{T}_L} | x |^{p}_* P_\a(t , x)^2,
\end{equation*}
as well as the moderated quantities
\begin{equation*}
    \bar{\mathcal{E}}_t := \int_t^\infty K_{s-t} \mathcal{E}_s \, ds, ~~ \bar{\mathcal{D}}_t := \int_t^\infty K_{s-t} \mathcal{D}_s \, ds ~~ \mbox{and}~~ \bar{\mathcal{N}}_t :=  \int_t^\infty K_{s-t} \mathcal{N}_s \, ds.
\end{equation*}
We note that the following identities hold
\begin{equation*}
    \partial_t \mathcal{E}_t = - 2\mathcal{D}_t, ~~ \partial_t \bar{\mathcal{E}}_t = - 2\bar{\mathcal{D}}_t, ~~\mbox{and}~~ \partial_t \bar{\mathcal{N}}_t = \int_t^\infty K_{s-t} \partial_s \mathcal{N}_s \, ds.
\end{equation*}
In particular the maps $\mathcal{E}$ and $\bar{\mathcal{E}}$ are decreasing, since $\mathcal{E}_0 = 1$, we have $\mathcal{E}_t \leq 1$ for any time $t \geq 0$.

\subsubsection{A differential inequality for the weighted $L^2$-norm of the heat kernel}

Following the proof of~\cite{MO16}, we will need to prove the following lemma which estimates the value of $\bar{\mathcal{N}}_0$ and the derivative $\partial_t \bar{\mathcal{N}}_t$. It closely follows~\cite[Proposition 3.3]{MO16} (written with the function $\mathcal{N}$ instead of $\bar{\mathcal{N}}$), which is itself based on~\cite[(81)]{GNO}.  We recall the notation for the maximal quantities $\mathcal{M}_0(t)$, $\mathcal{M}_1(t)$ and $\mathscr{M}_2$ introduced in Section~\ref{sec.moderatedevtandmaxquant}.

\begin{lemma} \label{lemma3.11}
There exists a constant $C := C(d) < \infty$, such that the following upper bounds hold
\begin{equation} \label{eq:20111012}
    \bar{\mathcal{N}}_0 \leq C \mathscr{M}_{2} ~~ \mbox{and}~~\partial_t \bar{\mathcal{N}}_t \leq C \mathcal{M}_1(t)^{\frac 2p} (\bar{\mathcal{N}_t})^{\frac{p-2}{p}} \mathcal{E}_t^{\frac 2p}.
\end{equation}
\end{lemma}

\begin{proof}
We first prove that the term $\mathcal{N}_t$ grows at most polynomially fast in the time $t$. Specifically, we will prove the upper bound, for any $t \geq 0$,
\begin{equation} \label{eq:15211612}
    \mathcal{N}_t \leq C \mathscr{M}_{2} ( 1 + t)^{p-1}.
\end{equation}
Using the definition of $K_t$  in~\eqref{def.kt} (which implies that it decays asymptotically like $t \mapsto t^{-p-1}$), and integrating the previous inequality, we deduce that
\begin{equation*}
    \bar{\mathcal{N}}_0 = \int_0^\infty K_t \mathcal{N}_t \, dt \leq  C \mathscr{M}_{2} \int_0^\infty (1 + t )^{-p-1} (1+t)^{p-1} \, dt  \leq C \mathscr{M}_{2}.
\end{equation*}
To prove~\eqref{eq:15211612}, we write, for $x \in \mathbb{T}_L$, $\rho(x) = |x|_*$. We first differentiate the function $\mathcal{N}_t$ and obtain
\begin{equation*}
    \frac{1}{2} \partial_t \mathcal{N}_t
     = - \sum_{e \in E \left( \mathbb{T}_L \right)} \nabla \left( \rho^p P_\a \right)(t ,e) \a(t , e) \nabla P_\a(t , e).
\end{equation*}
Expanding the discrete gradient, we see that
\begin{equation*}
    \nabla (\rho^p P_\a)(t , e) = \left(\nabla \rho^p(e) \right)  P_\a(t , e) + \rho^p(e) \nabla P_\a (t , e).
\end{equation*}
using that there exists a constant $C_0 := C_0(d) < \infty$ (as the exponent $p$ depends only on $d$) such that $\left| \nabla \rho^p(e) \right| \leq C_0 \rho^{p-1}(e)$, we deduce that
\begin{equation*}
    \frac{1}{2} \partial_t \mathcal{N}_t \leq - \sum_{e \in E \left( \mathbb{T}_L \right)} \rho^p(e) \a(t , e) (\nabla P_\a(t,e))^2 + C_0 \rho^{p-1}(e) P_\a(t,e) \a(t , e) \left| \nabla P_\a(t,e)\right| .
\end{equation*}
The second term in the right-hand side can be estimated using Young's inequality
\begin{multline*}
    \sum_{e \in E \left( \mathbb{T}_L \right)} \rho^{p-1}(e) P_\a(t,e) \a(t , e) \left|\nabla P_\a(t,e) \right| \\ \leq \frac1{2C_0} \sum_{e \in E \left( \mathbb{T}_L \right)} \rho^{p}(e) \a(t , e) (\nabla P_\a(t,e))^2  + \frac{C_0}2 \sum_{e \in E \left( \mathbb{T}_L \right)} \rho^{p-2}(e) \a(t , e) P_\a(t,e)^2 .
\end{multline*}
By the H\"{o}lder inequality and using that $\mathcal{E}_t \leq 1$, we have that
\begin{align*}
    \sum_{e \in E \left( \mathbb{T}_L \right)} \rho^{p-2}(e) \a(t , e) P_\a(t,e)^2 & \leq \left( \sum_{e \in E \left( \mathbb{T}_L \right)} \rho^{p-1}(e) \a(t , e)^{\frac{p-1}{p-2}} P_\a(t,e)^2 \right)^{\frac{p-2}{p-1}} \left( \sum_{e \in E \left( \mathbb{T}_L \right)} P_\a(t,e)^2 \right)^{\frac{1}{p-1}} \\
    & \leq C \left( \sum_{e \in E \left( \mathbb{T}_L \right)} \rho^{p-1}(e) \a(t , e)^{\frac{p-1}{p-2}} P_\a(t,e)^2 \right)^{\frac{p-2}{p-1}}.
\end{align*}
Using the definition of the random variable $\mathcal{M}_2$, we obtain
\begin{align*}
    \sum_{e \in E \left( \mathbb{T}_L \right)} \rho^{p-2}(e) \a(t , e) P_\a(t,e)^2  & \leq \left( \sum_{e \in E \left( \mathbb{T}_L \right)} \rho^{p-1}(e) \a(t , e)^{\frac{p-1}{p-2}} P_\a(t,e)^2 \right)^{\frac{p-2}{p-1}} \\
    & \leq  C \mathcal{M}_2(t) \left( \sum_{e \in E \left( \mathbb{T}_L \right)} \rho^{p}(e) P_\a(t,e)^2 \right)^{\frac{p-2}{p-1}} \\
    & \leq C \mathcal{M}_2(t)  \mathcal{N}_t^{\frac{p-2}{p-1}} .
\end{align*}
Combining the few previous displays, we obtain
\begin{equation*}
    \partial_t \mathcal{N}_t \leq C \mathcal{M}_2(t)  \mathcal{N}_t^{\frac{p-2}{p-1}}.
\end{equation*}
Integrating the inequality and using $\mathcal{N}_0 =1$, we obtain
\begin{equation*}
\mathcal{N}_t \leq C \mathscr{M}_{2} (1+t)^{p-1}.
\end{equation*}
The proof of the first part of~\eqref{eq:20111012} is complete. We next prove the inequality on the derivative of the function $\bar{\mathcal{N}}_t$. To this end, we first compute (similarly as before)
\begin{align*}
    \frac{1}{2} \partial_t \bar{\mathcal{N}}_t & = \frac12 \int_{t}^\infty K_{s-t} \partial_s \mathcal{N}_s \, ds\\
    & = \frac12 \int_{t}^\infty K_{s-t} \partial_s \mathcal{N}_s \, ds \\
    & = - \int_{t}^\infty K_{s-t} \sum_{e \in E \left( \mathbb{T}_L \right)} \nabla \left( \rho^p P_\a \right)(s ,e) \a(s , e) \nabla P_\a(s , e) \, ds.
\end{align*}
Expanding the discrete divergence, we deduce that
\begin{multline} \label{eq:09531712}
    \frac{1}{2} \partial_t \bar{\mathcal{N}}_t \leq - \int_{t}^\infty K_{s-t} \sum_{e \in E \left( \mathbb{T}_L \right)}  \rho^p(e) \a(s, e) (\nabla P_\a(s , e))^2 \, ds \\ + C_0 \int_{t}^\infty K_{s-t} \sum_{e \in E \left( \mathbb{T}_L \right)} \rho^{p-1}(e) P_\a(s,e) \a(s, e) \left|\nabla P_\a(s , e) \right| \, ds . 
\end{multline}
The second term in the right-hand side can be estimated using Young's inequality
\begin{multline} \label{new.4.277}
    \sum_{e \in E \left( \mathbb{T}_L \right)} \rho^{p-1}(e) P_\a(s,e) \a(s, e) \left|\nabla P_\a(s , e) \right| \\ \leq \frac1{2C_0} \sum_{e \in E \left( \mathbb{T}_L \right)} \rho^{p}(e) \a(s , e) (\nabla P_\a(s , e))^2  + \frac{C_0}2 \sum_{e \in E \left( \mathbb{T}_L \right)} \rho^{p-2}(e) \a(s , e) P_\a (s , e)^2 .
\end{multline}
We estimate the second term in the right-hand side. We will use the same technique as in the proof of Proposition~\ref{theoremanchoredNashtorus}. We first split the sum into dyadic scales. To this end, for $n \in \N$, we denote by $A_n := \Lambda_{2^{n+1}} \setminus \Lambda_{2^{n}} $ the dyadic annulus and by $\ln_2$ the binary logarithm. We then write
\begin{equation} \label{eq:16350912}
    \sum_{e \in E \left( \mathbb{T}_L \right)} \rho^{p-2}(e) \a(s , e) (P_\a(s , e))^2  \leq C \sum_{n = 0}^{\lfloor \ln_2 L \rfloor} 2^{(p-2)n} \sum_{e \in E \left( A_n \right)} \a(s , e) P_\a (s , e)^2.
\end{equation}
For each integer $n \in \{ 0 , \ldots, \lfloor \ln_2 N \rfloor \}$, we apply the same computation as in~\eqref{eq:16100812} based on the H\"{o}lder inequality and the Gagliardo-Nirenberg-Sobolev inequality. We obtain, for any $\ep > 0$,
    \begin{align} \label{eq:16310912}
   \lefteqn{\left(\frac{1}{\left| A_n \right|}\sum_{e \in E \left( A_n \right)} \a(s , e) P_\a (s , e)^2 \right)^{\frac 12}} \qquad & \\ & \leq \| \a(s , \cdot)^{1/2} \|_{\underline{L}^{\sigma_d}\left( A_n \right)} \left\| P_\a(s , \cdot)  \right\|_{\underline{L}^{\kappa_d} \left( A_n \right)} \notag \\
    & \leq \| \a(s , \cdot)^{1/2} \|_{\underline{L}^{\sigma_d}\left( A_n \right)} \left( \ep 2^n \left\| \nabla P_\a(s , \cdot)  \right\|_{\underline{L}^{\lambda_d'} \left( A_n \right)}  + C \ep^{-\frac{\theta_d}{1-\theta_d}} \left\| P_\a(s,\cdot) \right\|_{\underline{L}^2(A_n)} \right) \notag \\
    & \leq \ep 2^n \| \a(s , \cdot)^{1/2} \|_{\underline{L}^{\sigma_d}\left( A_n \right)} \| w^{-1}(s , \cdot) \|_{\underline{L}^{\tau_d'}\left( A_n \right)} \left\|  w(s , \cdot) \nabla P_\a(s , \cdot ) \right\|_{\underline{L}^2(A_n)} \notag \\
    & \quad + C \ep^{-\frac{\theta_d}{1-\theta_d}} \| \a(s , \cdot)^{1/2}\|_{\underline{L}^{\sigma_d}\left( A_n \right)} \left\| P_\a(s,\cdot) \right\|_{\underline{L}^2(A_n)}  . \notag 
\end{align}
Using the definition of the maximal function $\mathcal{M}_0$, the inequality~\eqref{eq:16310912} can be rewritten as follows
\begin{equation*}
    \sum_{e \in E \left( A_n \right)} \a(s , e) P_\a (s , e)^2  \leq \mathcal{M}_0(s)  2^{2n} \ep^2 \left\|  w(s , \cdot) \nabla P_\a(s , \cdot ) \right\|_{L^2(A_n)}^2 + C \ep^{-\frac{2\theta_d}{1-\theta_d}} \mathcal{M}_0(s) \left\| P_\a(s,\cdot) \right\|_{L^2(A_n)}^2 .
\end{equation*}
Using Proposition~\ref{prop4.3}, we deduce that
\begin{align*}
   \sum_{e \in E \left( A_n \right)} \a(s , e) P_\a (s , e)^2 & \leq C \mathcal{M}_0(s)  2^{2n} \ep^2 \int_s^\infty K_{s' - s} \sum_{e \in E \left( A_n \right)} \a(s' , e) (\nabla P_\a(s' , e))^2  \, ds'\\
   & \quad +  C \ep^{-\frac{2\theta_d}{1-\theta_d}} \mathcal{M}_0(s) \sum_{x \in  A_n } P_\a(s , x)^2 .
\end{align*}
Using that $ c 2^{n} \leq \rho(x) \leq C 2^{n+1}$ for any $x \in A_n$ and summing over the integers $n \in \{ 0, \ldots , \lfloor \ln_2 L \rfloor \}$, we deduce that
\begin{align*}
    \sum_{e \in E \left( \mathbb{T}_L \right)} \rho^{p-2}(e) \a(s , e) P_\a(s , e)^2  & \leq  C \mathcal{M}_0(s)  \ep^2 \int_s^\infty K_{s' - s} \sum_{ e \in  E \left( \mathbb{T}_L \right)} \rho(e)^p \a(s' , e) (\nabla P_\a(s' , e))^2 \, ds \\ & \quad +  C \ep^{-\frac{2\theta_d}{1-\theta_d}} \mathcal{M}_0(s) \sum_{x \in \mathbb{T}_L}\rho^{p-2}(x) P_\a(s , x)^2 .
\end{align*}
We next choose $\ep = 1/(C_0 \sqrt{C\mathcal{M}_0(s)})$ where $C_0$ is the constant appearing in~\eqref{eq:09531712} and $C$ is the one appearing in the previous display. We obtain
\begin{multline*}
    \sum_{e \in E \left( \mathbb{T}_L \right)} \rho^{p-2}(e) \a(s , e) P_\a(s , e)^2  \\ \leq  \frac{1}{C_0^2} \int_s^\infty K_{s' - s} \sum_{ e \in  E \left( \mathbb{T}_L \right)} \rho(e)^p \a(s' , e) (\nabla P_\a(s' , e))^2 \, ds + C \mathcal{M}_0(t)^{\frac{1}{1-\theta_d}} \sum_{x \in \mathbb{T}_L}\rho^{p-2}(x) (P_\a(s , x))^2 .
\end{multline*}
Multiplying the previous inequality by the weight function $K$ and integrating over the time variable yields, for any $t \geq 0$,
\begin{multline} \label{eq:90451712}
    \int_{t}^\infty K_{s-t} \sum_{e \in E \left( \mathbb{T}_L \right)} \rho^{p-2}(e) \a(s , e) P_\a(s , e)^2 \, ds \\ \leq \frac{1}{C_0^2} \int_{t}^\infty K_{s-t} \int_s^\infty K_{s' - s} \sum_{ e \in  E \left( \mathbb{T}_L \right)} \rho(e)^p \a(s' , e) (\nabla P_\a(s' , e))^2 \, ds \, ds' \\ + C   \int_{t}^\infty K_{s-t} \mathcal{M}_0(s)^{\frac{1}{1-\theta_d}} \sum_{x \in \mathbb{T}_L}\rho^{p-2}(x) P_\a(s  , x)^2 \, ds.
\end{multline}
The first term in the right-hand side can be simplified using the inequality~\eqref{K*KsmallerthanK} as follows
\begin{multline*}
    \int_{t}^\infty K_{s-t} \int_s^\infty K_{s' - s} \sum_{ e \in  E \left( \mathbb{T}_L \right)} \rho(e)^p \a(s' , e') (\nabla P_\a(s' , e))^2 \, ds \, ds' \\ \leq  \int_{t}^\infty K_{s-t} \sum_{ e \in  E \left( \mathbb{T}_L \right)} \rho(e)^p \a(s , e) (\nabla P_\a(s , e))^2 \, ds.
\end{multline*}
The second term in the right-hand side of~\eqref{eq:90451712} can be estimated using the H\"{o}lder inequality as follows
\begin{multline*}
    \int_{t}^\infty K_{s-t} \mathcal{M}_0(s)^{\frac{1}{1-\theta_d}} \sum_{x \in \mathbb{T}_L}\rho^{p-2}(x) P_\a(s  , x)^2 \, ds \\ \leq \left( \int_{t}^\infty K_{s-t} \sum_{x \in \mathbb{T}_L}\rho^{p}(x) P_\a(s  , x)^2 \, ds \right)^{\frac{p-2}{p}} \left( \int_{t}^\infty K_{s-t} \mathcal{M}_0(s)^{\frac{p}{2(1 - \theta_d)}} \sum_{x \in \mathbb{T}_L} P_\a(s  , x)^2 \, ds \right)^{\frac{2}{p}}.
\end{multline*}
Using that the energy $\mathcal{E}_t = \sum_{x \in \mathbb{T}_L} P_\a(t  , x)^2$ is decreasing together with the definition of the random variable $\mathcal{M}_1(t)$, we deduce that
\begin{equation*}
    \int_{t}^\infty K_{s-t} \mathcal{M}_0(s)^{\frac{p}{2(1 - \theta_d)}} \sum_{x \in \mathbb{T}_L} P_\a(s  , x)^2 \, ds  \leq \left(\int_{t}^\infty K_{s-t} \mathcal{M}_0(s)^{\frac{p }{2(1 - \theta_d)}} \, ds\right) \mathcal{E}_t \leq \mathcal{M}_1(t)\mathcal{E}_t.
\end{equation*}
A combination of the few previous displays shows that
\begin{multline*}
    \int_{t}^\infty K_{s-t} \sum_{e \in E \left( \mathbb{T}_L \right)} \rho^{p-2}(e) \a(s , e) P_\a(s, e)^2 \, ds \\ \leq \frac{1}{C_0^2} \int_{t}^\infty K_{s-t} \sum_{ e \in  E \left( \mathbb{T}_L \right)} \rho(e)^p \a(s , e) (\nabla P_\a(s , e))^2 \, ds + C\mathcal{M}_1(t)^{\frac 2p}(\bar{\mathcal{N}}_t)^{\frac{p-2}{p}} \mathcal{E}_t^{\frac 2p}.
\end{multline*}
Combining the previous inequality with~\eqref{eq:09531712} and~\eqref{new.4.277} completes the proof of Lemma~\ref{lemma3.11}.
\end{proof}

\subsubsection{An upper bound on the $L^2$-norm of the heat kernel}

We next deduce from Lemma~\ref{lemma3.11} the energy upper bound for the $L^2(\mathbb{T}_L)$-norm of the heat kernel.

\begin{proposition}\label{thm3.1222}
There exists a constant $C := C(d) < \infty$ such that, for any $t \geq 0$,
\begin{equation*}
    \mathcal{E}_t \leq \frac{C \mathscr{M}}{(1+t)^{\frac d2}}.
\end{equation*}
\end{proposition}

\begin{proof}
By Lemma~\ref{lemma3.11}, we have the inequality
\begin{equation} \label{eq:16151912}
    \partial_t \bar{\mathcal{N}}_t \leq C \mathcal{M}_1(t)^{\frac 2p} (\bar{\mathcal{N}}_t)^{\frac{p-2}{p}} \mathcal{E}_t^{\frac{2}{p}}.
\end{equation}
We define the quantity
\begin{equation*} 
\Lambda_t := \sup_{s \leq t} (1+s)^{\frac{d}{2}} \mathcal{E}_s,
\end{equation*}
and note that, for any $t \geq 0$, $\Lambda_t \geq 1$.
We first observe that~\eqref{eq:16151912} can be rewritten using the definition $\Lambda_t$ as follows: for every $t \geq 0$,
\begin{equation} \label{ineq:barNtdiff}
    \partial_t \bar{\mathcal{N}}_t^{ \frac{2}{p}} \leq C \mathcal{M}_1(t)^{\frac 2p} \Lambda_t^{\frac 2p} (1+t)^{-\frac{d}{p}}.
\end{equation}
Integrating the previous inequality, recalling the definition of the maximal quantity $\mathscr{M}_1$ and using that $\Lambda_t$ is increasing in $t$, we obtain, for any time $t \geq 2$,
\begin{align} \label{computationbarN}
    \bar{\mathcal{N}}_t^{\frac 2p} - \bar{\mathcal{N}}_{t/2}^{\frac 2p} & = \int_{t/2}^t \partial_s \bar{\mathcal{N}}_s \, ds \\
    & \leq C \Lambda_t^{\frac 2p} \left( 1 + \frac{t}{2} \right)^{- \frac dp} \int_{t/2}^t \mathcal{M}_1(s)^{\frac 2p} \, ds \notag \\
    & \leq C \Lambda_t^{\frac 2p} \left( 1 + t \right)^{- \frac dp} \int_{0}^t \mathcal{M}_1(s)^{\frac 2p} \, ds \notag \\
    & \leq C \mathscr{M}_1^{\frac 2p} \Lambda_t^{\frac 2p} (1+t)^{1-\frac{d}{p}}. \notag
\end{align}
Iterating the previous inequality (using that the map $t \mapsto \Lambda_t$ is increasing), treating the small values of $t$ (between $0$ and $1$) using the inequality~\eqref{ineq:barNtdiff}, and using the bound on $\bar{\mathcal{N}}_0$ provided by Lemma~\ref{lemma3.11}, we obtain that, for any $t \geq 0$,
\begin{align} \label{eq:barmathcalN}
    \bar{\mathcal{N}}_t & \leq C \mathscr{M}_1 \Lambda_t (1+t)^{\frac{p- d}{2}} + C \mathscr{M}_2 \\
    & \leq C \left( \mathscr{M}_1 + \mathscr{M}_2 \right) \Lambda_t ( 1 + t)^{\frac{p- d}{2}} . \notag
\end{align}
Applying the anchored Nash estimate, and using that the $L^1$-norm of the heat kernel $P_\a$ is bounded (see~\eqref{L1normheatkernelbounded}), we obtain,  for any $t \geq 0$,
\begin{equation*}
    \mathcal{E}_t \leq C \left( \mathcal{M}_{p'}(t) \left\| w(t , \cdot) \nabla P_\a(t , \cdot) \right\|^2_{L^2 \left( \mathbb{T}_L\right)} \right)^\alpha \mathcal{N}_t^\gamma.
\end{equation*}
Multiplying the previous inequality by the weight function $K$ and integrating over time, we deduce that,  for any $t \geq 0$,
\begin{equation} \label{eq:barEnashforit}
    \bar{\mathcal{E}}_t \leq \int_{t}^\infty K_{s-t} \left( \mathcal{M}_{p'}(s) \left\| w(s , \cdot) \nabla P_\a(s , \cdot)\right\|^2_{L^2 \left( \mathbb{T}_L\right)} \right)^\alpha \mathcal{N}_s^\gamma \, ds.
\end{equation}
Applying the H\"{o}lder inequality (recalling that $\alpha + \beta + \gamma = 1$), we deduce that,  for any $t \geq 0$,
\begin{multline} \label{eq:14421712}
\int_{t}^\infty K_{s-t} \left( \mathcal{M}_{p'}(t) \left\| w(s , \cdot) \nabla P_\a(s , \cdot) \right\|^2_{L^2 \left( \mathbb{T}_L\right)} \right)^\alpha \mathcal{N}_s^\gamma  \, ds\\
\leq \left( \int_{t}^\infty K_{s-t} \mathcal{M}_{p'}(s)^\frac{\alpha}{\beta} \, ds \right)^\beta \left( \int_{t}^\infty K_{s-t} \left\|  w(s , \cdot) \nabla P_\a(s , \cdot) \right\|^2_{L^2 \left( \mathbb{T}_L\right)} \, ds \right)^\alpha \left( \int_{t}^\infty K_{s-t} \mathcal{N}_s \, ds \right)^\gamma.
\end{multline}
The first term in the right-hand side is by definition smaller than $\mathcal{M}_3(t)$. We then estimate the second term in the right-hand side. To this end, we use Proposition~\ref{prop4.3} together with the bound~\eqref{K*KsmallerthanK}. We obtain, for any $t \geq 0$,
\begin{align*}
    \int_{t}^\infty K_{s-t} \left\|  w(s , \cdot) \nabla P_\a(s , \cdot) \right\|^2_{L^2 \left( \mathbb{T}_L\right)} \, ds & \leq C \int_{t}^\infty K_{s-t} \int_s^{\infty} K_{s' - s} \|  \a(s' , \cdot)^{1/2} \nabla P_\a(s' , \cdot) \|^2_{L^2 \left( \mathbb{T}_L\right)} \, ds' ds \\
    & \leq C \int_{t}^\infty K_{s-t} \|  \a(s , \cdot)^{1/2} \nabla P_\a(s , \cdot) \|^2_{L^2 \left( \mathbb{T}_L\right)} \, ds \\
    & \leq C \bar{\mathcal{D}}_t .
\end{align*}
Using the identity $\partial_t \bar{\mathcal{E}}_t = - 2\bar{\mathcal{D}}_t$, we may rewrite the inequality~\eqref{eq:14421712} as follows
\begin{equation*}
    \int_{t}^\infty K_{s-t} \left( \mathcal{M}_{p'}(t) \left\| w(s , \cdot) \nabla P_\a(s , \cdot) \right\|^2_{L^2 \left( \mathbb{T}_L\right)} \right)^\alpha \mathcal{N}_s^\gamma  \, ds \leq C \mathcal{M}_3(t) (- \partial_t \bar{\mathcal{E}}_t)^\alpha  \bar{\mathcal{N}}_t^\gamma.
\end{equation*}
Combining the previous inequality with~\eqref{eq:barEnashforit} and using the inequality~\eqref{eq:barmathcalN}, we deduce that
\begin{equation*}
    \bar{\mathcal{E}}_t \leq C \mathcal{M}_3(t) (- \partial_t \bar{\mathcal{E}}_t)^\alpha  \bar{\mathcal{N}}_t^\gamma \leq 
    C \mathcal{M}_3(t) (\mathscr{M}_1 + \mathscr{M}_2)^\gamma (- \partial_t \bar{\mathcal{E}}_t)^\alpha \Lambda_t^\gamma (1+t)^{\frac{(p-d)\gamma}{2}},
\end{equation*}
which can be rewritten as
\begin{equation} \label{eq:15561912}
    \left( - \partial_t \bar{\mathcal{E}}_t  \right) \bar{\mathcal{E}}_t^{-\frac{1}{\alpha}} \geq c\mathcal{M}_3(t)^{-\frac 1\alpha} (\mathscr{M}_1 + \mathscr{M}_2)^{-\frac{\gamma}{\alpha}} \Lambda_t^{-\frac{\gamma}{\alpha}} (1+t)^{-\frac{(p-d)\gamma}{2 \alpha}}.
\end{equation}
Using that $t \mapsto \Lambda_t$ is increasing and that $p > d$, we deduce that
\begin{equation*}
   \bar{\mathcal{E}}_t^{1 - \frac 1\alpha} \geq c \Lambda_t^{-\frac{\gamma}{\alpha}} (1+t)^{-\frac{(p-d)\gamma}{2 \alpha}}  (\mathscr{M}_1 + \mathscr{M}_2)^{-\frac{\gamma}{\alpha}}  \int_0^t \mathcal{M}_3^{-\frac{1}{\alpha}}(s) \, ds.
\end{equation*}
Using the identity~\eqref{eq:3.7} and the definitions of the random variables $\mathscr{M}_3$ and $\mathscr{M}$, we deduce that, for any $t \geq 1$,
\begin{equation*}
    (1+t)^{\frac d2} \bar{\mathcal{E}}_t \leq C \Lambda_t^{\frac{\gamma}{1-\alpha}} \mathscr{M}^{1 - \frac{\gamma}{1-\alpha}}.
\end{equation*}
Using that $\mathcal{E}_t$ is decreasing in $t$ and the definition of $\bar{\mathcal{E}}_t$, we have the inequality 
\begin{equation*}
\left( \int_0^1 K_s \, ds \right)  \mathcal{E}_{t+1} \leq \bar{\mathcal{E}}_t .
\end{equation*}
Using that $t \mapsto \Lambda_t$ is increasing, we obtain, for any $t \geq 1$,
\begin{equation*}
    (1+t)^{\frac{d}2} \mathcal{E}_{t+1} \leq C \mathscr{M}^{1- \frac{\gamma}{1-\alpha}}\Lambda_{t+1}^{\frac{\gamma}{1-\alpha}}.
\end{equation*}
Combining the previous inequality with the bound $\mathcal{E}_t \leq 1$, the observation that $\Lambda_{t}$ is increasing and larger than $1$, we deduce that, for any $t \geq 0$,
\begin{equation*}
    \Lambda_t \leq C  \mathscr{M}^{1- \frac{\gamma}{1-\alpha}} \Lambda_{t}^{\frac{\gamma}{1-\alpha}} \implies \Lambda_t \leq C \mathscr{M}.
\end{equation*}
The proof of Proposition~\ref{thm3.1222} is complete.
\end{proof}

\subsubsection{A refined upper bound on $L^2$-norm on the heat kernel}
This section is devoted to the proof of Theorem~\ref{thm3.13} building upon Proposition~\ref{thm3.1222}.

\begin{proof}[Proof of Theorem~\ref{thm3.13}]
We let $C_2 \geq 1$ be a large constant whose value will be selected later in the argument and shall depend only on $d$, and define
\begin{equation} \label{def.C1andC0}
    C_1 := C_2^{\frac{2 \alpha}{ d \beta + p \gamma}}  \mathscr{M}_3^{\frac{2\gamma}{(d \beta + p \gamma)}} ~~ \mbox{and} ~~ C_0 :=   2 \mathscr{M}_3^{\frac{\gamma}{\alpha}} C_1.
\end{equation}
We then define the three quantities
\begin{equation*}
\mathcal{H}_t := e^{\frac{1}{C_0 L^2} \int_0^t \mathcal{M}_4(s)^{-1} \, ds} \mathcal{E}_t,~~\bar{\mathcal{H}}_t := e^{\frac{1}{C_0 L^2} \int_0^t \mathcal{M}_4(s)^{-1} \, ds} \bar{\mathcal{E}}_t ~~ \mbox{and}~~ 
\Xi_t := \sup_{s \leq t} (1+s)^{\frac{d}{2}} \mathcal{H}_s.
\end{equation*}
We next prove the following upper bound
\begin{equation} \label{eq:13071912}
     \bar{\mathcal{E}}_t \leq C \mathcal{M}_4(t) L^2(-\partial_t \bar{\mathcal{E}}_t) .
\end{equation}
To prove the inequality~\eqref{eq:13071912}, we first use that the function $\mathcal{E}_t$ is decreasing and write
\begin{align*}
    \bar{\mathcal{E}}_t & = \int_t^\infty K_{s-t} \left\| P_\a(s , \cdot) \right\|_{L^2\left( \mathbb{T}_L \right)}^2 \, ds \\
    & = \int_t^{t+1} K_{s-t} \left\| P_\a(s , \cdot) \right\|_{L^2\left( \mathbb{T}_L \right)}^2 \, ds + \int_{t+1}^{\infty} K_{s-t} \left\| P_\a(s , \cdot) \right\|_{L^2\left( \mathbb{T}_L \right)}^2 \, ds \\
    & \leq \int_t^{t+1} K_{s-t} \left\| P_\a(s , \cdot) \right\|_{L^2\left( \mathbb{T}_L \right)}^2  ds+ \left\| P_\a(t+1 , \cdot) \right\|_{L^2\left( \mathbb{T}_L \right)}^2 \int_{1}^\infty K_s \, ds.
\end{align*}
Using a second time that $\mathcal{E}_t$ is decreasing, that the ratio $\int_{1}^\infty K_s \, ds / \int_{0}^1 K_s \, ds$ is a finite constant depending only on the parameter $d$, and the definition of the random variable $\mathcal{M}_4(t)$, we deduce that
\begin{align*}
    \bar{\mathcal{E}}_t & \leq \left( 1 + \frac{\int_{1}^\infty K_s \, ds}{\int_{0}^1 K_s \, ds} \right) \int_t^{t+1} K_{s-t} \left\| P_\a(s , \cdot) \right\|_{L^2\left( \mathbb{T}_L \right)}^2 \, ds \\
    & \leq C \mathcal{M}_4(t) \int_t^{t+1} K_{s-t} \| w^{-1} \left(s , \cdot \right)\|_{\underline{L}^d(\mathbb{T}_L)}^{-2} \left\| P_\a(s , \cdot)\right\|_{L^2\left( \mathbb{T}_L \right)}^2 \, ds \\
    & \leq C \mathcal{M}_4(t) \int_t^{\infty} K_{s-t} \| w^{-1} \left(s , \cdot \right)\|_{\underline{L}^d(\mathbb{T}_L)}^{-2} \left\| P_\a(s , \cdot)\right\|_{L^2\left( \mathbb{T}_L \right)}^2 \, ds .
\end{align*}
We next use the Gagliardo-Nirenberg-Sobolev inequality (using that $\sum_{x \in \mathbb{T}_L} P_\a(s , x) = 0$), then the H\"{o}lder inequality. We obtain
\begin{align*}
    \bar{\mathcal{E}}_t & \leq C \mathcal{M}_4(t) \int_t^{\infty} K_{s-t} \| w^{-1} \left(s , \cdot \right)\|_{\underline{L}^d(\mathbb{T}_L)}^{-2}  \left\| P_\a(s , \cdot)\right\|_{L^2\left( \mathbb{T}_L \right)}^2 \, ds \\
    & \leq C \mathcal{M}_4(t) \int_t^\infty K_{s-t} \| w^{-1} \left(s , \cdot \right)\|_{\underline{L}^d(\mathbb{T}_L)}^{-2} \left\| \nabla P_\a(s , \cdot) \right\|_{L^{\frac{2d}{d+2}}\left( \mathbb{T}_L \right)}^2 \, ds \\
    & \leq C \mathcal{M}_4(t) L^2 \int_t^\infty K_{s-t} \left\| w(s , \cdot) \nabla P_\a(s , \cdot) \right\|_{L^{2}\left( \mathbb{T}_L \right)}^2 \, ds.
\end{align*}
Using Proposition~\ref{prop4.3} together with the bound~\eqref{K*KsmallerthanK}, we deduce that
\begin{align} \label{eq:14561912}
    \bar{\mathcal{E}}_t & \leq C \mathcal{M}_4(t) L^2 \int_t^\infty K_{s-t} \| \a(s , \cdot)^{1/2} \nabla P_\a(s , \cdot) \|_{L^{2}\left( \mathbb{T}_L \right)}^2 \, ds \\
    & \leq  C \mathcal{M}_4(t) L^2 ( - \partial_t \bar{\mathcal{E}}_t). \notag
\end{align}
The proof of the inequality~\eqref{eq:13071912} is complete. We impose here a first condition on the constant $C_2$ and choose it sufficiently large so that the constant $C_2$ is larger than $2C$, where $C$ is the constant appearing in the right-hand side of~\eqref{eq:13071912}. We thus deduce that
\begin{align} \label{eq:HandEcompared}
   \left(  -  \partial_t  \bar{\mathcal{H}}_t  \right) \bar{\mathcal{H}}_t^{-\frac{1}{\alpha}} & = e^{ \left(1 - \frac{1}{\alpha} \right) \frac{1}{C_0L^2} \int_0^t \mathcal{M}_4(s)^{-1} \, ds} \left( - \frac{1}{C_0 \mathcal{M}_4(t) L^2}  \bar{\mathcal{E}}_t  - \partial_t   \bar{\mathcal{E}}_t  \right) \bar{\mathcal{E}}_t^{-\frac{1}{\alpha}} \\
   & \geq \frac12 e^{ \left(1 - \frac{1}{\alpha} \right) \frac{1}{C_0L^2} \int_0^t \mathcal{M}_4(s)^{-1} \, ds} \left( - \partial_t   \bar{\mathcal{E}}_t \right)  \bar{\mathcal{E}}_t^{-\frac{1}{\alpha}}. \notag
\end{align}
Applying the anchored Nash inequality, we have
\begin{equation*}
    \mathcal{E}_t \leq C \left\| P_\a(t , \cdot) \right\|^{2\beta}_{L^1\left( \mathbb{T}_L \right)} \left( \mathcal{M}_{p'}(t) \left\| w(t , \cdot) \nabla P_\a(t , \cdot) \right\|^2_{L^2 \left( \mathbb{T}_L\right)} \right)^\alpha \mathcal{N}_t^\gamma.
\end{equation*}
We then estimate the first and third terms in the right hand side by using the Cauchy-Schwarz inequality and the bound $\rho(x) \leq C L$, which is valid on the torus since its diameter is of order $L$. We obtain
\begin{equation*}
    \left\| P_\a(t , \cdot) \right\|_{L^1\left( \mathbb{T}_L \right)}^2 \leq C L^{d} \mathcal{E}_t ~~ \mbox{and}~~ \mathcal{N}_t \leq C L^p \mathcal{E}_t.
\end{equation*}
Combining the two previous displays, we may write 
\begin{equation*}
    \mathcal{E}_t \leq C \left( L^{d}\mathcal{E}_t \right)^{\beta} \left( \mathcal{M}_{p'}(t) \left\| w(t , \cdot) \nabla u(t , \cdot) \right\|^2_{L^2 \left( \mathbb{T}_L\right)} \right)^\alpha (L^p \mathcal{E}_t)^\gamma.
\end{equation*}
Performing the same computation as in~\eqref{eq:14421712} and using that the function $\mathcal{E}_t$ is decreasing, we deduce that 
\begin{equation*}
    \bar{\mathcal{E}}_t \leq C \mathcal{M}_3(t) L^{d\beta + p \gamma} \mathcal{E}_t ^{\beta + \gamma} (- \partial_t \bar{\mathcal{E}}_t)^\alpha.
\end{equation*}
Using the definition of $\Xi_t$, we further deduce that
\begin{align*}
    \bar{\mathcal{E}}_t  \leq C \mathcal{M}_3(t) L^{d\beta + p \gamma} (1 + t)^{- \frac{(\beta + \gamma) d}{2}} e^{-\frac{\beta + \gamma}{C_0 L^2} \int_0^t \mathcal{M}_4(s)^{-1} \, ds}  \Xi_t^{\beta + \gamma}(- \partial_t \bar{\mathcal{E}}_t)^\alpha.
\end{align*}
Rearranging the previous inequality, we obtain
\begin{equation*}
    (- \partial_t \bar{\mathcal{E}}_t)\bar{\mathcal{E}}_t^{-\frac 1\alpha} \geq c \mathcal{M}_3(t)^{-\frac{1}{\alpha}} L^{- \frac{d\beta + p \gamma}{\alpha}} (1 + t)^{ \frac{(\beta + \gamma) d}{2\alpha}} e^{\frac{\beta + \gamma}{\alpha C_0 L^2} \int_0^t \mathcal{M}_4(s)^{-1} \, ds}  \Xi_t^{-\frac{\beta + \gamma}{\alpha}}.
\end{equation*}
Combining the previous inequality with~\eqref{eq:HandEcompared} and noting that $1 - \frac{1}{\alpha} = -\frac{\beta + \gamma}{\alpha}$ (since $\alpha + \beta + \gamma =1$), we can rewrite the previous inequality as follows
\begin{align*}
     \partial_t \bar{\mathcal{H}}_t^{1-\frac{1}{\alpha}}
    & \geq c \mathcal{M}_3(t)^{-\frac{1}{\alpha}}L^{- \frac{d\beta + p \gamma}{\alpha}} (1 + t)^{ \frac{(\beta + \gamma) d}{2\alpha}} \Xi_t^{-\frac{\beta + \gamma}{\alpha}} \\
    & \geq c \mathcal{M}_3(t)^{-\frac{1}{\alpha}} \left(\frac{1 + t}{L^2} \right)^{ \frac{d\beta + p \gamma}{2\alpha}} (1 + t)^{- \frac{\gamma (p-d)}{2\alpha}} \Xi_t^{-\frac{\beta + \gamma}{\alpha}}.
\end{align*}
Using the definition of the constant $C_1$, we have, for any $t \geq C_1 L^2$,
\begin{equation*}
    \partial_t \bar{\mathcal{H}}_t^{1-\frac{1}{\alpha}} \geq c C_2 \mathscr{M}_3^{\frac{\gamma}{\alpha}} \mathcal{M}_3(t)^{-\frac{1}{\alpha}}   (1+t)^{-\frac{\gamma (p- d)}{2 \alpha}} \Xi_t^{-\frac{\beta + \gamma}{\alpha}}.
\end{equation*}
Integrating the previous inequality and using that $\Xi_t$ is increasing in $t$, we deduce that, for any $t \geq C_1 L^2$,
\begin{equation*}
    \bar{\mathcal{H}}_t^{1-\frac{1}{\alpha}} - \bar{\mathcal{H}}_{C_1 L^2}^{1-\frac{1}{\alpha}} \geq c C_2  \Xi_t^{-\frac{\beta + \gamma}{\alpha}}  (1+t)^{-\frac{\gamma (p- d)}{2 \alpha}} \mathscr{M}_3^{\frac{\gamma}{\alpha}} \int_{C_1 L^2}^{t}\mathcal{M}_3(s)^{-\frac{1}{\alpha}} \, ds.
\end{equation*}
We next recall the definition of the constant $C_0$ introduced in~\eqref{def.C1andC0}, and lower bound the term in the right-hand side for $t \geq C_0 L^2$. To this end, we use the definition of the random variable $\mathscr{M}_3$ and the lower bound $\mathcal{M}_3 \geq 1$, and obtain, for any $t \geq C_0 L^2$,
\begin{align*}
    \frac{\mathscr{M}_3^{\frac{\gamma}{\alpha}}}{t}  \int_{C_1 L^2}^{t}\mathcal{M}_3(s)^{-\frac{1}{\alpha}} \, ds
    & =  \frac{\mathscr{M}_3^{\frac{\gamma}{\alpha}}}{t}\int_{0}^{t} \mathcal{M}_3(s)^{-\frac{1}{\alpha}} \, ds - \frac{\mathscr{M}_3^{\frac{\gamma}{\alpha}}}{t} \int_{0}^{C_1 L^2} \mathcal{M}_3(s)^{-\frac{1}{\alpha}} \, ds \\
    & \geq 1 - \frac{C_1 L^2 \mathscr{M}_3^{\frac{\gamma}{\alpha}}}{t} \\
    & \geq \frac{1}{2}.
\end{align*}
A combination of the two previous displays yields, for any $t \geq C_0 L^2$,
\begin{equation*}
    \bar{\mathcal{H}}_t^{1-\frac{1}{\alpha}} \geq \bar{\mathcal{H}}_t^{1-\frac{1}{\alpha}} - \bar{\mathcal{H}}_{C_1 L^2}^{1-\frac{1}{\alpha}} \geq c C_2  \Xi_t^{-\frac{\beta + \gamma}{\alpha}}  (1+t)^{1-\frac{\gamma (p- d)}{2 \alpha}} = c C_2  \Xi_t^{ 1 - \frac{1}{\alpha}}  (1+t)^{1-\frac{\gamma (p- d)}{2 \alpha}} .
\end{equation*}
We finally remove the averaging from the previous inequality. Using that the map $\mathcal{E}_t$ is decreasing, we have the estimate
\begin{equation*}
     \left( \int_0^1 K_s \, ds \right) \mathcal{E}_{t+1} \leq \bar{\mathcal{E}}_t,
\end{equation*}
and combining the previous inequality with the bound $\mathcal{M}_4^{-1}(t) \leq 1$ (which follows from the definition of $\mathcal{M}_4$), we deduce that
\begin{equation*}
    \left( \int_0^1 K_s \, ds \right) e^{-1/(C_0L^2)} \mathcal{H}_{t+1} \leq \bar{\mathcal{H}}_t.
\end{equation*}
Combining the few previous displays and using that $t \mapsto \Xi_t$ is increasing, we deduce that, for any $t \geq C_0 L^2+1$,
\begin{equation*}
    \mathcal{H}_{t}^{1-\frac{1}{\alpha}} \geq c C_2 \Xi_{t}^{1 - \frac{1}{\alpha}}  (1+t)^{1-\frac{\gamma (p- d)}{2 \alpha}}.
\end{equation*}
We next impose a second condition on the constant $C_2$ and assume that $c C_2 \geq 2^{\frac{1}{\alpha} -1}$. This leads to the bound, for any $t \geq C_0 L^2+1$,
\begin{equation*}
   (1+t)^{\frac d2} \mathcal{H}_{t} \leq \frac12 \Xi_{t}.
\end{equation*}
We next apply Proposition~\ref{thm3.1222} and the bound $\mathcal{M}_4(t)^{-1} \leq 1$ and obtain
\begin{align*}
    \sup_{t \in [0 , C_0 L^2+1]} (1+t)^{\frac{d}{2}} \mathcal{H}_t & = \sup_{t \in [0 , C_0 L^2+1]} e^{\frac{1}{C_0 L^2} \int_0^t \mathcal{M}_4(s)^{-1} \, ds} (1+t)^{\frac{d}{2}} \mathcal{E}_t \\
    & \leq e^{\frac{C_0 L^2 + 1}{C_0 L^2} } \sup_{t \in [0 , C_0 L^2+1]}  (1+t)^{\frac{d}{2}} \mathcal{E}_t \\
    & \leq C \mathscr{M}.
\end{align*}
Combining the two previous displays, we deduce that, for any $t \geq 0$,
\begin{equation*}
    \Xi_t \leq \frac{1}{2} \Xi_t + C \mathscr{M},
\end{equation*}
and thus, for any $t \geq 0$,
\begin{equation*}
    \Xi_t \leq C \mathscr{M}.
\end{equation*}
The proof of Theorem~\ref{thm3.13} is complete.
\end{proof}

\subsection{On diagonal estimate for the heat kernel} \label{section:ondiagonalheat}

In this section, we deduce from Theorem~\ref{thm3.13} the on-diagonal upper bound on the heat-kernel $P_\a$. In order to state the result, we fix a time $t \in [0 , \infty)$ and define the reversed environment 
\begin{equation*}
\a^{(t)} (t',e) := \a(t - t', e).
\end{equation*}
The environment $\a^{(t)}$ is only defined for the times $t' \in [0 , t]$. This is the only relevant property for the statement below; but we note that we may extend its definition to all times so as to make $\a^{(t)}$ a stationary process by for instance extending the definition of the Langevin dynamic to negative times.  

Since the Langevin dynamic is stationary and reversible with respect to the Gibbs measure $\mu_{\mathbb{T}_L}$, the processes $\a$ and $\a^{(t)}$ have the same law. Let us denote by $\mathscr{M}^{(t)}$ the random variable $\mathscr{M}$ associated with the environment $\a^{(t)}$. Since the processes $\a$ and $\a^{(t)}$ have the same law, the random variables $\mathscr{M}$ and $\mathscr{M}^{(t)}$ also have the same law.

\begin{proposition}[On-diagonal heat kernel decay] \label{cor3.14}
There exists a constant $C := C(d) < \infty$ such that, for any time $t \geq 0$,
\begin{equation*}
P_\a (t , 0) \leq \frac{C \sqrt{\mathscr{M} \mathscr{M}^{(t)}}}{(1+t)^{\frac d2}} \exp \left( - \frac{t}{C \mathscr{M}'L^2} \right).
\end{equation*}
\end{proposition}

\begin{proof}
Using the convolution property of the heat kernel and the Cauchy-Schwarz inequality, we obtain
\begin{equation} \label{eq:CScorr3.13}
    P_\a (t , 0) = \sum_{x \in \mathbb{T}_L} P_\a (t , 0 ; t/2 , x) P_\a \left( t/2 , x \right) \leq \left(\sum_{x \in \mathbb{T}_L} P_\a (t , 0 ; t/2 , x)^2 \right)^{\frac 12} \left(\sum_{x \in \mathbb{T}_L} P_\a (t/2 , x )^2 \right)^{\frac 12}.
\end{equation}
To estimate the first term in the right-hand side, we use the following identity between the heat kernel and the heat kernel under the reversed environment (see~\cite[Lemma 4.5]{MO16})
\begin{equation*}
    P_\a (t , 0 ; t/2 , x) = P_{\a^{(t)}} \left( t/2 , x \right).
\end{equation*}
Applying Proposition~\ref{thm3.1222} with the environment $\a^{(t)}$ (and thus the random variable $\mathscr{M}^{(t)}$), we deduce that
\begin{equation*}
    \sum_{x \in \mathbb{T}_L} P_\a (t , 0 ; t/2 , x)^2  \leq \frac{C \mathscr{M}^{(t)}}{(1+t)^{\frac d2}}.
\end{equation*}
The second term in the right-hand side of~\eqref{eq:CScorr3.13} can be estimated using Theorem~\ref{thm3.13}. We obtain
\begin{equation*}
    \sum_{x \in \mathbb{T}_L} P_\a (t/2 ,x)^2  \leq \frac{C \mathscr{M}}{(1+t)^{\frac d2}} \exp \left( - \frac{t}{C \mathscr{M}' L^2} \right).
\end{equation*}
Combining the two previous displays with~\eqref{eq:CScorr3.13} completes the proof of Proposition~\ref{cor3.14}.

\end{proof}

\subsection{Helffer-Sj\"{o}strand representation formula and proof of Theorem~\ref{main.theorem}} \label{section:localization}

We are then able to complete the proof of the localization and delocalization estimate for the random surface stated in Theorem~\ref{main.theorem} by combining Proposition~\ref{cor3.14} and the Helffer-Sj\"{o}strand representation formula.

\begin{proof}[Proof of Theorem~\ref{main.theorem}]
By the Helffer-Sj\"{o}strand representation formula, we have the identity
\begin{equation*}
    \var_{\mathbb{T}_L} \left[\phi(0) \right]  = \E \left[ \int_0^\infty  P_\a \left( t , 0 \right) \, dt \right].
\end{equation*}
Applying Proposition~\ref{cor3.14} and using the inequality $\exp (-t) \leq 1/t$ for $t > 0$, we see that
\begin{align*}
    \int_0^\infty  P_\a \left( t , 0 \right) \, dt & 
     \leq \int_0^\infty \frac{C \sqrt{\mathscr{M} \mathscr{M}^{(\frac{t}{2})}}}{(1+t)^{\frac d2}} \exp \left( - \frac{t}{C \mathscr{M}' L^2} \right) \, dt \\
    & \leq  \int_0^{L^2} \frac{C \sqrt{\mathscr{M} \mathscr{M}^{(\frac{t}{2})}}}{(1+t)^{\frac d2}} dt + \int_{L^2}^\infty \frac{C \sqrt{\mathscr{M} \mathscr{M}^{(\frac{t}{2})}}}{(1+t)^{\frac d2}} \frac{\mathscr{M}' L^2}{t} \, dt.
\end{align*}
Taking the expectation in the previous inequality, and using that all the moments of the random variables $\mathscr{M}$, $\mathscr{M}^{(t)}$ and $\mathscr{M}' $ are finite (in particular the random variables $\sqrt{\mathscr{M} \mathscr{M}^{(\frac{t}{2})}}$ and $\sqrt{\mathscr{M} \mathscr{M}^{(\frac{t}{2})}} \mathscr{M}'$ have a finite expectation  whose value can be bounded uniformly in $t$) completes the proof of Theorem~\ref{main.theorem}.
\end{proof}

{\small
\bibliographystyle{abbrv}
\bibliography{degenerategradphi.bib}

\newcommand{\noop}[1]{} \def\cprime{$'$}
\begin{thebibliography}{10}

\bibitem{adams2019cauchy}
S.~Adams, S.~Buchholz, R.~Koteck{\'y}, and S.~M{\"u}ller.
\newblock Cauchy-{B}orn rule from microscopic models with non-convex
  potentials.
\newblock {\em arXiv preprint arXiv:1910.13564}, 2019.

\bibitem{adams2023hessian}
S.~Adams and A.~Koller.
\newblock The hessian of surface tension characterises scaling limit of
  gradient models with non-convex energy.
\newblock {\em arXiv preprint arXiv:2306.12226}, 2023.

\bibitem{AKM16}
S.~Adams, R.~Koteck{\'y}, and S.~M{\"u}ller.
\newblock Strict convexity of the surface tension for non-convex potentials.
\newblock {\em arXiv preprint arXiv:1606.09541}, 2016.

\bibitem{aizenman2021depinning}
M.~Aizenman, M.~Harel, R.~Peled, and J.~Shapiro.
\newblock Depinning in integer-restricted {G}aussian {F}ields and {BKT} phases
  of two-component spin models.
\newblock {\em arXiv preprint arXiv:2110.09498}, 2021.

\bibitem{andres2013invariance}
S.~Andres, M.~T. Barlow, J.-D. Deuschel, and B.~Hambly.
\newblock Invariance principle for the random conductance model.
\newblock {\em Probab. Theory Related Fields}, 156(3-4):535--580, 2013.

\bibitem{andres2018quenched}
S.~Andres, A.~Chiarini, J.-D. Deuschel, and M.~Slowik.
\newblock Quenched invariance principle for random walks with time-dependent
  ergodic degenerate weights.
\newblock {\em Ann. Probab.}, 46(1):302--336, 2018.

\bibitem{andres2021quenched}
S.~Andres, A.~Chiarini, and M.~Slowik.
\newblock Quenched local limit theorem for random walks among time-dependent
  ergodic degenerate weights.
\newblock {\em Probab. Theory Related Fields}, 179(3):1145--1181, 2021.

\bibitem{andres2016heat}
S.~Andres, J.-D. Deuschel, and M.~Slowik.
\newblock Heat kernel estimates for random walks with degenerate weights.
\newblock {\em Electron. J. Probab.}, 21:1--21, 2016.

\bibitem{andres2019heat}
S.~Andres, J.-D. Deuschel, and M.~Slowik.
\newblock Heat kernel estimates and intrinsic metric for random walks with
  general speed measure under degenerate conductances.
\newblock {\em Electron. Comm. Probab.}, 24:1--17, 2019.

\bibitem{ADS20}
S.~Andres, J.-D. Deuschel, and M.~Slowik.
\newblock {Green kernel asymptotics for two-dimensional random walks under
  random conductances}.
\newblock {\em Electron. Comm. Probab.}, 25:1 -- 14, 2020.

\bibitem{andres2021lower}
S.~Andres and N.~Halberstam.
\newblock Lower {G}aussian heat kernel bounds for the random conductance model
  in a degenerate ergodic environment.
\newblock {\em Stoch. Processes Appl.}, 139:212--228, 2021.

\bibitem{AT21}
S.~Andres and P.~A. Taylor.
\newblock Local limit theorems for the random conductance model and
  applications to the {G}inzburg-{L}andau {$\nabla\phi$} interface model.
\newblock {\em J. Stat. Phys.}, 182(2):Paper No. 35, 35, 2021.

\bibitem{armstrong2022quantitative}
S.~Armstrong and P.~Dario.
\newblock Quantitative hydrodynamic limits of the {L}angevin dynamics for
  gradient interface models.
\newblock {\em arXiv preprint arXiv:2203.14926}, 2022.

\bibitem{AW}
S.~Armstrong and W.~Wu.
\newblock {$C^2$} regularity of the surface tension for the {$\nabla\phi$}
  interface model.
\newblock {\em Comm. Pure Appl. Math.}, 75(2):349--421, 2022.

\bibitem{avena2012symmetric}
L.~Avena.
\newblock Symmetric exclusion as a model of non-elliptic dynamical random
  conductances.
\newblock {\em Electron. Comm. Probab.}, 17:1--8, 2012.

\bibitem{bandyopadhyay2005random}
A.~Bandyopadhyay and O.~Zeitouni.
\newblock Random walk in dynamic {M}arkovian random environment.
\newblock {\em ALEA Lat. Am. J. Probab. Math. Stat.}, 1:205–224, 2008.

\bibitem{Barlowheat}
M.~T. Barlow.
\newblock {Random walks on supercritical percolation clusters}.
\newblock {\em Ann. Probab.}, 32(4):3024 -- 3084, 2004.

\bibitem{BH}
M.~T. Barlow and B.~M. Hambly.
\newblock Parabolic {H}arnack inequality and local limit theorem for
  percolation clusters.
\newblock {\em Electron. J. Probab.}, 14:no. 1, 1--27, 2009.

\bibitem{BPR1}
R.~Bauerschmidt, J.~Park, and P.-F. Rodriguez.
\newblock The {D}iscrete {G}aussian model, {I}. {R}enormalisation group flow at
  high temperature.
\newblock {\em To appear in Ann. Probab., arXiv preprint arXiv:2202.02286},
  2022.

\bibitem{BPR2}
R.~Bauerschmidt, J.~Park, and P.-F. Rodriguez.
\newblock The {D}iscrete {G}aussian model, {II}. {I}nfinite-volume scaling
  limit at high temperature.
\newblock {\em To appear in Ann. Probab., arXiv preprint arXiv:2202.02287},
  2022.

\bibitem{BeliusWumax}
D.~Belius and W.~Wu.
\newblock {Maximum of the Ginzburg–Landau fields}.
\newblock {\em Ann. Probab.}, 48(6):2647 -- 2679, 2020.

\bibitem{bella2020quenched}
P.~Bella and M.~Sch{\"a}ffner.
\newblock Quenched invariance principle for random walks among random
  degenerate conductances.
\newblock {\em Ann. Probab.}, 48(1):296--316, 2020.

\bibitem{BB}
N.~Berger and M.~Biskup.
\newblock Quenched invariance principle for simple random walk on percolation
  clusters.
\newblock {\em Probab. Theory Related Fields}, 137(1-2):83--120, 2007.

\bibitem{berger2008anomalous}
N.~Berger, M.~Biskup, C.~E. Hoffman, and G.~Kozma.
\newblock Anomalous heat-kernel decay for random walk among bounded random
  conductances.
\newblock {\em Ann. Inst. Henri Poincar\'{e} Probab. Stat.}, 44(2):374--392,
  2008.

\bibitem{biskupsurvey}
M.~Biskup.
\newblock Recent progress on the random conductance model.
\newblock {\em Probab. Surv.}, 8:294--373, 2011.

\bibitem{biskup2012subdiffusive}
M.~Biskup and O.~Boukhadra.
\newblock Subdiffusive heat-kernel decay in four-dimensional iid random
  conductance models.
\newblock {\em Journal of the London Mathematical Society}, 86(2):455--481,
  2012.

\bibitem{BK07}
M.~Biskup and R.~Koteck\'{y}.
\newblock Phase coexistence of gradient {G}ibbs states.
\newblock {\em Probab. Theory Related Fields}, 139(1-2):1--39, 2007.

\bibitem{biskup2022invariance}
M.~Biskup and M.~Pan.
\newblock An invariance principle for one-dimensional random walks in
  degenerate dynamical random environments.
\newblock {\em arXiv preprint arXiv:2209.02246}, 2022.

\bibitem{BP}
M.~Biskup and T.~M. Prescott.
\newblock Functional {CLT} for random walk among bounded random conductances.
\newblock {\em Electron. J. Probab.}, 12:no. 49, 1323--1348, 2007.

\bibitem{biskup2018limit}
M.~Biskup and P.-F. Rodriguez.
\newblock Limit theory for random walks in degenerate time-dependent random
  environments.
\newblock {\em J. Funct. Anal.}, 274(4):985--1046, 2018.

\bibitem{BS11}
M.~Biskup and H.~Spohn.
\newblock Scaling limit for a class of gradient fields with nonconvex
  potentials.
\newblock {\em Ann. Probab.}, 39(1):224--251, 2011.

\bibitem{boldrighini1997almost}
C.~Boldrighini, R.~Minlos, and A.~Pellegrinotti.
\newblock Almost-sure central limit theorem for a markov model of random walk
  in dynamical random environment.
\newblock {\em Probab. Theory Related Fields}, 109(2):245--273, 1997.

\bibitem{boldrighini2007random}
C.~Boldrighini, R.~A. Minlos, and A.~Pellegrinotti.
\newblock Random walks in a random (fluctuating) environment.
\newblock {\em Russian Mathematical Surveys}, 62(4):663, 2007.

\bibitem{boukhadra2010heat}
O.~Boukhadra.
\newblock Heat-kernel estimates for random walk among random conductances with
  heavy tail.
\newblock {\em Stoch. Process. Appl.}, 120(2):182--194, 2010.

\bibitem{BLL75}
H.~J. Brascamp, E.~H. Lieb, and J.~L. Lebowitz.
\newblock The statistical mechanics of anharmonic lattices.
\newblock {\em Bull. Inst. Internat. Statist.}, 46(1):393--404 (1976), 1975.

\bibitem{bricmont1982surface}
J.~Bricmont, J.-R. Fontaine, and J.~L. Lebowitz.
\newblock Surface tension, percolation, and roughening.
\newblock {\em J. Stat. Phys.}, 29(2):193--203, 1982.

\bibitem{brydges2012fluctuation}
D.~Brydges and T.~Spencer.
\newblock Fluctuation estimates for sub-quadratic gradient field actions.
\newblock {\em J. Math. Phys.}, 53(9):095216, 2012.

\bibitem{BY}
D.~Brydges and H.-T. Yau.
\newblock Grad {$\phi$} perturbations of massless {G}aussian fields.
\newblock {\em Comm. Math. Phys.}, 129(2):351--392, 1990.

\bibitem{buch2021phase}
S.~Buchholz.
\newblock Phase transitions for a class of gradient fields.
\newblock {\em Probab. Theory Related Fields}, 179(3):969--1022, 2021.

\bibitem{Buckley}
S.~Buckley.
\newblock {Anomalous heat kernel behaviour for the dynamic random conductance
  model}.
\newblock {\em Electron. Comm. Probab.}, 18:1 -- 11, 2013.

\bibitem{CD12}
C.~Cotar and J.-D. Deuschel.
\newblock Decay of covariances, uniqueness of ergodic component and scaling
  limit for a class of {$\nabla\phi$} systems with non-convex potential.
\newblock {\em Ann. Inst. Henri Poincar\'{e} Probab. Stat.}, 48(3):819--853,
  2012.

\bibitem{CDM09}
C.~Cotar, J.-D. Deuschel, and S.~M\"{u}ller.
\newblock Strict convexity of the free energy for a class of non-convex
  gradient models.
\newblock {\em Comm. Math. Phys.}, 286(1):359--376, 2009.

\bibitem{DD05}
T.~Delmotte and J.-D. Deuschel.
\newblock On estimating the derivatives of symmetric diffusions in stationary
  random environment, with applications to {$\nabla\phi$} interface model.
\newblock {\em Probab. Theory Related Fields}, 133(3):358--390, 2005.

\bibitem{deuschel2000entropic}
J.-D. Deuschel and G.~Giacomin.
\newblock Entropic repulsion for massless fields.
\newblock {\em Stoch. Process. Appl.}, 89(2):333--354, 2000.

\bibitem{DGI00}
J.-D. Deuschel, G.~Giacomin, and D.~Ioffe.
\newblock Large deviations and concentration properties for {$\nabla\phi$}
  interface models.
\newblock {\em Probab. Theory Related Fields}, 117(1):49--111, 2000.

\bibitem{DNV19}
J.-D. Deuschel, T.~Nishikawa, and Y.~Vignaud.
\newblock Hydrodynamic limit for the {G}inzburg-{L}andau {$\nabla\phi$}
  interface model with non-convex potential.
\newblock {\em Stoch. Processes Appl.}, 129(3):924--953, 2019.

\bibitem{deuschel2022isomorphism}
J.-D. Deuschel and P.-F. Rodriguez.
\newblock An isomorphism theorem for {G}inzburg-{L}andau interface models and
  scaling limits.
\newblock {\em arXiv preprint arXiv:2206.14805}, 2022.

\bibitem{dobrushin1980nonexistence}
R.~Dobrushin and S.~Shlosman.
\newblock Nonexistence of one-and two-dimensional {G}ibbs fields with
  noncompact group of continuous symmetries.
\newblock {\em Multicomponent random systems}, 6:199--210, 1980.

\bibitem{dolgopyat2008random}
D.~Dolgopyat, G.~Keller, and C.~Liverani.
\newblock Random walk in {M}arkovian environment.
\newblock {\em Ann. Probab.}, 36(5):1676--1710, 2008.

\bibitem{Efron1965}
B.~Efron.
\newblock Increasing properties of {P}olya frequency function.
\newblock {\em The Annals of Mathematical Statistics}, pages 272--279, 1965.

\bibitem{frohlich1981absence}
J.~Fr{\"o}hlich and C.~Pfister.
\newblock On the absence of spontaneous symmetry breaking and of crystalline
  ordering in two-dimensional systems.
\newblock {\em Comm. Math. Phys.}, 81(2):277--298, 1981.

\bibitem{frohlich1981kosterlitz}
J.~Fr{\"o}hlich and T.~Spencer.
\newblock The {K}osterlitz-{T}houless transition in two-dimensional {A}belian
  spin systems and the {C}oulomb gas.
\newblock {\em Comm. Math. Phys.}, 81(4):527--602, 1981.

\bibitem{F05}
T.~Funaki.
\newblock Stochastic interface models.
\newblock In {\em Lectures on {P}robability {T}heory and {S}tatistics}, volume
  1869 of {\em Lecture Notes in Math.}, pages 103--274. Springer, Berlin, 2005.

\bibitem{funaki2001large}
T.~Funaki and T.~Nishikawa.
\newblock Large deviations for the {G}inzburg--{L}andau $\nabla \phi$ interface
  model.
\newblock {\em Probab. Theory and Related Fields}, 120(4):535--568, 2001.

\bibitem{FS}
T.~Funaki and H.~Spohn.
\newblock Motion by mean curvature from the {G}inzburg-{L}andau {$\nabla \phi$}
  interface model.
\newblock {\em Comm. Math. Phys.}, 185(1):1--36, 1997.

\bibitem{GOS}
G.~Giacomin, S.~Olla, and H.~Spohn.
\newblock Equilibrium fluctuations for {$\nabla\phi$} interface model.
\newblock {\em Ann. Probab.}, 29(3):1138--1172, 2001.

\bibitem{GNO}
A.~Gloria, S.~Neukamm, and F.~Otto.
\newblock Quantification of ergodicity in stochastic homogenization: optimal
  bounds via spectral gap on {G}lauber dynamics.
\newblock {\em Invent. Math.}, 199(2):455--515, 2015.

\bibitem{helffer1994correlation}
B.~Helffer and J.~Sj{\"o}strand.
\newblock On the correlation for {K}ac-like models in the convex case.
\newblock {\em J. Stat. Phys.}, 74:349--409, 1994.

\bibitem{hilger2016scaling}
S.~Hilger.
\newblock Scaling limit and convergence of smoothed covariance for gradient
  models with non-convex potential.
\newblock {\em arXiv preprint arXiv:1603.04703}, 2016.

\bibitem{hilger2020decay}
S.~Hilger.
\newblock Decay of covariance for gradient models with non-convex potential.
\newblock {\em arXiv preprint arXiv:2007.10869}, 2020.

\bibitem{hilger2020scaling}
S.~Hilger.
\newblock Scaling limit and strict convexity of free energy for gradient models
  with non-convex potential.
\newblock {\em arXiv preprint arXiv:2005.12973}, 2020.

\bibitem{ioffe20022d}
D.~Ioffe, S.~Shlosman, and Y.~Velenik.
\newblock 2{D} models of statistical physics with continuous symmetry: the case
  of singular interactions.
\newblock {\em Comm. Math. Phys.}, 226(2):433--454, 2002.

\bibitem{kharash2017fr}
V.~Kharash and R.~Peled.
\newblock The {F}r\"{o}hlich-{S}pencer {P}roof of the
  {B}erezinskii-{K}osterlitz-{T}houless {T}ransition.
\newblock {\em arXiv preprint arXiv:1711.04720}, 2017.

\bibitem{AK81}
U.~Krengel and M.~Akcoglu.
\newblock Ergodic theorems for superadditive processes.
\newblock {\em Journal für die reine und angewandte Mathematik}, 323:53--67,
  1981.

\bibitem{lammers2022dichotomy}
P.~Lammers.
\newblock A dichotomy theory for height functions.
\newblock {\em arXiv preprint arXiv:2211.14365}, 2022.

\bibitem{lammers2022height}
P.~Lammers.
\newblock Height function delocalisation on cubic planar graphs.
\newblock {\em Prob. Theory and Related Fields}, 182(1-2):531--550, 2022.

\bibitem{lammers2023bijecting}
P.~Lammers.
\newblock Bijecting the {BKT} transition.
\newblock {\em arXiv preprint arXiv:2301.06905}, 2023.

\bibitem{leindler1972certain}
L.~Leindler.
\newblock On a certain converse of {H}{\"o}lder’s inequality.
\newblock In {\em Linear Operators and Approximation/Lineare Operatoren und
  Approximation}, pages 182--184. Springer, 1972.

\bibitem{magazinov2020concentration}
A.~Magazinov and R.~Peled.
\newblock {Concentration inequalities for log-concave distributions with
  applications to random surface fluctuations}.
\newblock {\em Ann. Probab.}, 50(2):735 -- 770, 2022.

\bibitem{mathieu2008quenched}
P.~Mathieu.
\newblock Quenched invariance principles for random walks with random
  conductances.
\newblock {\em J. Stat. Phys.}, 130(5):1025--1046, 2008.

\bibitem{MP}
P.~Mathieu and A.~Piatnitski.
\newblock Quenched invariance principles for random walks on percolation
  clusters.
\newblock {\em Proc. R. Soc. Lond. Ser. A Math. Phys. Eng. Sci.},
  463(2085):2287--2307, 2007.

\bibitem{mathieu2004isoperimetry}
P.~Mathieu and E.~Remy.
\newblock Isoperimetry and heat kernel decay on percolation clusters.
\newblock {\em Ann. Probab.}, 32(1A):100--128, 2004.

\bibitem{Mil}
J.~Miller.
\newblock Fluctuations for the {G}inzburg-{L}andau {$\nabla \phi$} interface
  model on a bounded domain.
\newblock {\em Comm. Math. Phys.}, 308(3):591--639, 2011.

\bibitem{MilosPeled2015}
P.~Mi{\l}o{\'s} and R.~Peled.
\newblock Delocalization of two-dimensional random surfaces with hard-core
  constraints.
\newblock {\em Comm. Math. Phys.}, 340(1):1--46, 2015.

\bibitem{MO16}
J.-C. Mourrat and F.~Otto.
\newblock Anchored {N}ash inequalities and heat kernel bounds for static and
  dynamic degenerate environments.
\newblock {\em J. Funct. Anal.}, 270(1):201--228, 2016.

\bibitem{NS}
A.~Naddaf and T.~Spencer.
\newblock On homogenization and scaling limit of some gradient perturbations of
  a massless free field.
\newblock {\em Comm. Math. Phys.}, 183(1):55--84, 1997.

\bibitem{nash1958continuity}
J.~Nash.
\newblock Continuity of solutions of parabolic and elliptic equations.
\newblock {\em Amer. J. Math.}, 80:931--954, 1958.

\bibitem{Nirenberg59}
L.~Nirenberg.
\newblock On elliptic partial differential equations.
\newblock {\em Annali della Scuola Normale Superiore di Pisa - Scienze Fisiche
  e Matematiche}, Ser. 3, 13(2):115--162, 1959.

\bibitem{nishikawa2002hydrodynamic}
T.~Nishikawa.
\newblock Hydrodynamic limit for the {G}inzburg-{L}andau $\nabla \phi$
  interface model with a conservation law.
\newblock {\em J. Math. Sci. Univ. Tokyo}, 9(3):481--519, 2002.

\bibitem{nishikawa2003hydrodynamic}
T.~Nishikawa.
\newblock Hydrodynamic limit for the {G}inzburg-{L}andau {$\nabla\phi$}
  interface model with boundary conditions.
\newblock {\em Probab. Theory Related Fields}, 127(2):205--227, 2003.

\bibitem{prekopa1971logarithmic}
A.~Pr{\'e}kopa.
\newblock Logarithmic concave measures with applications to stochastic
  programming.
\newblock {\em Acta scientiarum mathematicarum}, 32:301--316, 1971.

\bibitem{prekopa1973logarithmic}
A.~Pr{\'e}kopa.
\newblock On logarithmic concave measures and functions.
\newblock {\em Acta Scientiarum Mathematicarum}, 34:335--343, 1973.

\bibitem{procaccia2016quenched}
E.~B. Procaccia, R.~Rosenthal, and A.~Sapozhnikov.
\newblock Quenched invariance principle for simple random walk on clusters in
  correlated percolation models.
\newblock {\em Probab. Theory Related Fields}, 166(3):619--657, 2016.

\bibitem{rassoul2005almost}
F.~Rassoul-Agha and T.~Sepp{\"a}l{\"a}inen.
\newblock An almost sure invariance principle for random walks in a space-time
  random environment.
\newblock {\em Probab. Theory Related Fields}, 133(3):299--314, 2005.

\bibitem{redig2013random}
F.~Redig and F.~V{\"o}llering.
\newblock Random walks in dynamic random environments: a transference
  principle.
\newblock {\em Ann. Probab.}, 41(5):3157--3180, 2013.

\bibitem{SS}
V.~Sidoravicius and A.-S. Sznitman.
\newblock Quenched invariance principles for walks on clusters of percolation
  or among random conductances.
\newblock {\em Probab. Theory Related Fields}, 129(2):219--244, 2004.

\bibitem{taylor2006measure}
M.~E. Taylor.
\newblock {\em Measure theory and integration}.
\newblock American Mathematical Soc., 2006.

\bibitem{van2023elementary}
D.~van Engelenburg and M.~Lis.
\newblock An elementary proof of phase transition in the planar {XY} model.
\newblock {\em Comm. Math. Phys.}, 399(1):85--104, 2023.

\bibitem{van2023duality}
D.~van Engelenburg and M.~Lis.
\newblock On the duality between height functions and continuous spin models.
\newblock {\em arXiv preprint arXiv:2303.08596}, 2023.

\bibitem{V06}
Y.~Velenik.
\newblock Localization and delocalization of random interfaces.
\newblock {\em Probab. Surv.}, 3:112--169, 2006.

\bibitem{wirth2019}
M.~Wirth.
\newblock Maximum of the integer-valued {G}aussian free field.
\newblock {\em arXiv preprint arXiv:1907.08868}, 2019.

\bibitem{wu2022local}
W.~Wu.
\newblock Local central limit theorem for gradient field models.
\newblock {\em arXiv preprint arXiv:2202.13578}, 2022.

\bibitem{wu2018subsequential}
W.~Wu and O.~Zeitouni.
\newblock Subsequential tightness of the maximum of two dimensional
  {G}inzburg-{L}andau fields.
\newblock {\em Electron. Comm. Probab}, 24:22, 2018.

\bibitem{ye2019models}
Z.~Ye.
\newblock Models of gradient type with sub-quadratic actions.
\newblock {\em J. Math. Phys.}, 60(7):073304, 2019.

\end{thebibliography}
}

\end{document}